\documentclass[11pt]{article}
\RequirePackage[OT1]{fontenc}
\RequirePackage{amsthm,amsmath,natbib}
\RequirePackage[colorlinks]{hyperref}
\usepackage{amsfonts}
\usepackage{amssymb}
\usepackage{amstext}
\usepackage{parskip}
\usepackage{fullpage}
\bibpunct{(}{)}{;}{a}{,}{,}

\usepackage{times}
\usepackage{graphicx}
\usepackage{epsfig}

\usepackage{amsthm,amsmath,amssymb}
\usepackage{graphicx}
\usepackage{epsfig}

\newenvironment{subeqnarray}
{\begin{subequations}\begin{eqnarray}}%
{\end{eqnarray}\end{subequations}\hskip-4.0pt}

\def\fatnorm#1{|\kern-.2ex|\kern-.2ex| #1 |\kern-.2ex|\kern-.2ex|}
\newcommand{\twonorm}[1]{\left\lVert#1\right\rVert_2}

\newcommand{\shtwonorm}[1]{\lVert#1\rVert_2}
\newcommand{\fnorm}[1]{\left\lVert#1\right\rVert_{F}}
\newcommand{\norm}[1]{\left\lVert#1\right\rVert}
\newcommand{\abs}[1]{\left\lvert#1\right\rvert}

\newcommand{\Sp}{\mathbb{S}}

\newcommand{\V}{\mathcal{V}}

\newcommand{\N}{{\mathcal N}}

\newcommand{\cov}{\textsf{Cov}}
\newcommand{\SNR}{\ensuremath\textsf{S/N}}
\newcommand{\signal}{\ensuremath\textsf{S}}
\newcommand{\noise}{\ensuremath\textsf{N}}

\newcommand{\spin}{\textsf{span}}

\newcommand{\half}{\ensuremath{\frac{1}{2}}}
\newcommand{\inv}[1]{\frac{1}{#1}}
\newcommand{\ip}[1]{\;\langle{\,#1\,}\rangle\;}
\newcommand{\size}[1]{\ensuremath{\left|#1\right|}}
\newcommand{\onenorm}[1]{\ensuremath{\left\|#1\right\|_1}}

\newcommand{\expct}[1]{\ensuremath{\mathbb E}\left(#1\right)}
\newcommand{\silent}[1]{}
\newcommand{\mvec}[1]{\rm{vec}\left\{\,#1\,\right\}}

\newcommand{\ve}{\varepsilon}

\def\qed{\hskip1pt $\;\;\scriptstyle\Box$}
\def\Ber{\mathop{\text{Bernoulli}\kern.2ex}}
\def\corr{\mathop{\text{corr}\kern.2ex}}
\def\prec{\mathop{\text{precision}\kern.2ex}}
\def\recall{\mathop{\text{recall}\kern.2ex}}
\def\cov{\mathop{\text{Cov}\kern.2ex}}
\def\mnorm{\mathcal{N}_{f,m}\kern.2ex}
\def\var{\mathop{\text{Var}\kern.2ex}}
\def\ess{\mathop{\text{ess}\kern.2ex}}
\def\dom{\mathop{\text{dom}\kern.2ex}}
\def\lin{\mathop{\text{lin}\kern.2ex}}

\newcommand{\func}[1]{\ensuremath{\mathrm{#1}}}
\newcommand{\diag}{\func{diag}}

\newcommand{\RE}{\textnormal{\textsf{RE}}}

\let\hat\widehat
\let\tilde\widetilde

\newcommand{\tr}{{\rm tr}}

\def\E{{\mathbb E}}

\newcommand{\prob}[1]{\ensuremath{\mathbb P}\left(#1\right)}

\newcommand{\beq}{\begin{equation}}
\newcommand{\eeq}{\end{equation}}
\newcommand{\ben}{\begin{eqnarray}}
\newcommand{\een}{\end{eqnarray}}
\newcommand{\bnum}{\begin{enumerate}}
\newcommand{\enum}{\end{enumerate}}
\newcommand{\bit}{\begin{itemize}}
\newcommand{\eit}{\end{itemize}}
\newcommand{\bens}{\begin{eqnarray*}}
\newcommand{\eens}{\end{eqnarray*}}

\newcommand{\R}{{\bf R}}

\newcommand{\Sc}{\ensuremath{S^c}}

\newcommand{\ora}{{\rm oracle}}
\newcommand{\SMR}{\ensuremath\textsf{S/M}}
\newcommand{\bignoise}{\ensuremath\textsf{M}}

\newcommand{\PaulBhalf}{\ensuremath{\hat\tau_B^{1/2}}}

\newcommand{\e}{\epsilon}
 
\newcommand{\vp}{\varpi}
\newcommand{\onef}{\textstyle \frac{1}{f}}
\newcommand{\onem}{\textstyle \frac{1}{m}}

\newcommand{\Ball}{{B}}
\newcommand{\B}{\mathcal{B}}
\newcommand{\A}{\mathcal{A}}
\newcommand{\U}{\Upsilon}
\newcommand{\up}{\upsilon}

\newcommand{\ind}{\mathbb{I}}

\newtheorem{theorem}{Theorem}
\newtheorem{assumption}{Assumption}[section]

\newcommand{\Cone}{{\rm Cone}}
\newcommand{\cone}{{\rm Cone}}

\def\supp{\mathop{\text{\rm supp}\kern.2ex}}
\def\conv{\mathop{\text{\rm conv}\kern.2ex}}
\def\span{\mathop{\text{\rm span}\kern.2ex}}
\def\argmin{\mathop{\text{arg\,min}\kern.2ex}}
\def\absconv{\mathop{\text{\rm absconv}\kern.2ex}}
\def\half{\frac{1}{2}}
\let\hat\widehat
\def\W{\Cone}

\newtheorem{lemma}[theorem]{Lemma}

\newtheorem{definition}[assumption]{Definition}

\newtheorem{remark}[assumption]{Remark}
\newtheorem{corollary}[theorem]{Corollary}

\def\qed{\hskip1pt $\;\;\scriptstyle\Box$}

\newenvironment{proofof}[1]{\hspace*{20pt}{\it Proof}{ of #1}.\hskip10pt}{\qed\vskip5pt}
\newenvironment{proofof2}{\hskip10pt}{\qed\vskip5pt}

\def\math{${}^\star$}
\def\stat{${}^\dag$}

\begin{document}
\title{Errors-in-variables models with dependent measurements
\footnote{
Mark Rudelson is partially supported by NSF grant DMS 1161372 and
USAF Grant FA9550-14-1-0009. Shuheng Zhou was supported in part by NSF under Grant
  DMS-1316731 and Elizabeth Caroline Crosby Funding from the Advance
  Program at the University of Michigan.
This manuscript was submitted for peer review in August 1, 2015; minor
typos are being corrected in this version.} 
}

\author{Mark Rudelson\math\; and\; Shuheng Zhou\stat\;\; \\
{\math}Department of Mathematics,\\
{\stat}Department of Statistics,\\
University of Michigan, Ann Arbor, MI 48109 \\[20pt]
Department of Statistics, Tech Report 538,  July 31, 2015
}

\maketitle

\begin{abstract}
Suppose that we observe $y \in \R^f$ and $X \in \R^{f \times m}$ in the following errors-in-variables model:
\begin{eqnarray*}
y & =  & X_0 \beta^* + \e \\
X & = & X_0 + W
\end{eqnarray*}
where $X_0$ is a $f \times m$ design matrix with independent
subgaussian row vectors, $\e \in \R^f$ is a noise vector and 
$W$ is a mean zero $f \times m$ random noise matrix with independent
subgaussian column vectors, independent of $X_0$ and $\e$.
This model is significantly different from those analyzed in the literature
in the sense that we allow the measurement error for each covariate 
to be a dependent vector across its $f$ observations. Such error structures appear in the 
science literature when modeling the trial-to-trial fluctuations in response strength shared
across a set of neurons.

Under sparsity and restrictive eigenvalue type of conditions, we show that one is able to 
recover a sparse vector $\beta^* \in \R^m$ from the model
given a single observation matrix $X$ and the response vector $y$.
We establish consistency in estimating $\beta^*$ and obtain the rates of convergence 
in the $\ell_q$ norm, where $q = 1, 2$ for the Lasso-type estimator,
and for $q \in [1, 2]$ for a Dantzig-type conic programming estimator.
We show error bounds which approach that of the regular Lasso and the Dantzig selector 
in case the errors in $W$ are tending to 0.
\end{abstract}

\section{Introduction}
The matrix variate normal model has a long history in psychology and social sciences, 
and is becoming increasingly popular in biology and genomics, neuroscience, econometric theory, image and signal processing, wireless
communication, and machine learning
in recent years, see for
example~\cite{Dawid81,GV92,Dut99,WJS08,BCW08,Yu09,Efr09,AT10,KLLZ13}, and the references therein.
\silent{
We model the data matrix $X$ as a single random sample coming from a
matrix variate normal distribution with a separable covariance  $A
\otimes B$, which we denote by $X_{f \times m} \sim \N_{f,m}(M, A_{m \times m} \otimes B_{f \times f})$, where $\otimes$ is the Kronecker product. This is equivalent to say  that $\mvec{X}$ follows a multivariate normal distribution with mean $\mvec{M}$ and covariance $A \otimes B$. Here $\mvec{X}$ is formed by stacking the columns of $X$ into a vector in $\R^{mf}$.}
We call the random matrix $X$ which contains $f$ rows and $m$
columns a single data matrix, or one instance from the matrix variate normal distribution. We say that an $f \times m$ random matrix $X$ follows a 
matrix normal distribution with a separable covariance matrix
$\Sigma_X = A \otimes B$, which we write 
$X_{f \times m} \sim \N_{f,m}(M, A_{m \times m} \otimes B_{f \times f}).$
This is equivalent to say $\mvec{X}$ follows a multivariate normal distribution with mean
$\mvec{M}$ and covariance $\Sigma_X = A \otimes B$. 
Here, $\mvec{X}$ is formed
by stacking the columns of $X$ into a vector in $\R^{mf}$.
Intuitively, $A$ describes the covariance between columns of $X$
while $B$  describes the covariance between rows of $X$.
 See~\cite{Dawid81,GV92} for more characterization and examples.

In this paper, we introduce the related Kronecker Sum models to encode
the covariance structure of a matrix variate distribution. The
proposed models and methods incorporate ideas from recent advances in
graphical models, high-dimensional regression model with observation
errors, and matrix decomposition. 
Let $A_{m \times m}, B_{f \times f}$ be symmetric positive definite 
covariance matrices. Denote the Kronecker sum of $A =(a_{ij})$ 
and $B = (b_{ij})$ by
\bens
\Sigma & = & A \oplus B := A \otimes I_f + I_m \otimes B \\
& = &
\left[
\begin{array}{cccc} 
a_{11}I_f + B & a_{12} I_f & \ldots & a_{1m} I_f\\
a_{21} I_f & a_{22}I_f +  B & \ldots & a_{2m} I_f\\
\ldots &  & & \\
a_{m1} I_f & a_{m2} I_f & \ldots & a_{mm}I_f + B
\end{array} \right]_{(m f) \times (m f)}
\eens
where $I_f$ is an $f \times f$ identity matrix.
This covariance model arises naturally from the context
of errors-in-variables regression model defined as follows.
Suppose that we observe $y \in \R^f$ and $X \in \R^{f \times m}$ in the following model:
\begin{subeqnarray}
\label{eq::oby}
y & =  & X_0 \beta^* + \e \\
\label{eq::obX}
X & = & X_0 + W
\end{subeqnarray}
where $X_0$ is a $f \times m$ design matrix with independent row vectors, $\e \in \R^f$ is a noise vector and 
$W$ is a mean zero $f \times m$ random noise matrix, independent of $X_0$ and $\e$, with independent column
vectors $\omega^1, \ldots, \omega^m$.
In particular, we are interested in the additive model of $X = X_0 + W$ such that
\ben
\label{eq::dataplus}
\mvec{X} \sim \N(0, \Sigma) \; \; \text{ where } \; \; 
\Sigma = A \oplus B := A \otimes I_f + I_m \otimes B
\een
where  we use one covariance component $A \otimes
I_f$ to describe the covariance of matrix $X_0 \in \R^{f \times m}$, which is considered as the {\it signal} matrix, and the other component $I_m \otimes B$ to describe that of the {\it noise matrix } $W \in \R^{f \times m}$, where $\E \omega^j \otimes \omega^j = B$ for all $j$,
where $\omega^j$ denotes the $j^{th}$ column vector of $W$.
Our focus is on deriving the statistical properties of two estimators for estimating $\beta^*$ in
\eqref{eq::oby} and~\eqref{eq::obX} despite the presence of the
additive error $W$ in the observation matrix $X$.
We will show that  our theory and analysis works with a model much more general than
that in \eqref{eq::dataplus}, which we will define in
Section~\ref{sec::method}.

Before we go on to define our estimators, we now use an example to motiviate~\eqref{eq::dataplus} and its
subgaussian generalization in Definition~\ref{def::subgdata}.
Suppose that there are $f$ patients in a particular study, for which
 we use $X_0$ to model the "systolic blood pressure" and $W$ to model
the seasonal effects. In this case, $X$ models the fact that 
among the $f$ patients we measure, each patient has its own row vector
of observed set of blood pressures across time, and each column vector
in $W$ models the seasonal variation on top of the true signal at a
particular day/time. Thus we consider $X$ as measurement of $X_0$ with
$W$ being the observation error. 
That is, we model the seasonal effects on blood pressures across a
set of patients in a particular study with a vector of dependent
entries. Thus $W$ is a matrix which consists of repeated independent 
sampling of spatially dependent vectors, if we regard the individuals 
as having spatial coordinates, for example, through their geographic locations.
We will come back to discuss this example in Section~\ref{sec::discuss}.

\subsection{The model and the method} 
\label{sec::method}
We first need to define an independent isotropic vector with {\em subgaussian} marginals as in Definition~\ref{def:psi2-vector}.
\begin{definition}
\label{def:psi2-vector} 
Let $Y$ be a random vector in $\R^p$
\bnum
\item
$Y$ is called isotropic
if for every $y \in \R^p$, $\expct{\abs{\ip{Y, y}}^2} = \twonorm{y}^2$.
\item
$Y$ is $\psi_2$ with a constant $\alpha$ if for every $y \in \R^p$,
\beq
\norm{\ip{Y, y}}_{\psi_2} := \;
\inf \{t: \expct{\exp(\ip{Y,y}^2/t^2)} \leq 2 \}
\; \leq \; \alpha \twonorm{y}.
\eeq
\enum
The  $\psi_2$ condition on a scalar random variable $V$ is equivalent to
the subgaussian tail decay of $V$, which means
$\prob{|V| >t} \leq 2 \exp(-t^2/c^2), \; \; \text{for all} \; \; t>0.$
\end{definition}
Throughout this paper, we use $\psi_2$ vector, a vector with subgaussian marginals
and subgaussian vector interchangeably.
\begin{definition}
\label{def::subgdata}
Let $Z$ be an $f \times m$ random matrix with independent entries $Z_{ij}$ satisfying
$\E Z_{ij} = 0$, $1 = \E Z_{ij}^2 \le \norm{Z_{ij}}_{\psi_2} \leq K$. 
Let $Z_1, Z_2$ be independent copies of $Z$.
Let $X = X_0 + W$ such that
\bnum
\item
$X_0 = Z_1 A^{1/2}$ is the design matrix with independent subgaussian
row vectors, and 
\item
$W =B^{1/2} Z_2$ is a random noise matrix with independent subgaussian
column vectors.
\enum
\end{definition}
Assumption (A1) allows the covariance model in~\eqref{eq::dataplus}
and its subgaussian variant in Definition~\ref{def::subgdata} to be identifiable. 
\bnum
\item[(A1)]
We assume $\tr(A) = m$ is a known parameter, where $\tr(A)$ denotes
the trace of matrix $A$.
\enum
In the kronecker sum model, we could assume we know 
$\tr(B)$, in order not to assume knowing $\tr(A)$. 
Assuming one or the other is known is unavoidable as the covariance model is not identifiable otherwise. 
Moreover, by knowing $\tr(A)$, we can construct an estimator for
$\tr(B)$:
\ben
\label{eq::trBest}
\hat\tr(B) &= &
\onem \big(\fnorm{X}^2 -f \tr(A)\big)_{+} \; \; \;
\text{ and define } \; \; 
\hat\tau_B  := \onef \hat\tr(B) \ge 0 
\een
where $(a)_{+} = a \vee 0$. 
We first introduce the Lasso-type estimator, adapted from those as
considered in~\cite{LW12}.

Suppose that $\hat\tr(B)$ is an estimator for $\tr(B)/f$; for example, 
as constructed in~\eqref{eq::trBest}. Let 
\ben
\label{eq::hatGamma}
\hat\Gamma & = & 
\inv{f} X^T X - \inv{f} \hat\tr(B)  I_{m}\;\;
\; \text{ and } \;
\hat\gamma \; = \; \inv{f} X^T y.
\een
For a chosen penalization parameter $\lambda \geq 0$, and parameters
$b_0$ and $d$, we consider the following  regularized estimation with the $\ell_1$-norm penalty,
\begin{eqnarray}
\label{eq::origin} \; \; 
\hat \beta & = & \argmin_{\beta: \norm{\beta}_1 \le b_0 \sqrt{d}} \frac{1}{2} \beta^T \hat\Gamma \beta 
- \ip{\hat\gamma, \beta} + \lambda \|\beta\|_1, \; \;
\end{eqnarray}
which is a variation of the Lasso \cite{Tib96} or the Basis
Pursuit~\cite{Chen:Dono:Saun:1998} estimator.
Although in our analysis, we set $b_0 \ge \twonorm{\beta^*}$ and
$d = \size{\supp(\beta^*)}$ for simplicity. In practice, both $b_0$
and $d$ are understood to be parameters chosen to provide an upper bound on the $\ell_2$ norm and the sparsity of  the true $\beta^*$.

Recently,~\cite{BRT14} discussed the following conic programming
compensated matrix uncertainly (MU) selector
, which is a variant of the Dantzig selector~\cite{CT07,RT10,RT13}. Adapted to our setting, it is
defined as follows. Let $\lambda, \mu, \tau >0$, 
\begin{eqnarray}
\label{eq::Conic} \; \; 
\hat \beta & = & \arg\min\big\{\norm{\beta}_1 +\lambda t\; :\; (\beta, t) \in \U\big\} 
\text{ where } \; \\
\nonumber
\U &= & \left\{(\beta, t) \; : \; \beta \in \R^m, \norm{\hat\gamma - \hat\Gamma \beta}_{\infty}
\le \mu t + \tau, \twonorm{\beta} \le t\right\}
\end{eqnarray}
where $\hat\gamma$ and $\hat\Gamma$ are as defined in \eqref{eq::hatGamma} 
with $\mu \sim \sqrt{\frac{\log m}{f}}$, $\tau \sim \sqrt{\frac{\log m}{f}}$. We refer to this estimator as the 
Conic programming estimator from now on.

\subsection{Our contributions}
We provide a unified analysis of the rates of convergence for both the
Lasso-type estimator~\eqref{eq::origin} as well as the Conic Programming estimator \eqref{eq::Conic},
which is a Dantzig selector-type, although under slightly different conditions. 
We will show the rates of convergence in the $\ell_q$ norm for $q =1, 2$
for estimating a sparse vector $\beta^* \in \R^m$ in the
model~\eqref{eq::oby} and~\eqref{eq::obX} using the Lasso-type
estimator~\eqref{eq::origin} 
in Theorems~\ref{thm::lasso} and~\ref{thm::lassora},  
and the Conic Programming estimator~\eqref{eq::Conic} in
Theorems~\ref{thm::DS} and~\ref{thm::DSoracle} 
for $1\le q\le 2$.
For the Conic Programming estimator, we also show bounds on the predictive errors.
The bounds we derive in both Theorems~\ref{thm::lasso} and~\ref{thm::DS}  focus on cases
where the errors in $W$ are not too small in their magnitudes  in the
sense that $\tau_B := \tr(B)/f$ is bounded from below. For the extreme case when $\tau_B$
approaches $0$, one hopes to recover bounds close to those for the
regular Lasso or the Dantzig selector as the effect of the noise in matrix $W$ on the procedure 
becomes negligible. We show in Theorems~\ref{thm::lassora} 
and~\ref{thm::DSoracle} that this is indeed the case. These results are new to the best of our knowledge.

In Theorems~\ref{thm::lasso} to~\ref{thm::DSoracle}, 
we consider the regression model in \eqref{eq::oby} and \eqref{eq::obX} with subgaussian random design, 
where $X_0 = Z_1 A^{1/2}$ is a subgaussian random matrix with independent row
vectors, and $W =B^{1/2} Z_2$ is a $f \times m$ random noise matrix with
independent column vectors where $Z_1, Z_2$ are independent
subgaussian random matrices with independent entries
(cf. Definition~\ref{def::subgdata}). 
This model is significantly different from
those analyzed in the literature. For example, unlike the present work, the authors in~\cite{LW12} apply 
Theorem~\ref{thm::main} which states a general result on statistical
convergence properties of the estimator~\eqref{eq::origin} to cases
where $W$ is composed of independent subgaussian row vectors, 
when the row vectors of $X_0$ are either independent or follow a
Gaussian vector auto-regressive model. 
See also~\cite{RT10,RT13,CC13,BRT14} for the corresponding results on
the compensated MU selectors, variant on the Orthogonal Matching
Pursuit algorithm and the Conic Programming estimator
\eqref{eq::Conic}. 

The second key difference between our framework and
the existing work is that we assume that only one observation matrix $X$ with the single measurement
error matrix $W$ is available. Assuming (A1) allows us to estimate $\E W^T W$ as required
in the estimation procedure \eqref{eq::hatGamma} directly, given the
knowledge that $W$ is composed of independent column vectors.  
In contrast, existing work needs to assume that the covariance matrix $\Sigma_W := \onef \E W^T W$ of the independent row vectors
of $W$ or its functionals are either known a priori, or can be 
estimated from an dataset independent of $X$, or from replicated $X$
measuring the same $X_0$; see for example~\cite{RT10,RT13,BRT14,LW12, carr:rupp:2006}.
Such repeated measurements are not always available or are costly to obtain in practice~\cite{carr:rupp:2006}.

A noticeable exception is the work of~\cite{CC13}, which deals with
the scenario when the noise covariance is not assumed to be known. 
We now elaborate on their result, which is a variant of the orthogonal
matching pursuit (OMP) algorithm~\cite{Tropp:04,TG07}.
Their support recovery result, that is, recovering the support set of
$\beta^*$, applies only to the case when both signal matrix and the
measurement error matrix have isotropic subgaussian row vectors; that
is, they assume independence among both rows and columns in $X$ 
($X_0$ and $W$); moreover, their algorithm requires the knowledge
of the sparsity parameter $d$, which is the number of non-zero entries
in $\beta^*$, as well as a $\beta_{\min}$ condition: $\min_{j \in
  \supp{\beta^*}}  \abs{\beta^*_j} = \Omega\left(\sqrt{\frac{\log
      m}{f}}(\twonorm{\beta^*}+1)\right)$. 
They recover essentially the same $\ell_2$-error bounds
as in~\cite{LW12} and the current work when the covariance $\Sigma_W$
is known.

In summary, oblivion in $\Sigma_W$ and a general dependency
condition in the data matrix $X$ are not simultaneously allowed in
existing work. In contrast, while we assume that $X_0$ is composed of
independent subgaussian row vectors, we allow rows of $W$ to be
dependent, which brings dependency to the row vectors of the
observation matrix $X$.
In the current paper, we focus on the proof-of-the-concept on using
the kronecker sum covariance and additive model to model two way
dependency in data matrix $X$, and derive bounds in statistical
convergence for~\eqref{eq::origin} and~\eqref{eq::Conic}.
In some sense, we are considering a parsimonious model for fitting
observation data with two-way dependencies; that is, we use the signal
matrix to encode column-wise dependency among covariates in $X$, 
and error matrix $W$ to explain its row-wise dependency.
When replicates of $X$ or $W$ are available, we are able to study more
sophisticated models and inference problems to be described in Section~\ref{sec::discuss}.

\subsection{Discussion}
\label{sec::discuss}
The key modeling question is: would each row vector in $W$ for a
particular patient across all time points be a correlated normal or
subgaussian vector as well? It is our conjecture that combining the
newly developed techniques, namely, the concentration of measure
inequalities we have derived in the current framework with techniques
from existing work, we can handle the case when $W$ follows a matrix
normal distribution with a separable covariance matrix  $\Sigma_W = C
\otimes B$, where $C$ is an $m \times m$ positive semi-definite
covariance matrix.  Moreover, for this type of "seasonal effects" as 
the measurement errors, the time varying covariance model 
would make more sense to model $W$, which we elaborate in the second example.

As a second example, in neuroscience applications, population coding refers to the information
contained in the combined activity of multiple neurons~\cite{KassVB05}.
The relationship between population encoding and correlations is complicated and is an area of 
active investigation, see for example~\cite{RC14a,CK11}
It becomes more often that repeated measurements (trials) simultaneously recorded
across a set of neurons and over an ensemble of stimuli are
available. In this context, one can imagine using a random matrix  $X_0 \sim
N_{f,m}(\mu, A \otimes B)$  which follows a matrix-variate normal
distribution, or its subgaussian correspondent, 
to model the ensemble of mean response variables, e.g.,
the membrane potential, corresponding to the cross-trial average 
over a set of experiments. Here we use $A$ to model the task correlations and $B$ 
to model the baseline correlation structure among all pairs of neurons at the {\it signal } level.
It has been observed that the onset of stimulus and task events not only
change the cross-trial mean response in $\mu$, but also alter the structure
and correlation of the {\it noise } for a set of neurons,
which correspond to the trial-to-trial fluctuations of the neuron responses.
We use $W$ to model such task-specific trial-to-trial fluctuations of
a set of neurons  recorded over the time-course of a variety of tasks.
Models as in~\eqref{eq::oby} and~\eqref{eq::obX} are useful in predicting the response of set of 
neurons based on the current and past mean responses of all neurons.
Moreover, we could incorporate non-i.i.d. non-Gaussian $W = [w_1,
\ldots, w_m]$ where $w_t = B^{1/2}(t) z(t)$, where $z(1), \ldots, z(m)$ 
are independent isotropic subgaussian random vectors and $B(t) \succ 0$ for all $t$,
to model the time-varying correlated noise as observed in the
trial-to-trial fluctuations. 
It is possible to combine the techniques developed in the present
paper with those in ~\cite{ZLW08,Zhou14a} to develop estimators for
$A$, $B$ and the time varying $B(t)$
which is itself an interesting topic, however, beyond the scope of the current work.
 
We leave the investigation of this more general modeling framework and
relevant statistical questions to future work. 
We refer to ~\cite{carr:rupp:2006}
for an excellent survey of the classical as well as modern
developments in measurement error models. 
In future work, we will also extend the estimation methods to the settings where the covariates are
measured with multiplicative errors which are shown to be reducible to
the additive error problem as studied in the present work;
see~\cite{RT13,LW12}. Moreover, we are interested in applying the
analysis and concentration of measure results developed in the current
paper and in our ongoing work to the more general contexts and settings where measurement
error models are introduced and investigated; see for
example~\cite{DLR77,CGG85,Stef:1985,HWang86,Full:1987,Stef:1990,CW91,CGL93,Cook:Stef:1994,Stef:Cook:1995,ICF99,LHC99,Str03,XY07,HM07,LL09,ML10,AT10,SSB14,SFT14,SFT14b} and the references therein.

\section{Assumptions and preliminary results}
We will now define some parameters related to the restricted and
sparse eigenvalue conditions that are needed to state our main
results. We also state a preliminary result in
Lemma~\ref{lemma::REcomp} regarding the relationships between the
two conditions in Definitions~\ref{def:memory} and~\ref{def::lowRE}.
\begin{definition}
\label{def:memory}
\textnormal{\bf (Restricted eigenvalue condition $\RE(s_0, k_0, A)$).}
Let $1 \leq s_0 \leq p$, and let $k_0$ be a positive number.
We say that a $p \times q$ matrix $A$ satisfies $\RE(s_0, k_0, A)$
 condition with parameter $K(s_0, k_0, A)$ if for any $\upsilon
 \not=0$,
\beq
\inv{K(s_0, k_0, A)} := 
\min_{\stackrel{J \subseteq \{1, \ldots, p\},}{|J| \leq s_0}}
\min_{\norm{\upsilon_{J^c}}_1 \leq k_0 \norm{\upsilon_{J}}_1}
\; \;  \frac{\norm{A \upsilon}_2}{\norm{\upsilon_{J}}_2} > 0.
\eeq
\end{definition}
It is clear that when $s_0$ and $k_0$ become smaller,
this condition is easier to satisfy.
We also consider the following variation of the baseline $\RE$ condition.
\begin{definition}{\textnormal{(Lower-$\RE$ condition)~\cite{LW12}}}
\label{def::lowRE}
The matrix $\Gamma$ satisfies a Lower-$\RE$ condition with curvature
$\alpha >0$ and tolerance $\tau > 0$ if 
\bens
\theta^T \Gamma \theta \ge 
\alpha \twonorm{\theta}^2 - \tau \onenorm{\theta}^2 \; \;  \forall \theta \in \R^m.
\eens
\end{definition}
As $\alpha$ becomes smaller, or as $\tau$ becomes larger, the
Lower-$\RE$ condition is easier to be satisfied.
\begin{lemma}
\label{lemma::REcomp}
Suppose that the Lower-$\RE$ condition holds for $\Gamma := A^T A$
with $\alpha, \tau > 0$ such that $\tau (1 + k_0)^2 s_0 \le \alpha/2$.
Then the $\RE(s_0, k_0, A)$ condition holds for $A$ with 
\bens
\inv{K(s_0, k_0, A)} \ge \sqrt{\frac{\alpha}{2}} >0.
\eens
Assume that $\RE((k_0+1)^2, k_0, A)$ holds.
Then the Lower-$\RE$ condition holds for  $\Gamma = A^T A$ with 
\bens
\alpha =\inv{(k_0 + 1)K^2(s_0, k_0, A)} > 0
\eens
where $s_0 = (k_0+1)^2$, and $\tau > 0$ which satisfies
\ben
\label{eq::tauchoice}
\lambda_{\min}(\Gamma) \ge \alpha - \tau s_0/4.
\een
The condition above holds for any 
$\tau \ge \frac{4}{(k_0 + 1)^3K^2(s_0, k_0, A)} - \frac{4 \lambda_{\min}(\Gamma)}{(k_0+1)^2}$.
\end{lemma}
The first part of Lemma \ref{lemma::REcomp} means that, if 
$k_0$ is fixed, then smaller values of $\tau$ guarantee $\RE(s_0, k_0,
A)$ holds with larger $s_0$, that is, a stronger $\RE$ condition. 
The second part of the Lemma implies that a weak $\RE$ condition
implies that the Lower-$\RE$ (LRE) holds with a large $\tau$. 
On the other hand, if one assumes $\RE((k_0+1)^2,k_0,A)$ holds with a
large value of  $k_0$ (in other words, a strong $\RE$ condition), this
would imply LRE with a small $\tau$. In short, the two
conditions are similar but require tweaking the parameters. 
Weaker $\RE$ condition implies LRE condition holds with a
larger $\tau$, and Lower-$\RE$ condition with a smaller $\tau$, that
is, stronger LRE implies stronger $\RE$.
We prove Lemma \ref{lemma::REcomp} in Section~\ref{sec::proofoflemmaREcomp}.
\begin{definition}{\textnormal{(Upper-$\RE$ condition)~\cite{LW12}}}
\label{def::upRE}
The matrix $\Gamma$ satisfies an upper-$\RE$ condition with curvature
$\bar\alpha >0$ and tolerance $\tau > 0$ if 
\bens
\theta^T \Gamma \theta \le \bar\alpha \twonorm{\theta}^2 + \tau
\onenorm{\theta}^2 \; \;  \forall \theta \in \R^m.
\eens
\end{definition}

\begin{definition}
\label{def::sparse-eigen}
Define the largest and smallest
$d$-sparse eigenvalue of a $p \times q$ matrix $A$ to be
\ben
\label{eq::eigen-Sigma}
\rho_{\max}(d, A) & := &
\max_{t \not= 0; d-\text{sparse}} \; \;\shtwonorm{A
  t}^2/\twonorm{t}^2, \text{ where } \; \; d< p, \\
\label{eq::eigen-Sigma-min}
 \text{ and } \; \;
\rho_{\min}(d, A) & := &
\min_{t \not= 0; d-\text{sparse}} \; \;\shtwonorm{A t}^2/\twonorm{t}^2.
\een
\end{definition}
The rest of the paper is organized as follows. 
In Section~\ref{sec::tworesults}, we present two main results
Theorems~\ref{thm::lasso} and~\ref{thm::DS}. 
We state results which improve upon Theorems~\ref{thm::lasso} and Theorem~\ref{thm::DS}
in Section~\ref{sec::smallW}, when the measurement errors in $W$ are {\it small} 
in their magnitudes in the sense of $\tr(B)$ being small.
In Section~\ref{sec::proofall}, we outline the proof of the main theorems.
In particular,  In Section~\ref{sec::proofall}, we outline the proof for 
Theorems~\ref{thm::lasso},~\ref{thm::DS},~\ref{thm::lassora},and~\ref{thm::DSoracle} 
in Section~\ref{sec::proofall},~\ref{sec::lassooutline},~\ref{sec::lassooracle}
and~\ref{sec::DSoracle} respectively.
In Section~\ref{sec::AD}, we show a deterministic result as well as its
application to the random matrix $\hat\Gamma - A$ for  $\hat\Gamma$ as in~\eqref{eq::hatGamma} with regards to the
upper and Lower $\RE$ conditions.
In section~\ref{sec::Best}, we show the concentration properties of
the gram matrices $XX^T$ and $X^TX$ after we correct them with the
corresponding {\it  population} error terms defined by $\tr(A) I_f$
and $\tr(B) I_m$ respectively. These results might be of independent
interests.  
The technical details of the proof are collected at the end of the
paper.
We prove Theorem~\ref{thm::lasso} in Section~\ref{sec::proofofthmlasso}.
We prove Theorem~\ref{thm::DS} in Section~\ref{sec::proofofDSthm}.
We prove Theorem~\ref{thm::lassora} and~\ref{thm::DSoracle} in Section~\ref{sec::classoproof}
and Section~\ref{sec::DSoraproof} respectively.
The paper concludes with a discussion of the results 
in Section~\ref{sec::conclude}.
Additional proofs and theoretical results are collected in the Appendix.

\noindent{\bf Notation.}
Let $e_1,\ldots, e_p$ be the canonical basis of $\R^p$.
For a set $J \subset \{1, \ldots, p\}$, denote
$E_J = \spin\{e_j: j \in J\}$.
For a matrix $A$, we use $\twonorm{A}$ to denote its operator norm.
For a set $V \subset \R^p$,
we let $\conv V$ denote the convex hull of $V$. For a finite set
$Y$, the cardinality is denoted by $|Y|$. Let $\Ball_1^p$, $\Ball_2^p$
and $S^{p-1}$ be the unit $\ell_1$ ball, the unit Euclidean ball and the
unit sphere respectively.
For a matrix $A = (a_{ij})_{1\le i,j\le m}$, let $\norm{A}_{\max} =
\max_{i,j} |a_{ij}|$ denote  the entry-wise max norm. 
Let $\norm{A}_{1} = \max_{j}\sum_{i=1}^m\abs{a_{ij}}$ 
denote the matrix $\ell_1$ norm.
The Frobenius norm is given by $\norm{A}^2_F = \sum_i\sum_j a_{ij}^2$. 
Let $|A|$ denote the determinant and ${\rm tr}(A)$ be the trace of $A$.
Let $\lambda_{\max}(A)$ and $\lambda_{\min}(A)$ be the
largest and smallest eigenvalues, and  $\kappa(A)$ be the condition
number for matrix $A$. 
The operator or $\ell_2$ norm $\twonorm{A}^2$ is given by
$\lambda_{\max}(AA^T)$.  

For a matrix $A$, denote by $r(A)$ the effective rank
$\tr(A)/\twonorm{A}$. Let ${\fnorm{A}^2 }/{\twonorm{A}^2}$ denote the
stable rank for matrix $A$.
We write $\diag(A)$ for a diagonal matrix with the same diagonal as
$A$.  For a symmetric matrix $A$, let $\Upsilon(A) =
\left(\upsilon_{ij}\right)$ where $\upsilon_{ij}   = \ind(a_{ij} \not=0)$, where $\mathbb{I}(\cdot)$ is the indicator function.
Let $I$ be the identity matrix. 
We let $C$ be a constant which may change from line to line.
For two numbers $a, b$, $a \wedge b := \min(a, b)$ and $a \vee b := \max(a, b)$. 
We write $a \asymp b$ if $ca \le b \le Ca$ for some positive absolute
constants $c,C$ which are independent of $n, f, m$ or sparsity
parameters. Let $(a)_+ := a \vee 0$.
We write $a =O(b)$ if $a \le Cb$ for some positive absolute
constants $C$ which are independent of $n, f, m$ or sparsity
parameters.
 These absolute constants $C, C_1, c, c_1, \ldots$ may change line by line.

\section{Main results}
\label{sec::tworesults}
In this section, 
we will state our main results in Theorems~\ref{thm::lasso}
and~\ref{thm::DS} where we consider the regression model in \eqref{eq::oby}
and \eqref{eq::obX} with random matrices $X_0, W \in \R^{f \times m}$
as defined in Definition~\ref{def::subgdata}.


For the Lasso-type estimator, we are interested in the case where the smallest eigenvalue
of the column-wise covariance matrix $A$ does not approach $0$ too quickly
and the effective rank of the row-wise covariance matrix $B$ is 
bounded from below (cf.~\eqref{eq::trBLasso}).
For the Conic Programming estimator, we impose a restricted eigenvalue condition as formulated
in~\cite{BRT09,RZ13} on $A$ and assume that the sparsity of $\beta^*$
is bounded by $o(\sqrt{f/\log m})$. These conditions will be relaxed
in Section~\ref{sec::smallW} where we allow $\tau_B$ to approach 0.

Before stating our main result for the Lasso-type estimator in
Theorem~\ref{thm::lasso}, we need to introduce some more notation and assumptions.
Let $a_{\max} = \max_{i} a_{ii}$ and $b_{\max} =\max_{i} b_{ii}$ be the maximum diagonal entries of $A$ and $B$ respectively.
In general, under (A1), one can think of $\lambda_{\min}(A) \le 1$ and 
for $s \ge 1$,
$$1 \le a_{\max} \le \rho_{\max}(s, A) \le \lambda_{\max}(A),$$
where $\lambda_{\mathrm{max}}(A)$ denotes the maximum eigenvalue of
$A$. 
\bnum
\item[(A2)]
The minimal eigenvalue $\lambda_{\min}(A)$
of the covariance matrix $A$ is bounded: $1 \ge \lambda_{\min}(A) > 0$.
\item[(A3)]
Moreover, we assume that the condition number $\kappa(A)$ is upper bounded
by $O\left(\sqrt{\frac{f}{\log m}}\right)$ and $\tau_B = O(\lambda_{\max}(A))$.
\enum
Throughout the rest of the paper, $s_0 \ge 1$ is understood to be the largest integer chosen such that the following inequality still holds:
\ben
 \label{eq::s0cond}
\sqrt{s_0} \vp(s_0) \le \frac{\lambda _{\min}(A)}{32 C}\sqrt{\frac{f }{\log m}}
\; \text{ where  }\; \vp(s_0) := \rho_{\max}(s_0, A)+\tau_B
\een
where we denote by $\tau_B = \tr(B)/f$ and $C$ is to be defined.
Denote by
\ben
\label{eq::defineM}
M_A = \frac{64 C \vp(s_0)}{\lambda_{\min}(A)} \ge 64 C.
\een
Throughout this paper, for the Lasso-type estimator, we will use the expression 
$$\tau := \frac{\alpha}{s_0}, \; \; \text{where} \; \; \alpha =
\lambda_{\min}(A)/2;$$
(A2) thus ensures that the Lower-$\RE$ condition as in
Definition~\ref{def::lowRE} is not vacuous. 
(A3) ensures that~\eqref{eq::s0cond} holds for some $s_0 \ge 1$.
\begin{theorem}{\textnormal{(\bf{Estimation for the Lasso-type estimator})}}
\label{thm::lasso}
Set $1 \leq f \leq m$.
Suppose $m$ is sufficiently large. Suppose (A1), (A2) and (A3) hold.
Consider the regression model in  \eqref{eq::oby} and \eqref{eq::obX}
with independent random matrices $X_0, W$ as in
Definition~\ref{def::subgdata},  and an error vector $\e \in \R^f$
independent of $X_0, W$,  with independent entries $\e_{j}$ satisfying
$\E \e_{j} = 0$ and  $\norm{\e_{j}}_{\psi_2} \leq M_{\e}$.
Let $C_0, c' > 0$ be  some absolute constants. Let $D_2 := 2(\twonorm{A} + \twonorm{B})$.
Suppose that  $\fnorm{B}^2/\twonorm{B}^2 \ge \log m$.
Suppose that $c' K^4 \le 1$ and
\ben
\label{eq::trBLasso}
\quad
r(B) := \frac{\tr(B)}{\twonorm{B}} & \ge & 16c' K^4 \frac{f}{\log m}
\log \frac{\V m \log m }{f}
\een
where $\V$ is a constant which depends on $\lambda_{\min}(A)$,
$\rho_{\max}(s_0, A)$ and $\tr(B)/f$. 

Let $b_0, \phi$ be numbers which satisfy
\ben
\label{eq::snrcond}
\frac{M^2_{\e}}{K^2 b_0^2}   \le \phi  \le 1.
\een
Assume that the sparsity of $\beta^*$ satisfies for some $0 < \phi \le
1$
\ben
\label{eq::dlasso}
&& d:= \abs{\supp(\beta^*)} \le 
\frac{c' \phi K^4}{128 M_A^2} \frac{f}{\log  m}< f/2.  
\een
Let $\hat\beta$ be an optimal solution to the Lasso-type estimator as in~\eqref{eq::origin} with 
\ben
\label{eq::psijune}
&& \lambda \ge 4 \psi \sqrt{\frac{\log m}{f}} \; \; \text{ where } \;\;
\psi  := C_0 D_2 K \left(K \twonorm{\beta^*}+ M_{\e}\right)
\een
Then for any $d$-sparse vectors $\beta^* \in \R^m$, such that
$\phi b_0^2 \le \twonorm{\beta^*}^2 \le b_0^2$, we have with probability  at least $1- 16/m^3$,
\bens
\twonorm{\hat{\beta} -\beta^*} \leq \frac{20}{\alpha}  \lambda \sqrt{d} \; \;
\text{ and } \; \norm{\hat{\beta} -\beta^*}_1 \leq \frac{80}{\alpha}
\lambda d.
\eens
\end{theorem}
We give an outline of the proof
of Theorem~\ref{thm::lasso} in Section~\ref{sec::lassooutline}.
We prove Theorem~\ref{thm::lasso} in Section~\ref{sec::proofofthmlasso}.

\noindent{\bf Discussions.}
Denote the Signal-to-noise ratio by 
\bens
\SNR := {K^2 \twonorm{\beta^*}^2}/{M^{2}_{\e}}
\; \text{ where } \; \signal := K^2 \twonorm{\beta^*}^2 \; \text{ and
} \; \; \noise := M^{2}_{\e}.
\eens
The two conditions on $b_0, \phi$ imply that $\noise \le \phi
\signal$. Notice that this could be restrictive if $\phi$ is small.

We will show in Section~\ref{sec::lassooutline} that condition
~\eqref{eq::snrcond} is not needed in order for the error
bounds in terms of the $\ell_p, p=1, 2$ norm of $\hat\beta - \beta^*$,
as shown in the Theorem~\ref{thm::lasso} statement  to hold.  It was
indeed introduced so as to simplify the expression for the condition on $d$ as shown in
\eqref{eq::dlasso}. There we provide a slightly more general condition on $d$ in
\eqref{eq::dlassoproof}, where~\eqref{eq::snrcond} is not required.
In summary, we prove that Theorem~\ref{thm::lasso} holds with $\noise = M_{\e}^2$ and $\signal
=\phi K^2 b_0^2$ in arbitrary orders, so long as condition
\eqref{eq::trBLasso} holds and
$$d =O\left(\inv{M_A^2} \frac{f}{\log m}\right).$$
For both cases, we require that 
$\lambda  \asymp (\twonorm{A} + \twonorm{B}) K \sqrt{\signal+\noise}
\sqrt{\frac{\log m}{f}}$ as expressed in \eqref{eq::psijune}.
That is, when either the noise level $M_{\e}$ or the signal strength
$K \norm{\beta^*}$ increases, we need to increase
$\lambda$ correspondingly; moreover, when $\noise$ dominates the
signal $K^2 \twonorm{\beta^*}^2$, we have for $d \asymp \inv{M_A^2} \frac{f}{\log m}$,
\bens
\twonorm{\hat{\beta} -\beta^*} /\twonorm{\beta^*} 
\leq \frac{20}{\alpha}  D_2 K^2 \sqrt{\frac{\noise}{\signal}}
\inv{M_A} \asymp D_2 K^2 \sqrt{\frac{\noise}{\signal}}  \inv{\vp(s_0)}
\eens
which eventually becomes a vacuous bound when $\noise \gg \signal$.
We will present an improved bound in Theorem~\ref{thm::lassora}.
We further elaborate on the relationships among the noise, the
measurement error and the signal strength in Section~\ref{sec::discusslassocoro}.
\begin{theorem}
\label{thm::DS}
Suppose (A1) holds.
Set $0< \delta < 1$. Suppose that $f < m \ll \exp(f)$ and $1\le  d_0 < f$.
Let $\lambda >0$ be the same parameter as in~\eqref{eq::Conic}.
Assume that $\RE(2d_0,  3(1+ \lambda), A^{1/2})$ holds. 
Suppose that $\fnorm{B}^2/\twonorm{B}^2 \ge \log m$.
Suppose that the sparsity of $\beta^*$ is bounded by 
\ben
\label{eq::sqrt-sparsity}
d_0 := \abs{\supp(\beta^*)} \le c_0 \sqrt{f/\log m}
\een
for some constant $c_0>0$;
Suppose $k_0 := 1+\lambda$
\ben
\label{eq::samplebound}
f & \geq & \frac{2000 d K^4}{\delta^2} \log \left(\frac{60 e m}{d
    \delta}\right) \; \text{where } \\
\label{eq::sparse-dim-Ahalf}
d & = & 2d_0 + 2d_0 a_{\max} \frac{16 K^2(2d_0, 3k_0, A^{1/2}) (3k_0)^2
  (3k_0 + 1)}{\delta^2}.
\een
Consider the regression model in  \eqref{eq::oby} and \eqref{eq::obX}
with $X_0$, $W$ as in Definition~\ref{def::subgdata}
 and an error vector $\e \in \R^f$, independent of $X_0, W$,  with independent entries $\e_{j}$ satisfying
$\E \e_{j} = 0$ and  $\norm{\e_{j}}_{\psi_2} \leq M_{\e}$.
Let $\hat\beta$ be an optimal solution to the Conic Programming estimator
as  in \eqref{eq::Conic} with input $(\hat\gamma, \hat\Gamma)$ as
defined in~\eqref{eq::hatGamma}, where $\tr(B)$ is as defined in
\eqref{eq::trBest}.
Choose for $D_2 = 2 (\twonorm{A} + \twonorm{B})$ and $D_0 = \sqrt{\tau_B} + \sqrt{a_{\max}}$,
\bens
\mu \asymp D_2 K^2  \sqrt{\frac{\log m}{f}} \;\;\text{ and } \; \; 
\tau \asymp D_0 K M_{\e} \sqrt{\frac{\log m}{f}}.
\eens
Then with probability at least $1-\frac{c'}{m^2} - 2 \exp(-\delta^2 f/2000 K^4)$, for $2 \ge q \ge 1$,
\ben
\label{eq::ellqnorm}
\norm{\hat{\beta} -\beta^*}_q \le C D_2  K^2  d_0^{1/q} \sqrt{\frac{\log m}{f}}
\left(\twonorm{\beta^*} + \frac{M_{\e}}{K}\right).
\een
Under the same assumptions, the predictive risk admits the following
bounds with the same probability as above,
\bens
\inv{f} \twonorm{X (\hat{\beta} -\beta^*)}^2 \le C' D_2^2  K^4  d_0 \frac{\log m}{f}
\left(\twonorm{\beta}^* + \frac{M_{\e}}{K}\right)^2
\eens
where $c', C_0, C, C' > 0$ are some absolute constants.
\end{theorem}
We give an outline of the proof of Theorem~\ref{thm::DS} in
Section~\ref{sec::proofall} while leaving the detailed proof in
Section~\ref{sec::proofofDSthm}.

\noindent{\bf Discussions.}
Similar results have been derived in~\cite{LW12,BRT14},
however, under different assumptions on the distribution of the noise
matrix $W$. When $W$ is a random matrix with i.i.d. subgaussian noise,
our results will essentially recover the results in~\cite{LW12} and~\cite{BRT14}. 
The choice of $\lambda$ for the Lasso estimator and parameters
$\mu, \tau$ for the DS-type  estimator satisfy
\bens
\lambda \asymp \mu \twonorm{\beta^*} + \tau
\eens
This relationship is made clear through Theorem~\ref{thm::main}
regarding the Lasso-type estimator, which follows from Theorem
1~\cite{LW12}, 
Lemmas~\ref{lemma::low-noise},~\ref{lemma::DS},~\ref{lemma::D2improv},
and~\ref{lemma::DSimprov},
 which are the key results in proving Theorems~\ref{thm::lasso},~\ref{thm::DS},~\ref{thm::lassora}, and~\ref{thm::DSoracle}.
Finally, we note that following Theorem~2 as in~\cite{BRT14}, one can
show that without the relatively restrictive sparsity
condition~\eqref{eq::sqrt-sparsity}, a bound similar to that
in~\eqref{eq::ellqnorm} holds, however with $\twonorm{\beta^*}$ being
replaced by $\onenorm{\beta^*}$,  so long as the sample size satisfies
the requirement as in \eqref{eq::samplebound}.

\section{Improved bounds when the measurement errors are small}
\label{sec::smallW}
Throughout our analysis of Theorems~\ref{thm::lasso}
and~\ref{thm::DS}, we focused on the case when the errors in $W$
are sufficiently large in the sense that $\tau_B  = \tr(B)/f > 0$ is
bounded from below;  
for example, this is explicitly indicated by
the lower bound on the effective rank $r(B) = \tr(B)/\twonorm{B}$, when
$\twonorm{B}$ is bounded away from $0$. More precisely, 
by the condition on the effective rank as in \eqref{eq::trBLasso}, we have 
\bens
\tau_{B} = \frac{\tr(B)}{f} & \ge &
 16c' K^4 \frac{\twonorm{B}}{\log m} \log \frac{\V m \log m }{f} \; \;
 \text{ where} \; \; \V = 3eM_A^3/2.
\eens
The bounds we derive in this section focus on cases
where the measurement errors in $W$ are {\it small} in their magnitudes 
in the sense of $\tau_B$ being small. For the extreme case when $\tau_B$
approaches $0$, one hopes to recover a bound close to the
regular Lasso or the Dantzig selector as the effect of the noise on
the procedure should become negligible. We show in Theorems~\ref{thm::lassora}
and~\ref{thm::DSoracle} that this is indeed the case. First, we define
some contants  which we use throughout the rest of the paper.
Denote by 
\ben
\label{eq::defineD0}
& & D_0 = \sqrt{\tau_B} + a_{\max}^{1/2}, \; \; \; D_0' =
\twonorm{B}^{1/2} + a_{\max}^{1/2},  \; \text{ and } \; \tau_B^+ :=
(\tau_B^{+/2})^2 \\
\label{eq::defineDtau}
&& 
\text{ where} \; \; \tau_B^{+/2}  :=   \sqrt{\tau_B} +
\frac{D_{\ora}}{\sqrt{m}}\; \text{ and } \;
D_{\ora} \;  = \; 2(\twonorm{A}^{1/2} + \twonorm{B}^{1/2}).
\een
We first state a more refined result for the Lasso-type estimator.
\begin{theorem}
\label{thm::lassora}
Suppose all conditions in Theorem~\ref{thm::lasso} hold, except that
we drop \eqref{eq::snrcond} and replace~\eqref{eq::psijune} with
\ben
\label{eq::psijune15}
&& \lambda \ge 2 \psi \sqrt{\frac{\log m}{f}} \; \; \text{ where } \;\;
\psi:= 2 C_0 D_0' K \left(\tau_B^{+/2} K \twonorm{\beta^*} + M_{\e} \right).
\een
Suppose that for $0< \phi \le 1$ and $C_A := \inv{128 M_{A}^2}$
\ben
\label{eq::doracle}
&& d:= \abs{\supp(\beta^*)} \le C_A \frac{f}{\log m} \left\{c' C_{\phi}
  \wedge 2 \right\} \;\;\text{ where }  \\
\nonumber
&& C_\phi := \frac{\twonorm{B} + a_{\max}}{D^2} D_{\phi} \; \;
\text{ for  }  \; \; D_{\phi}
 = \frac{M^2_{\e}K^2}{ b_0^2}  + \tau_B^{+} K^4  \phi,
\een
$D = \rho_{\max}(s_0, A) +\tau_B$, and $c', \phi, b_0, M_{\e}$ and $K$ as defined in
Theorem~\ref{thm::lasso}. 

Then for any $d$-sparse vectors 
$\beta^* \in \R^m$, such that
$\phi b_0^2 \le \twonorm{\beta^*}^2 \le b_0^2$, 
we have with probability  at least $1- 16/m^3$,
\bens
\twonorm{\hat{\beta} -\beta^*} \leq \frac{20}{\alpha}  \lambda \sqrt{d} \; \;
\text{ and } \; \norm{\hat{\beta} -\beta^*}_1 \leq \frac{80}{\alpha}
\lambda d.
\eens
\end{theorem}
We give an outline for the proof of Theorem~\ref{thm::lassora} in
Section~\ref{sec::lassooracle}, and show the actual proof in Section~\ref{sec::classoproof}.

\begin{remark}
Let us redefine the Signal-to-noise ratio by 
\bens
\SMR  &:= & 
\frac{K^2 \twonorm{\beta^*}^2}{\tau_B^+ K^2 \twonorm{\beta^*}^2 +
  M^{2}_{\e} } \; \; \text{ where } \;\\
\signal & := & K^2 \twonorm{\beta^*}^2 \; \text{ and
}  \; \; \bignoise := M^{2}_{\e} + \tau_B^+ K^2 \twonorm{\beta^*}^2 \; 
\eens
We now only require that
$\lambda  \asymp (a^{1/2}_{\max} + \twonorm{B}^{1/2}) K \sqrt{\bignoise}  \sqrt{\frac{\log m}{f}}.$
That is, when either the noise level $M_{\e}$ or the measurement error strength in terms of $\tau_B^{+/2} K \twonorm{\beta^*}$ increases, 
we need to increase the penalty parameter $\lambda$ correspondingly;
moreover, when $d \asymp \inv{M_A^2} \frac{f}{\log m}$ 
\bens
\frac{\twonorm{\hat{\beta} -\beta^*}}{\twonorm{\beta^*} }
\leq \frac{20}{\alpha}  D_0' K^2 \sqrt{\frac{\bignoise}{\signal}}
\inv{M_A} \asymp D_0' K^2 \sqrt{\frac{\bignoise}{\signal}}
\inv{\vp(s_0)},
\eens
which eventually becomes a vacuous bound when $\bignoise \gg \signal$.
\end{remark}

\subsection{A Corollary for Theorem~\ref{thm::DS}}
\label{sec::DScoro}
We next state in Theorem~\ref{thm::DSoracle} an improved bound for
the Conic programming estimator~\eqref{eq::Conic}, which improves upon
Theorem~\ref{thm::DS} when $\tau_B$ is small.
\begin{theorem}
\label{thm::DSoracle}
Suppose all conditions in Theorem~\ref{thm::DS} hold, except that
we replace the condition on $d$ as in~\eqref{eq::sqrt-sparsity} with
the following.
Suppose that the sample size $f$ and the size of the support of $\beta^*$ satisfy the following
requirements: for $C_{6} \ge D_{\ora}$ and  $r_{m,m} =2 C_0 \sqrt{ \frac{\log m}{fm}}$,
\ben
\label{eq::ora-sparsity}
d_0 & = & O \left(\tau_B^-\sqrt{\frac{f}{\log m}} \right) 
\; \;  \text{ where } \; 
\tau_B^- \le \inv{\tau_B^{1/2} + 2C_{6} K r_{m,m}^{1/2}} \\
\label{eq::samplebound}
\; \; \text{ and } \; \; 
f & \geq & \frac{2000 d K^4}{\delta^2} \log \left(\frac{60 e m}{d
    \delta}\right) \; \text{where } \\
\label{eq::sparse-dim-Ahalf}
d & = & 2d_0 + 2d_0 a_{\max} \frac{16 K^2(2d_0, 3k_0, A^{1/2}) (3k_0)^2
  (3k_0 + 1)}{\delta^2}.
\een
Let $\hat\beta$ be an optimal solution to the Conic Programming estimator
as in \eqref{eq::Conic} with input $(\hat\gamma, \hat\Gamma)$ as
defined in~\eqref{eq::hatGamma}, where $\hat\tr(B)$ is as defined in
\eqref{eq::trBest}.  
Suppose
\ben
\label{eq::tauchoice}
\tau & \asymp & D_0 M_{\e}  r_{m,f} \; \; \text{ where } \; \; r_{m,f} = C_0
K \sqrt{\frac{\log  m}{f}} \; \text{ and } \\
\label{eq::muchoice}
\mu & \asymp & D_0' \tilde\tau_B^{1/2} K r_{m,f}\;\; \; \text{ where}
\; \; \tilde\tau_B^{1/2}  := \PaulBhalf+ C_{6} K r_{mm}^{1/2}.
\een
Then with probability at least $1-\frac{c''}{m^2} - 2 \exp(-\delta^2
f/2000 K^4)$, for $2 \ge q \ge 1$, and $\tau_B^{\dagger/2} 
=(\tau_B^{1/2} + \frac{3}{2}C_{6} K r_{m,m}^{1/2})$ 
\ben
\label{eq::ellqnormimp}
\norm{\hat{\beta} -\beta^*}_q  \le  C' D_0' K^2 d_0^{1/q} \sqrt{\frac{\log m}{f}} 
\left(\tau_B^{\dagger/2} \twonorm{\beta^*} + \frac{M_{\e}}{K}\right);
\een
Under the same assumptions, the predictive risk admits the following bounds
\bens
\onef \twonorm{X (\hat{\beta} -\beta^*)}^2 \le 
C'' (\twonorm{B} + a_{\max}) K^2 d_0 \frac{\log m}{f} \left(\tau_B^{\ddagger} 
 K^2 \twonorm{\beta^*}^2 + M_{\e}^2\right)
\eens
 with the same probability as above, where $c'', C', C'' > 0$ are some
 absolute constants, and $\tau_B^{\ddagger} \asymp 2 \tau_B+ 3 C_{6}^2 K^2 r_{m,m}$.
\end{theorem}

\subsection{Discussions}
\label{sec::discusslassocoro}
In particular, when $\tau_B \to 0$, Theorem~\ref{thm::DSoracle} 
allows us to recover a rate close that of the Dantzig selector with an exact 
recovery if $\tau_B =0$ is known a priori; see Section~\ref{sec::conclude}.
Moreover the constraint~\eqref{eq::sqrt-sparsity} on the sparsity
parameter $d_0$ appearing in Theorem~\ref{thm::DS} can now be
relaxed as in \eqref{eq::ora-sparsity}. Roughly speaking, one can think of $d_0$ being bounded 
as follows for the Conic programming estimator~\eqref{eq::Conic}:
\ben
\label{eq::ora-sparsity-rem}
d_0 & = & O \left(\tau_B^-\sqrt{\frac{f}{\log m}}  \bigwedge
  \frac{f}{ \log(m/d_0) }\right)
\; \;  \text{ where } \; 
\tau_B^- \asymp \inv{\tau_B^{1/2}}
\een
That is,  when $\tau_B$ decreases, we allow larger values of $d_0$; however, when 
$\tau_B \to 0$, the sparsity level of $d = O\left(f/ \log (m/d)\right)$ starts to 
dominate, which enables the Conic Programming estimator to achieve results similar
to the Dantzig Selector when the design matrix $X_0$ is a subgaussian random matrix satisfying the Restricted Eigenvalue conditions;
See for example~\cite{CT07,BRT09,RZ13}.

The condition on $d$ (and  $D_{\phi}$) for the Lasso estimator as defined in~\eqref{eq::doracle} 
suggests that as $\tau_B \to 0$, and thus $\tau^+_B \to 0$  
the requirement on the sparsity parameter $d$ becomes slightly more
stringent when $K^2 M_{\e}^2/b_0^2 \asymp 1$ and much more restrictive when
$K^2 M_{\e}^2/b_0^2 = o(1)$; however, suppose we require
\bens
M_{\e}^2 = \Omega(\tau_B^+ K^2 \twonorm{\beta^*}^2),
\eens 
that is, the stochastic error $\e$ in the response variable $y$ as in~\eqref{eq::oby} 
does not converge to $0$ as quickly as the measurement error $W$ in
\eqref{eq::obX} does, then the sparsity constraint becomes essentially 
unchanged as $\tau_B^+ \to 0$.
In this case, essentially, we require that for some $c'' := c' C_{\phi}$ 
\bens 
&& 
d  \le  C'_A \frac{f}{\log m} \left\{\frac{c''K^2 M^2_{\e}}{b_0^2} 
  \wedge 1 \right\} \;\;\text{ where }\;  D_{\phi}  \asymp \frac{K^2
  M^2_{\e}}{b_0^2} \text{ and } \; C'_A := \inv{64 M_{A}^2}, \\
&& \text{ given that } \; \; \tau_B^+ K^4 \phi \le  \frac{\tau_B^+ K^4
  \twonorm{\beta^*}^2}{b_0^2} \ll \frac{K^2 M^2_{\e}}{b_0^2}.
\eens 
These tradeoffs are somehow different from the behavior of the Conic 
programming estimator (cf \eqref{eq::ora-sparsity-rem}).
\silent{
Analogous to~\eqref{eq::dlasso}, we could represent the condition on $d$ as follows:
we note that the following condition on $d$ is enough for~\eqref{eq::doracle} to hold:
\bens
d & \le &
 C_A c'' \tau_B^+ K^4 \frac{\phi f}{\log m}
\le  C_A c'' D_{\phi} K^4 \frac{f}{\log m}
\eens
Indeed, for $c'' \tau_B^+ K^4 \le 1$, we have
\bens
d  & \le  & C_A (c'' \tau_B^+ K^4 \wedge 1)\frac{\phi f}{\log m}
\le C_A  \left(c'' D_{\phi} \wedge  1\right) \frac{f}{\log  m}
\eens
where clearly
\bens
D_{\phi}
 = \frac{M^2_{\e}K^2}{ b_0^2}  + \tau_B^{+} K^4  \phi \ge \tau_B^{+} K^4  \phi \ge \tau_B^{+} \phi.
\eens
This condition, however, is unnecessarily strong, when $\tau_B \to 0$, in view of the previous discussion. 
}

\section{Proof of theorems}
\label{sec::proofall}
We first consider the following 
large deviation bound on $\norm{\hat\gamma - \hat\Gamma \beta^*}$ 
as stated in Lemma~\ref{lemma::low-noise}. 
This entity appears in the constraint set in the conic programming estimator
\eqref{eq::Conic}, and is directly related to the choice of $\lambda$ 
for the lasso-type estimator in view of Theorem~\ref{thm::main}.
Events $\B_0$ and $\B_{10}$ are defined in Section~\ref{sec::stoc} in
the Appendix.
\begin{lemma}
\label{lemma::low-noise}
Suppose (A1) holds. 
Let $X = X_0 + W$, where $X_0, W$ are as defined in
Theorem~\ref{thm::lasso}. Suppose that
$$\fnorm{B}^2/\twonorm{B}^2 \ge \log m \; \; \text{ where } \; m \ge 16.$$
Let $\hat\Gamma$ and $\hat\gamma$ be as in~\eqref{eq::hatGamma}.
On event $\B_0$, we have for $D_2 = 2(\twonorm{A} +\twonorm{B})$ and some absolute constant $C_0$ 
\bens
\norm{\hat\gamma - \hat\Gamma \beta^*}_{\infty}
\le  \psi \sqrt{\frac{\log m}{f}} \; \text{ where } \;  
\psi =  C_0 D_2 K \left(K \twonorm{\beta^*} + M_{\e} \right)
\eens
is as defined in Theorem~\ref{thm::lasso}.
Then $\prob{\B_0} \ge 1- 16/m^3$.
\end{lemma}

\begin{lemma}
\label{lemma::trBest}
Let $m \ge 2$.
Let $X$ be defined as in Definition~\ref{def::subgdata} and
$\hat\tau_B$ be as defined in \eqref{eq::trBest}. Denote by $\tau_B = \tr(B)/f$ and $\tau_A = \tr(A)/m$.
Suppose that $f \vee (r(A)  r(B)) > \log m$.
Denote by $\B_6$ the event such that 
\bens
\abs{\hat\tau_B - \tau_B} & \le &  
2 C_0 K^2 \sqrt{\frac{\log m}{m f}}
\left(\frac{\fnorm{A}  }{\sqrt{m}} +\frac{\fnorm{B}  }{\sqrt{f}}
\right) =:  D_1 K^2 r_{m,m} \\ 
\text{ where} \; \; D_1 & := &  \frac{\fnorm{A}}{\sqrt{m}} +
\frac{\fnorm{B}}{\sqrt{f}} \; 
\text{ and } \; \;   r_{m,m} = 2 C_0 \sqrt{\frac{\log m}{m f}}.
\eens 
Then $\prob{\B_6} \ge 1-\frac{3}{m^3}$.
If we replace $\sqrt{\log m}$ with $\log m$ in the definition of event
$\B_6$, then we can drop the condition on $f$ or $r(A)r(B) =
\frac{\tr(A)}{\twonorm{A}} \frac{\tr(B)}{\twonorm{B}}$ to achieve 
the same bound on event $\B_6$.
\end{lemma}
We prove Lemma~\ref{lemma::trBest} in
Section~\ref{sec::proofoferrors} in the Appendix.
We prove Lemma~\ref{lemma::low-noise} in Section~\ref{sec::lownoiseproof}. 
We mention in passing that Lemma~\ref{lemma::low-noise} is essential
in proving Theorem~\ref{thm::DS} as well.

We state variations on this inequality in Lemma~\ref{lemma::D2improv}
and the remark which immediately follows.
\begin{theorem}
\label{thm::main}
Consider the regression model in 
\eqref{eq::oby} and \eqref{eq::obX}.  Let $d \le f/2$.
Let $\hat\gamma, \hat\Gamma$ be as constructed in \eqref{eq::hatGamma}.
Suppose that the matrix $\hat\Gamma$ satisfies the Lower-$\RE$ condition with curvature $\alpha >0$ and
tolerance $\tau > 0$,
\ben
\label{eq::taumain}
\sqrt{d} \tau \le \min
 \left\{\frac{\alpha}{32 \sqrt{d}}, \frac{\lambda}{4b_0} \right\}
\een
where $d, b_0$ and $\lambda$ are as defined in \eqref{eq::origin}.
Then for any $d$-sparse vectors $\beta^* \in \R^m$, such that
$\twonorm{\beta^*} \le b_0$ and 
\ben
\label{eq::psimain}
&& \norm{\hat\gamma - \hat\Gamma \beta^*}_{\infty}
\le  \half \lambda, \text{ the following bounds hold:}  \\
\label{eq::2-loss}
&& \twonorm{\hat{\beta} -\beta^*} \leq \frac{20}{\alpha}  \lambda \sqrt{d}, \; \;
\text{ and } \; \norm{\hat{\beta} -\beta^*}_1 \leq \frac{80}{\alpha} \lambda d
\een
where $\hat\beta$ is an optimal solution to the Lasso-type estimator as in~\eqref{eq::origin}.
\end{theorem}
We defer the proof of Theorem~\ref{thm::main} to
Section~\ref{sec::proofofmain}, for clarity of presentation.
In section~\ref{sec::lassooutline}, we provide two Lemmas~\ref{lemma::lowerREI} and~\ref{lemma::dmain}
 in checking the RE conditions as well condition \eqref{eq::taumain}.
One can then combine with Theorem~\ref{thm::main}, Lemmas~\ref{lemma::low-noise},
~\ref{lemma::lowerREI} and~\ref{lemma::dmain}
 to prove Theorem~\ref{thm::lasso}. 
In more details, Lemma~\ref{lemma::lowerREI} checks the Lower and the Upper $\RE$ conditions on the modified gram matrix:
\ben
\label{eq::modifiedgramA}
\hat \Gamma_A :=   X^T X - \hat\tr(B) I_{m}
\een
while Lemma~\ref{lemma::dmain} checks condition \eqref{eq::taumain} as
stated in Theorem~\ref{thm::main} for curvature $\alpha$ and tolerance $\tau$ 
as derived in Lemma~\ref{lemma::lowerREI}. Finally Lemma~\ref{lemma::low-noise} ensures that 
\eqref{eq::psimain} holds with high probability for $\lambda$ 
chosen as in~\eqref{eq::psijune}. We defer stating these lemmas in Section~\ref{sec::lassooutline}.
The full proof of Theorem~\ref{thm::lasso} appears in Section~\ref{sec::proofofthmlasso}. 

For Theorem~\ref{thm::DS}, our first goal is to show that the
following holds  with high probability
\bens
\norm{\hat\gamma - \hat\Gamma  \beta^*}_{\infty} 
& = & 
\norm{\onef X^T(y - X \beta^*) + \onef \hat\tr(B) \beta^*}_{\infty} \;
\le \; \mu \twonorm{\beta^*}  + \tau,
\eens
where  $\mu, \tau$ are as chosen in \eqref{eq::paraDS}.
This forms the basis for proving the $\ell_q$ convergence, where $q
\in [1, 2]$,  for the Conic Programming estimator \eqref{eq::Conic}.
This follows immediately from Lemma~\ref{lemma::low-noise}. More
explicitly, we will state it in Lemma~\ref{lemma::DS}.
Before we proceed, we first need to introduce some notation and definitions.
Let $X_0 = Z_1 A^{1/2}$ be defined as in Definition~\ref{def::subgdata}. Let $k_0 = 1+\lambda$. 
First we need to define the $\ell_q$-sensitivity parameter for $\Psi
:= \onef X_0^T X_0$ following~\cite{BRT14}: 
\ben
\label{eq::sense}
\kappa_{q}(d_0, k_0) 
& = & \min_{J : \abs{J} \le d_0} \min_{\Delta \in \Cone_J(k_0)}
\frac{\norm{\Psi \Delta}_{\infty}}{\norm{\Delta}_q} \; \; \text{ where
} \; \\
\W_J(k_0) &  = & \left\{x \in \R^m \;| \; \mbox{ s.t. } \; \norm{x_{J^c}}_1 \leq k_0 \norm{x_{J}}_1 \right\}.
\een
See also~\cite{GT11}.
Let $(\hat\beta, \hat{t})$ be the optimal solution
to~\eqref{eq::Conic} and denote by $v = \hat\beta-\beta^*$.
We will state the following auxiliary lemmas, the first of which is 
deterministic in nature.
The two lemmas reflect the two geometrical constraints 
on the optimal solution to \eqref{eq::Conic}. 
The optimal solution $\hat\beta$ satisfies:
\bnum
\item
$v$ obeys the following cone constraint:
$\onenorm{v_{\Sc}} \le   k_0 \onenorm{v_S}$ and $\hat{t} \le
\inv{\lambda} \onenorm{v} + \twonorm{\beta^*}$.
\item
$\norm{\Psi v}_{\infty}$ is upper bounded by a 
quantity at the order of $O\left(\mu (\twonorm{\beta^*} + \onenorm{v}) + \tau\right)$
\enum
Now combining Lemma 6 of~\cite{BRT14} and an earlier result of the two
authors (cf. Theorem~\ref{thm:subgaussian-T-intro}~\cite{RZ13}), we can show that 
 the $\RE(2d_0,  3(1+ \lambda), A^{1/2})$ condition and the sample requirement as in~\eqref{eq::samplebound}
 are enough to ensure that the $\ell_q$-sensitivity parameter satisfies the following
lower bound for all $1\le q \le 2$: for some contant $c$,
\ben
\nonumber
\kappa_{q}(d_0, k_0) & \ge & c d_0^{-1/q} \; \;  \text{ which ensures
  that for } \; \; v =  \hat\beta-\beta^*, \\
\label{eq::lqcond}
{\norm{\Psi v}_{\infty}}  &\ge & \kappa_{q}(d_0,
k_0) {\norm{v}_q} \ge  c d_0^{-1/q} \norm{v}_q \; \; \text{ where } \;
\; \Psi = \onef X_0^T X_0.
\een
Combining \eqref{eq::lqcond} with Lemmas \ref{lemma::DS},~\ref{lemma::DS-cone} 
and~\ref{lemma::grammatrix} gives us both the lower and upper bounds on 
$\norm{\Psi v}_{\infty}$, with the lower bound being 
$\kappa_{q}(d_0, k_0) \norm{v}_q$ and the upper bound as specified in Lemma~\ref{lemma::grammatrix}.
Following some algebraic manipulation, this yields the bound on the
$\norm{v}_q$ for all $1\le q \le 2$. 
We state Lemmas \ref{lemma::DS} to \ref{lemma::grammatrix} in
Section~\ref{sec::DSoutline} while leaving the proof for 
Theorem~\ref{thm::DS} in Section~\ref{sec::proofofDSthm}.

\subsection{Additional technical results for Theorem~\ref{thm::lasso}}
\label{sec::lassooutline}
The main focus of the current section is
to apply Theorem~\ref{thm::main} to show Theorem~\ref{thm::lasso}, which applies to the general
 subgaussian model as considered in the present work.
We first state Lemma~\ref{lemma::lowerREI}, which follows immediately
from Corollary~\ref{coro::BC}.
First, we replace (A3) with (A3') which reveals some additional information
regarding the constant hidden inside the $O(\cdot)$ notation.
\bnum
\item[(A3')]
Suppose (A3) holds; moreover, for $D_2 = 2(\twonorm{A} +
\twonorm{B})$, $m f \ge 1024 C_0^2 D_2^2 K^4 \log
m/{\lambda_{\min}^2(A)}$ or equivalently,
\bens
\frac{\lambda_{\min}(A)}{\twonorm{A} + \twonorm{B}}
> C_K \sqrt{\frac{\log m}{mf}}  \text{ for some large enough contant $C_K$}.
\eens
\enum
\begin{lemma}{{\bf(Lower and Upper-RE conditions)}}
\label{lemma::lowerREI} 
Suppose (A1), (A2) and (A3') hold. 
Denote by $\V := 3eM_A^3 /2$, where $M_A$ is as defined
in~\eqref{eq::defineM}. 
Let $s_0$ be as defined in \eqref{eq::s0cond}.
Suppose  that for some $c' > 0$, 
\ben
\label{eq::trBlem}
\frac{\tr(B)}{\twonorm{B}} 
& \ge & c' K^4 \frac{s_0}{\ve^2} \log\left(\frac{3 e m}{s_0
    \ve}\right)\; \; \;\text{ where } \; \ve =\inv{2 M_A}.
\een
Let $\A_0$ be the event that the modified gram
matrix $\hat\Gamma_A$ as defined in~\eqref{eq::modifiedgramA}
satisfies the Lower as well as Upper $\RE$ conditions with
\bens
\text{curvature} && 
\alpha = \half \lambda_{\min}(A), \;\text{smoothness} \; \; \bar\alpha
= 3 \lambda_{\max}(A)/2, \\
\text{ and tolerance } && 
\frac{512 C^2 \vp(s_0)^2}{\lambda_{\min}(A)}\frac{\log m}{f}
 \le  \tau :=  \frac{\alpha}{s_0}  \le \frac{1024 C^2 \vp^2(s_0+1)}{\lambda_{\min}(A)}\frac{\log m}{f}
\eens
for $\alpha, \bar\alpha$ and $\tau$ as defined in
Definitions~\ref{def::lowRE} and \ref{def::upRE}, and $C, s_0, \vp(s_0)$ in~\eqref{eq::s0cond}.
Then $\prob{\A_0} \ge 1- 4\exp\left(-\frac{c_3 f}{M_A^2 \log m}
  \log\left(\frac{\V m \log m}{f}\right)\right) -2\exp\left(- \frac{4c_2 f}{M_A^2 K^4 }\right) - 6/m^3$.
\end{lemma}

\begin{lemma}
  \label{lemma::dmain}
Suppose all conditions in Lemma~\ref{lemma::lowerREI} hold.
Suppose that $s_0 \ge 3$ and
\ben
\label{eq::dlassoproof}
&& d:= \abs{\supp(\beta^*)} \le C_A \frac{f}{\log m} \left\{c' D_\phi
  \wedge 2 \right\} \;\;\text{ where }  C_A := \inv{128 M_{A}^2}, \\
\nonumber
&& D_{\phi} = \left(\frac{K^2M^2_{\e}}{b_0^2}   +  K^4 \phi \right) \ge
 K^4 \phi \ge \phi
\een
where $c', \phi, b_0, M_{\e}$ and $K$ are as defined in
Theorem~\ref{thm::lasso}, where we assume that
$ \twonorm{\beta^*}^2 \ge \phi b_0^2  \; \text{ for some } \; 0 < \phi
\le 1$.
Then the following condition holds
\ben
\label{eq::dcond}
d \le  \frac{s_0}{32} \bigwedge
\left(\frac{s_0}{\alpha}\right)^2 \frac{\log m}{f} \left(\frac{\psi}{b_0}\right)^2
\een
where $\psi$ is as defined in~\eqref{eq::psijune} and 
$\alpha = \lambda_{\min}(A)/2$.
\end{lemma}
We prove Lemmas~\ref{lemma::lowerREI} and \ref{lemma::dmain} in Sections~\ref{sec::records} and~\ref{sec::dmain} respectively.

\begin{remark}
Clearly for $d, b_0, \phi$ as bounded in Theorem~\ref{thm::lasso}, 
we have by assumption~\eqref{eq::snrcond}
the following upper and lower bound on $D_{\phi}$:
\bens
2 K^4 \phi \ge D_{\phi} := \left(\frac{M^2_{\e}K^2}{b_0^2}   +  K^4 \phi \right) \ge
K^4 \phi.
\eens
In this regime, the conditions on $d$ as in~\eqref{eq::dlassoproof}
can be conveniently expressed as in \eqref{eq::dlasso}. 
\end{remark}

\subsection{Technical lemmas for Theorem~\ref{thm::DS} }
\label{sec::DSoutline}
We state the technical lemmas needed for proving  Theorem~\ref{thm::DS}.
The proof for Lemma~\ref{lemma::DS-cone} follows directly from that
in~\cite{BRT14} in view of  Lemma~\ref{lemma::DS}.
\begin{lemma}
\label{lemma::DS}
Suppose all conditions in Lemma~\ref{lemma::low-noise} hold.
Then on event $\B_0$ as defined therein,
the pair $(\beta, t) = (\beta^*, \twonorm{\beta^*})$ 
belongs to the feasible set of the minimization problem
\eqref{eq::Conic} with $r_{m,f} := C_0 K  \sqrt{\frac{\log m}{f}}$,
\ben
\label{eq::paraDS}
\mu \asymp 2 D_2 K r_{m,f} \; \; \text{ and } \; \; 
\tau \asymp D_0 M_{\e} r_{m,f} 
\een
where $D_0 =  (\sqrt{\tau_B} + \sqrt{a_{\max}})$ 
and $D_2 = 2 (\twonorm{A} + \twonorm{B})$ as in Theorem~\ref{thm::DS}.
\end{lemma}
\begin{lemma}
\label{lemma::DS-cone}
Let $\mu, \tau >0$ be set.
Suppose that the pair $(\beta, t) = (\beta^*, \twonorm{\beta^*})$ belongs to 
the feasible set of the minimization problem \eqref{eq::Conic}, for which
$(\hat\beta, \hat{t})$ is an optimal solution.
Denote by $v = \hat\beta-\beta^*$. Then
\bens
\onenorm{v_{\Sc}} & \le &  (1+\lambda) \onenorm{v_S} \; \text{ and }
\; \hat{t} \; \le \; \inv{\lambda}\onenorm{v} + \twonorm{\beta^*}.
\eens
\end{lemma}

\begin{lemma}
\label{lemma::grammatrix}
On event $\B_0\cap \B_{10}$, 
\bens
\norm{\Psi v}_{\infty} \le \mu_1 \twonorm{\beta^*} + \mu_2
\onenorm{v} + \tau
\eens
where $\mu_1 = 2 \mu$, $\mu_2  = \mu(\inv{\lambda} + 1)$ and $\tau' = 2 \tau$
for $\mu, \tau$ as defined in \eqref{eq::paraDS}.
\end{lemma}
We prove Lemmas~\ref{lemma::DS},~\ref{lemma::DS-cone} and~\ref{lemma::grammatrix} in Section~\ref{sec::proofofDSlemma}.

\subsection{Improved bounds for the Lasso-type estimator}
\label{sec::lassooracle}
We give an outline illustrating where the improvement for the lasso error
bounds as stated in Theorem~\ref{thm::lassora} come from.
We emphasize the impact of this improvement over sparsity parameter $d_0$.
The proof for Theorem~\ref{thm::lassora} 
follows exactly the same
line of arguments as in Theorem~\ref{thm::lasso} except that we now
use the improved bound on the error term 
$\norm{\hat\gamma -   \hat\Gamma \beta^*}_{\infty}$ given in Lemma~\ref{lemma::D2improv}
instead of that in Lemma~\ref{lemma::low-noise} which is used in
proving Theorems~\ref{thm::lasso} and~\ref{thm::DS}. 
See Section~\ref{sec::classoproof} for details, as well as the proof
for Theorem~\ref{thm::lassora} and  the following two lemmas.
\begin{lemma}
\label{lemma::D2improv}
Suppose all conditions in Lemma~\ref{lemma::low-noise} hold.
Let $D_0, D_0', D_{\ora}$, and
$\tau_B^{+/2} :=  \tau_B^{1/2}  +  \frac{D_{\ora}}{\sqrt{m}}$  be as defined in~\eqref{eq::defineD0} and~\eqref{eq::defineDtau}.
On event $\B_0$,
\ben
\norm{\hat\gamma - \hat\Gamma \beta^*}_{\infty}
& \le & 
\label{eq::psioracle} 
\psi \sqrt{\frac{\log m}{f}}
\een
where 
$\psi:= C_0 K \left(D_0' \tau_B^{+/2} K \twonorm{\beta^*} + D_0  M_{\e} \right)$.
Then $\prob{\B_0} \ge 1- 16/m^3$.
\end{lemma}
Moreover, we replace Lemma~\ref{lemma::dmain} with Lemma
\ref{lemma::dmainoracle}, the proof of which follows from Lemma~\ref{lemma::dmain} with 
$d$ now being bounded as in \eqref{eq::doracle}  and $\psi$ being
redefined as immediately above in \eqref{eq::psioracle}.
\begin{lemma}
\label{lemma::dmainoracle}
Suppose all conditions in Lemma~\ref{lemma::lowerREI} hold.
Suppose that \eqref{eq::doracle} holds.
Then~\eqref{eq::dcond} holds with $\psi$ as defined in Theorem~\ref{thm::lassora} and 
$\alpha = \lambda_{\min}(A)/2$.
\end{lemma}

\subsection{Improved bounds for the DS-type estimator}
\label{sec::DSoracle}
An ``oracle'' rate for the Conic programming estimator~\eqref{eq::Conic} is defined
as follows. Recall the following notation: $r_{m,f} = C_0 K \sqrt{\frac{\log m}{f}}$.
 The trick is that we assume that we know the noise level
in $W$ by knowing $\tau_B := \tr(B)/f$, then we can set 
\bens
\mu \asymp D_0' (\tau_B^{1/2}  + D_{\ora}/\sqrt{m}) K  r_{m,f}
\; \; \text{ while retaining} \; \; \tau \asymp D_0 M r_{m,f}
\eens
in view of the improved error bounds over 
$\norm{\hat\gamma -   \hat\Gamma \beta^*}_{\infty}$ as given in
Lemma~\ref{lemma::D2improv}.
Without knowing this parameter, we could rely on the estimate from
$\hat\tau_B$ as in~\eqref{eq::trBest}, which is what we do next.
For a chosen parameter $C_{6}$, we use $\PaulBhalf + C_{6} K
r_{m,m}^{1/2}$ to replace $\tau_B^{+/2} := \tau_B^{1/2}  +
D_{\ora}/\sqrt{m}$ and set
\bens
\mu & \asymp & C_0 D_0' (\PaulBhalf + C_{6} K r_{m,m}^{1/2}) K^2  
\sqrt{\frac{\log m}{f}} \; \; \; \text{ where } \;\;  C_{6} \ge
D_{\ora}, \\ 
r_{m,m} & = & C_0 \sqrt{\frac{\log m}{mf}} > C_0 \frac{\sqrt{\log
    m}}{m} \;  \text{ and } \; D_{\ora} :=  2(\twonorm{A}^{1/2}  + \twonorm{B}^{1/2} ).
\eens
Notice that we know neither $D_0'$ nor $D_{\ora}$,
where recall $D_0' =\sqrt{\twonorm{B}} + a_{\max}^{1/2}$.
However, assuming that we normalize 
the column norms of design matrix $X$ to be roughly at the same scale, we have 
\bens
D_0' \asymp 1 \; \; \text{ while } \; \; D_{\ora}/\sqrt{m} = o(1) \;
\; \text{in case} \; \; \twonorm{A}, \twonorm{B} \le M
\eens
for some large enough constant $M$.
This is crucial in deriving and putting the faster rates of convergence in estimating $\hat\beta$ and in
predictive error $\twonorm{X v}^2$ when $\tau_B = o(1)$ in
perspective, in view of Lemmas~\ref{lemma::DSimprov} and~\ref{lemma::grammatrixopt}.
Lemma~\ref{lemma::DSimprov} follows directly from Lemma~\ref{lemma::D2improv}.
\begin{lemma}
\label{lemma::DSimprov}
Suppose all conditions in Lemma~\ref{lemma::D2improv} hold.
Let $D_0 =  (\sqrt{\tau_B} + \sqrt{a_{\max}}) \asymp 1$ under (A1).
Then on event $\B_0$, the pair $(\beta, t) = (\beta^*, \twonorm{\beta^*})$ 
belongs to the feasible set $\U$ of the minimization problem
\eqref{eq::Conic} with
\ben
\label{eq::paraDSimprov}
\mu \ge 
D_0' \tau_B^{+/2} K r_{m,f}  \; \; \text{ and } \; \; 
\tau \ge D_0 M_{\e} r_{m,f}.
\een
where  $\tau_B^{+/2} :=  \tau_B^{1/2}  +  \frac{D_{\ora}}{\sqrt{m}}$  is as defined in~\eqref{eq::defineDtau}.
\end{lemma}

\begin{lemma}
\label{lemma::tauB}
On event $\B_6$ and (A1), the choice of $\tilde\tau_B^{1/2}$ as in
\eqref{eq::muchoice} satisfies for $m \ge 16$ and $C_0 \ge 1$,
\ben
\label{eq::tildetauB}
\tau_B^{+/2}    &  \le &   \tilde\tau_B^{1/2} 
\le \tau_B^{1/2} + \frac{3}{2} C_{6} K r_{mm}^{1/2} =:
\tau^{\dagger/2}_B \\
\label{eq::tildetauBbound}
\tilde\tau_B  & \le &  2 \tau_B +3 C_{6}^2 K^2 r_{mm} \asymp \tau^{\ddagger}_B \; \text{ and
  moreover  }\; \;  \tilde\tau_B^{1/2}  \tau_B^- \le  1
\een
\end{lemma}
We next state an updated result in Lemma~\ref{lemma::grammatrixopt}.
\begin{lemma}
\label{lemma::grammatrixopt}
On event $\B_0\cap \B_{10}$,  the solution $\hat\beta$ to \eqref{eq::Conic} with
$\mu, \tau$ as in \eqref{eq::muchoice}  and \eqref{eq::tauchoice},
satisfies for $v := \hat\beta - \beta^*$
\bens
\norm{\onef X_0^T X_0 v}_{\infty} \le \mu_1 
\twonorm{\beta^*}+ \mu_2 \onenorm{v} + \tau'
\eens
where $\mu_1 = 2 \mu $, $\mu_2  =2 \mu (1+\inv{2 \lambda})$ and 
$\tau' = 2\tau$.
\end{lemma}

\section{Lower and Upper RE conditions}
\label{sec::AD}
The goal of this section is to show that for $\Delta$ defined
in~\eqref{eq::defineAD}, the presumption in 
Lemmas~\ref{lemma::bigcone} and~\ref{lemma::bigconeII} as restated in \eqref{eq::Deltacond} holds with high probability (cf
Theorem~\ref{thm::AD}). 
We first state a deterministic result showing that the Lower and Upper $\RE$
conditions hold for $\hat\Gamma_A$ under condition
\eqref{eq::Deltacond} in Corollary~\ref{coro::BC}.
This allows us to prove Lemma~\ref{lemma::lowerREI} in  Sections~\ref{sec::records}.
See Sections~\ref{sec::geometry} and~\ref{sec::appendLURE},
where we show that Corollary~\ref{coro::BC} follows immediately from
the geometric analysis result as stated in Lemma~\ref{lemma::bigconeII}.
\begin{corollary}
\label{coro::BC}
Let $1/8 > \delta > 0$. Let $1 \le \zeta < m/2$.
Let $A_{m \times m}$ be a symmetric positive semidefinite 
covariance matrice. Let $\hat\Gamma_A$ be an $m \times m$ symmetric
matrix and $\Delta = \hat\Gamma_A - A$.
Let $E=\cup_{|J| \leq \zeta} E_J$, where $E_J = \spin\{e_j: j \in J\}$. 
Suppose that $\forall u, v \in E \cap S^{m-1}$
\ben
\label{eq::Deltacond}
\abs{u^T \Delta v} \le \delta \le \inv{8}\lambda_{\min}(A).
\een
Then the Lower and Upper $\RE$ conditions holds: for all $\up \in \R^m$,
\ben
\label{eq::ADlow}
\up^T \hat\Gamma_A \up & \ge & \half \lambda_{\min}(A)
 \twonorm{\up}^2 
-\frac{\lambda_{\min}(A)}{2 \zeta} \onenorm{\up}^2 \\
\label{eq::ADup}
\up^T \hat\Gamma_A \up & \le &  \frac{3}{2} \lambda_{\max}(A) \twonorm{\up}^2 
+\frac{\lambda_{\min}(A)}{2 \zeta} \onenorm{\up}^2.
\een
\end{corollary}

\begin{theorem}
\label{thm::AD}
Let $A_{m \times m}, B_{f \times f}$ be symmetric positive definite 
covariance matrices. Let $E=\cup_{|J| \leq \zeta} E_J$ for $1 \le \zeta
< m/2$. Let $Z, X$ be $f \times m$
random matrices defined as in Theorem~\ref{thm::lasso}.
Let $\hat\tau_B$ be defined as in~\eqref{eq::trBest}.
Let
\ben
\label{eq::defineAD}
\Delta := \hat\Gamma_{A} -A := \onef X^TX - \hat\tau_B I_{m} -A.
\een
Suppose that for some absolute constant $c' > 0$ and $0 < \ve \le \inv{C}$ 
\ben
\label{eq::trB}
\frac{\tr(B)}{\twonorm{B}} & \ge & 
\left(c' K^4 \frac{ \zeta}{\ve^2} \log\left(\frac{3e m}{\zeta \ve}\right)\right) \bigvee \log m
\een
where $C = C_0/\sqrt{c'}$ for $C_0$ as chosen to
satisfy~\eqref{eq::defineC0}.

Then with probability at least $1- 4 \exp\left(-c_2\ve^2 \frac{\tr(B)}{K^4\twonorm{B}}\right)-2
  \exp\left(-c_2\ve^2 \frac{f}{K^4}\right) - 6/m^3$, where
  $c_2 \ge 2$, we  have for all $u, v \in E \cap S^{m-1}$ and
  $\vp(\zeta) = \tau_B+  \rho_{\max}(\zeta, A)$, and $D_1 \le \frac{\fnorm{A}}{\sqrt{m}}  + \frac{\fnorm{B}}{\sqrt{f}}$,
\bens
\abs{u^T \Delta v} \le  8 C \vp(\zeta) \ve + 4 C_0 D_1 K^2 \sqrt{\frac{\log m}{m f}}.
\eens
\end{theorem}
We prove Theorem~\ref{thm::AD} in Section~\ref{sec::proofofthmAD}.
As a corollary of Theorem~\ref{thm::AD}, we will state Corollary~\ref{coro::sparseK} in Section~\ref{sec::Best}.

\section{Concentration bounds for error-corrected gram matrices}
\label{sec::Best}
In this section, we show an upper bound on the operator norm
convergence as well as an isometry property for estimating $B$ using the
corrected gram matrix $\tilde{B} := \onem (X X^T - \tr(A) I_f)$.
Theorem~\ref{thm::kronopB} and Corollary~\ref{coro::kronopB}
state that for the matrix $B \succ 0$ with the smaller dimension, 
$\tilde{B}$ tends to stay positive definite after this error
correction step with an overwhelming probability, 
where we rely on $f$ being dominated by the effective rank of the
positive definite matrix $A$.
When we subtract a diagonal matrix $\tau_B I_m$ from the gram matrix
$\onef X^TX$ to form an estimator, 
we clearly introduce a large number of negative eigenvalues when $f
\ll m$.  This in general is a bad idea. 
However, the sparse eigenvalues for $\tilde{A}$ can stay pretty close to those of $A$ as
we will show in Corollary~\ref{coro::sparseK}.
\begin{theorem}
\label{thm::kronopB}
Let $\ve >0$.  Let $X$ be defined as in Definition~\ref{def::subgdata}.
Suppose that  for some $c' > 0$ and $0<\ve < 1/2$,
\ben
\label{eq::trAlow}
\frac{\tr(A)}{\twonorm{A}} & \ge & c' f K^4 \frac{\log(3/\ve)}{\ve^2}.
\een 
Then with probability at least 
$1 - 2 \exp\left(-c\ve^2 \frac{m}{K^4}\right) - 4 \exp\left(-c_5 \ve^2 \frac{\tr(A)}{K^4\twonorm{A}}\right)$,
\bens
\twonorm{\inv{m} X X^T - \frac{\tr(A)I_f}{m}- B}
& \le & C_2 \ve \left(\tau_A +  \twonorm{B}\right)
\eens
where $C_2, c_5$ are absolute constants depending on $c', C$, 
where $C > 4 \max(\inv{c c'}, \inv{\sqrt{4 c c'}})$ is a large enough
constant.
\end{theorem}

\begin{corollary}
\label{coro::kronopB}
Suppose all conditions in Theorem~\ref{thm::kronopB} hold.
Suppose
\ben
\label{eq::trAdelta}
\frac{\tr(A)}{\twonorm{A}} & \ge & 
c' f K^4 \frac{C_3^2}{\delta^2}\log\left(\frac{3 C_3}{\delta}\right).
\een
where $C_3 = C_2 \left(\frac{\tau_A}{\lambda_{\min}(B)} \vee 1\right)$ for $C_2$ as in Theorem~\ref{thm::kronopB}.
Then with the probability as stated in Theorem~\ref{thm::kronopB},
\bens
 (1+2\delta) B \succ \frac{X X^T}{m} -\frac{\tr(A)I_f}{m} \succ  (1-2\delta) B \succ 0
\eens
where for the last inequality to hold, we assume that  $\lambda_{\min}(B) > 0$.
\end{corollary}
Next we show a large deviation bound on the
sparse eigenvalues of the error corrected $\tilde{A}$: $\tilde{A} :=
\onef X^T X - \tau_B I_m$.
\begin{corollary}
\label{coro::sparseK}
 Let $X$ be defined as in Definition~\ref{def::subgdata}.
Let $\tilde{A} := \onef X^TX - \tau_B I_{m}$. Suppose
\ben
\label{e::trBlow}
\frac{\tr(B)}{\twonorm{B}} & \ge & c'k K^4 \frac{\log(3e m/k \ve)}{\ve^2}.
\een
Then with probability at least 
$1 -2\exp(-c_4 \ve^2 \frac{f}{K^4}) - 4\exp(-c_4 \ve^2\frac{\tr(B)}{K^4\twonorm{B}})$,
\bens
\rho_{\max}(k, \tilde{A} )
\le  \rho_{\max}(k,A) (1+10\ve) + C_4 \ve \tau_B 
\eens
where $C_4$ is an absolute constant.
Moreover, suppose
for $C_5 = C_4 \left(\frac{\tau_B}{\rho_{\min}(k, A)} \vee 1\right)$ 
\ben
\label{eq::trBlow}
\frac{\tr(B)}{\twonorm{B}} & \ge &
 c'k  K^4 \frac{C_5^2}{\delta^2}\log(\frac{12C_5 e m}{k \delta}).
\een
Then with the probability as stated immediately above, we have
\bens
\rho_{\min}(k, \tilde{A} ) \ge \rho_{\min}(k,A) (1-2\delta).
\eens
\end{corollary}
We prove Theorem~\ref{thm::kronopB} in Section~\ref{sec::proofofgramB}.
We also prove the concentration of measure bounds on error-corrected
gram matrices in Corollaries~\ref{coro::kronopB}
and~\ref{coro::sparseK} in Sections~\ref{sec::kronopB}
and~\ref{sec::sparseK} respectively.

\section{Proof of Lemma~\ref{lemma::REcomp}}
\label{sec::proofoflemmaREcomp}

We define $\W(d_0, k_0)$, where $0 < d_0 < m$
 and $k_0$ is a positive number,
as the set of vectors in $\R^m$ which satisfy the following cone constraint:
\bens
\label{eq::cone-init}
\W(d_0, k_0) = \left\{x \in \R^m \;|\; \exists I \in \{1, \ldots, p\}, \size{I} = d_0
\; \mbox{ s.t. } \; \norm{x_{I^c}}_1 \leq k_0 \norm{x_{I}}_1
\right\}. 
\eens
For each vector $x \in \R^p$, let ${T_0}$ denote the locations of the $s_0$
largest coefficients of $x$ in absolute values. The following elementary estimate~\cite{RZ13} will be used in conjunction with the RE condition.
\begin{lemma}
\label{lemma::lower-bound-Az}
For each vector $x \in \W(s_0, k_0)$, let
${T_0}$ denotes the locations of the $s_0$
largest coefficients of $x$ in absolute values.  Then
\ben
\label{eq::cone-top-norm}
\twonorm{x_{T_{0}}} 
\geq \frac{\twonorm{x}}{\sqrt{1 + k_0}}.
\een
\end{lemma}

\begin{proofof}{Lemma~\ref{lemma::REcomp}}
Part I: Suppose that the Lower-$\RE$ condition holds for $\Gamma := A^T A$.
Let $x \in \W(s_0, k_0)$. Then 
\bens
\onenorm{x} \le (1+k_0) \onenorm{x_{T_0}} \le  (1 + k_0)\sqrt{s_0} \twonorm{x_{T_{0}}}.
\eens
Thus for $x \in \W(s_0, k_0) \cap S^{p-1}$ and  $\tau (1 + k_0)^2 s_0 \le \alpha/2$, 
we have 
\bens
\twonorm{A x} = (x^T A^T A x)^{1/2} & \ge &  
\left(\alpha \twonorm{x}^2 - \tau \onenorm{x}^2\right)^{1/2} \\
& \ge &  
\left(\alpha \twonorm{x}^2 - \tau (1 + k_0)^2 s_0 \twonorm{x_{T_{0}}}^2\right)^{1/2} \\& \ge &  
\left(\alpha - \tau (1 + k_0)^2 s_0 \right)^{1/2}\ge
\sqrt{\frac{\alpha}{2}}.
\eens
Thus the $\RE(s_0, k_0, A)$ condition holds with
\bens
\inv{K(s_0, k_0, A)} & := & 
\min_{x \in \W(s_0, k_0)}\frac{\twonorm{Ax}}{\twonorm{x_{T_0}}}
\ge \sqrt{\frac{\alpha}{2}} 
\eens
where we use the fact that for any $J \in \{1, \ldots, p\}$ such that $\abs{J} \le
s_0$, $\twonorm{x_J} \le \twonorm{x_{T_0}}$.
We now show the other direction.

Part II. 
Assume that $\RE(4R^2, 2R-1, A)$ holds for some integer $R > 1$.
Assume that for some $R  >1$
\bens
\onenorm{x} \le R \twonorm{x}.
\eens
Let $(x_i^*)_{i=1}^p$ be non-increasing arrangement of
$(\abs{x_i})_{i=1}^p$. Then
\bens
\onenorm{x} 
& \le & R \left(\sum_{j=1}^s (x^*_j)^2 +
  \sum_{j=s+1}^{\infty}\left(\frac{\onenorm{x}}{j}\right)^2
\right)^{1/2} \\
& \le & 
R \left(\twonorm{x^*_{J}}^2 + \onenorm{x}^2 \inv{s}\right)^{1/2}  \le R \left(\twonorm{x^*_{J}} + \onenorm{x}  \inv{\sqrt{s}}\right)
\eens
where $J := \{1, \ldots, s\}$.
Choose $s = 4R^2$.   Then
\bens
\onenorm{x} \le R\twonorm{x^*_{J}} + \half \onenorm{x}.
\eens
Thus we have 
\ben
\label{eq::T02bound}
\onenorm{x} & \le & 2 R\twonorm{x^*_{J}} \le  2 R\onenorm{x^*_{J}} \;
\; \text{ and hence } \; \; \\
\onenorm{x^*_{{J}^c}} & \le &  (2 R - 1) \onenorm{x^*_{J}}.
\een
Then $x \in \W(4R^2, 2R-1)$.
Then for all $x \in S^{p-1}$ such that $\onenorm{x} \le R \twonorm{x}$,
we have for $k_0 =2R-1$ and $s_0 := 4R^2$,
\bens
x^T \Gamma x \ge \frac{\twonorm{x_{T_0}}^2 }{K^2(s_0, k_0, A)} 
\ge 
\frac{\twonorm{x}^2}{\sqrt{s_0} K^2(s_0, k_0, A)} =:
\alpha \twonorm{x}^2
\eens
where we use the fact that $(1 + k_0) \twonorm{x_{T_{0}}}^2 \ge
\twonorm{x}^2$ by Lemma~\ref{lemma::lower-bound-Az} with $x_{T_0}$ as
defined therein. 
Otherwise, suppose that $\onenorm{x} \ge R \twonorm{x}$.
Then for a given $\tau > 0$,
\ben
\label{eq::alphatau}
\alpha\twonorm{x}^2 - \tau \onenorm{x}^2 \le
(\inv{\sqrt{s_0}K^2(s_0, k_0, A)} - \tau R^2) \twonorm{x}^2.
\een
Thus we have by the choice of $\tau$ as in~\eqref{eq::tauchoice}  and \eqref{eq::alphatau}
\bens
x^T \Gamma x  \ge  
\lambda_{\min}(\Gamma) \twonorm{x}^2 &\ge & 
(\inv{\sqrt{s_0}K^2(s_0, k_0, A)} - \tau R^2) \twonorm{x}^2 \\
& \ge & 
\alpha \twonorm{x}^2 - \tau \onenorm{x}^2.
\eens
The Lemma thus holds.
\end{proofof}

\section{Proof of Theorem~\ref{thm::lasso}}
\label{sec::proofofthmlasso}
\begin{proofof2}
First we note that it is sufficient to have~\eqref{eq::trBLasso} in order for \eqref{eq::trBlem} to hold. 
~\eqref{eq::trBLasso} guarantees that
for $\V=3 e M_A^3/2$ 
\ben
\nonumber
r(B) := \frac{\tr(B)}{\twonorm{B}} & \ge & 16 c' K^4 \frac{f}{\log m}
\log \frac{\V m \log m }{f}  \\
\nonumber
&\ge &
16 c' K^4 \frac{f}{\log m} 
\log\left(\frac{3 e m M_A^3\log m }{2f}\right) \\
\nonumber
&= & c' K^4 \inv{\ve^2} \frac{4}{M_A^2}\frac{f}{\log m} 
\log\left(\frac{6 e m M_A }{\frac{4}{M_A^2}(f/\log m)}\right) \\
\label{eq::trBlemup}
& \ge & 
c' K^4 \inv{\ve^2} s_0
\log\left(\frac{6 e m M_A }{s_0} \right) = c' K^4\frac{s_0}{\ve^2} \log\left(\frac{3 e m}{s_0 \ve }\right)
\een
where $\ve = \inv{2M_A} \le \inv{128 C}$, and the last inequality
holds given that $k \log (c m/k)$ on the 
RHS of~\eqref{eq::trBlemup} is a 
monotonically increasing function of $k$, and
\bens
 s_0 & \le & \frac{4 f}{M_A^2 \log m} 
\; \text{  and } \;  M_A =  \frac{64 C (\rho_{\max}(s_0,A) +
  \tau_B)}{\lambda_{\min}(A)} \ge 64 C.
\eens
Next we check that the choice of $d$ as in~\eqref{eq::dlasso} ensures
that \eqref{eq::dlassoproof} holds.
Indeed, for $c' K^4 \le 1$, we have
\bens
d  & \le  & C_A (c' K^4 \wedge 1)\frac{\phi f}{\log m}
\le C_A  \left(c' D_{\phi} \wedge  1\right) \frac{f}{\log  m}.
\eens
By Lemma~\ref{lemma::lowerREI}, we have on event $\A_0$, the
modified gram matrix $\hat \Gamma_A :=   \onef (X^T X - \hat\tr(B) I_{m})$
satisfies the Lower $\RE$ conditions with
\ben
\label{eq::localtau}
\text{ curvature}\; \;
\alpha = \half \lambda_{\min}(A) 
\text{ and tolerance } \; \; \tau = \frac{\lambda_{\min}(A)}{2s_0} =
\frac{\alpha}{s_0}.
\een
Theorem~\ref{thm::lasso} follows from Theorem~\ref{thm::main}, so long
as we can show that condition~\eqref{eq::taumain} holds for $\lambda \ge 4 \psi
\sqrt{\frac{\log m}{f}}$ where  the parameter $\psi$ is as
defined~\eqref{eq::psijune}, and $\alpha$ and $\tau =
\frac{\alpha}{s_0}$ are as defined immediately above.
Combining~\eqref{eq::localtau} and~\eqref{eq::taumain},
we need to show~\eqref{eq::dcond} holds.
This is precisely the content of Lemma~\ref{lemma::dmain}. 
This is the end of the proof for Theorem~\ref{thm::lasso}
\end{proofof2}

\section{Proof of Theorem~\ref{thm::DS}}
\label{sec::proofofDSthm}
For the set $\Cone_J(k_0)$ as in~\eqref{eq::cone-init}, 
\bens
\kappa_{\RE}(d_0, k_0)  & := &  \min_{J : \abs{J} \le d_0}
\min_{\Delta \in \Cone_J(k_0)} 
\frac{\abs{\Delta^T \Psi \Delta}}{\twonorm{\Delta_J}^2} 
= \left(\inv{K(d_0, k_0,  (1/\sqrt{f}) Z_1 A^{1/2})}\right)^2.   
\eens
Recall the following Theorem \ref{thm:subgaussian-T-intro} from~\cite{RZ13}.
\begin{theorem}{\textnormal{\cite{RZ13}}}
\label{thm:subgaussian-T-intro}
Set $0< \delta < 1$,  
$k_0 > 0$, and $0< d_0 < p$.
Let $A^{1/2}$ be an $m \times m$ matrix satisfying $\RE(d_0, 3k_0, A^{1/2})$ condition
as in Definition~\ref{def:memory}. Let $d$ be as defined in~\eqref{eq::sparse-dim-A}
\ben
\label{eq::sparse-dim-A}
d & = & d_0 + d_0 \max_j  \twonorm{A^{1/2} e_{j}}^2 \frac{16 K^2(d_0, 3k_0, A^{1/2}) (3k_0)^2 (3k_0 + 1)}{\delta^2}.
\een
Let $\Psi$ be an $n \times m$ matrix whose rows are
independent isotropic $\psi_2$ random vectors in $\R^m$ with constant $\alpha$.
Suppose the sample size satisfies
\ben
\label{eq::UpsilonSampleBound-intro}
n \geq \frac{2000 d \alpha^4}{\delta^2} \log \left(\frac{60 e m}{d \delta}\right).
\een
Then with probability at least $1- 2 \exp(-\delta^2 n/2000 \alpha^4)$,
$\RE(d_0, k_0, (1/\sqrt{n})\Psi A^{1/2})$ condition holds for matrix $(1/\sqrt{n}) \Psi A$
with
\ben
\label{eq::RE-subg}
0< K(d_0, k_0,  (1/\sqrt{n}) \Psi A^{1/2}) \leq \frac{K(d_0, k_0, A^{1/2})}{1-\delta}.
\een
\end{theorem}

\begin{proofof}{Theorem~\ref{thm::DS}}
Suppose $\RE(2d_0, 3k_0, A^{1/2})$ holds.
Then for $d$ as defined in~\eqref{eq::sparse-dim-Ahalf} and $f
= \Omega(d K^4 \log (m/d))$, we have with probability at least $1 - 2
\exp(\delta^2 f/2000 K^4)$, the 
$\RE(2d_0, k_0, \inv{\sqrt{f}} Z_1 A^{1/2})$ condition holds with 
\bens
\kappa_{\RE}(2d_0, k_0)  & = &  \left(\inv{K(2d_0, k_0,  (1/\sqrt{f}) Z_1 A^{1/2})}\right)^2  \ge
\left(\inv{2 K(2d_0, k_0,  A^{1/2})}\right)^2  
\eens
by Theorem~\ref{thm:subgaussian-T-intro}.

The rest of the proof follows from~\cite{BRT14} Theorem 1 and thus we only provide
a sketch. In more details, in view of the lemmas shown in
Section~\ref{sec::proofall}, we need
\bens
\kappa_{q}(d_0, k_0) \ge c d_0^{-1/q}
\eens
to hold for some constant $c$ for $\Psi := \onef X_0^T X_0$.
It is shown in Appendix C in~\cite{BRT14} that under the $\RE(2d_0,
k_0, \inv{\sqrt{f}} Z_1 A^{1/2})$ 
condition, for any $d_0 \le m/2$ and $1\le q\le 2$, we have
\ben
\nonumber
\kappa_{1}(d_0, k_0) & \ge & c d_0^{-1} \kappa_{\RE}(d_0, k_0), \\
\label{eq::kqbound}
\kappa_{q}(d_0, k_0) & \ge & c(q) d_0^{-1/q} \kappa_{\RE}(2d_0, k_0)
\een
where $c(q) > 0$ depends on $k_0$ and $q$.
The theorem is thus proved following exactly the same line of
arguments as in the proof of Theorem 1 in~\cite{BRT14} in view of the
$\ell_q$ sensitivity condition derived immediately above, 
in view of Lemmas~\ref{lemma::DS},~\ref{lemma::DS-cone} and~\ref{lemma::grammatrix}.
Indeed, we have for $v := \hat\beta - \beta^*$, we have by definition
of $\ell_q$ sensitivity as in~\eqref{eq::sense}
\ben
\nonumber
 c(q) d_0^{-1/q} \kappa_{\RE}(2d_0, k_0) \norm{v}_q
&  \le &  
\nonumber
\kappa_{q}(d_0, k_0) \norm{v}_q \le  \norm{\onef X_0^T X_0 v}_{\infty}
\\
\nonumber
& \le & \mu_1 \twonorm{\beta^*} + \mu_2 \onenorm{v} + \tau \\
\nonumber
& \le & \mu_1 \twonorm{\beta^*} + \mu_2 (2+\lambda) \onenorm{v_S} +
\tau \\
\nonumber
& \le & \mu_1 \twonorm{\beta^*} + \mu_2 (2+\lambda) d_0^{1-1/q}
\norm{v_S}_q + \tau \\
\label{eq::prelude}
& \le & \mu_1 \twonorm{\beta^*} + \mu_2 (2+\lambda) d_0^{1-1/q} \norm{v}_q + \tau.
\een
Thus we have for $d_0 = c_0 \sqrt{f/\log m}$ where $c_0$ is sufficiently small,
\bens
&& d_0^{-1/q} 
(c(q) \kappa_{\RE}(2d_0, k_0)-  \mu_2 (2+\lambda)
  d_0) \norm{v}_q
\le \mu_1 \twonorm{\beta^*} + \tau \\
\text{ hence }  && 
\norm{v}_q  \le  C( 4 D_2  r_{m,f} K \twonorm{\beta^*} + 2 D_0 M_{\e}
r_{m,f} ) d_0^{1/q} \\
\nonumber
&  & \quad \quad \quad \le  4 C D_2  r_{m,f} (K \twonorm{\beta^*} + M_{\e}) d_0^{1/q}
\eens
for some constant $C = 1/\left(c(q) \kappa_{\RE}(2d_0, k_0)-  \mu_2 (2+\lambda)
  d_0\right) \ge 1/\left(2 c(q) \kappa_{\RE}(2d_0, k_0)\right)$ given that 
$$\mu_2(2+\lambda) d_0= 2 D_2 K r_{m,f}(\inv{\lambda} + 1)
(2+\lambda) c_0 \sqrt{f/\log m}=  2 c_0 C_0 D_2 K^2 (2+\lambda)
(\inv{\lambda} + 1)$$ 
is sufficiently small 
 and thus \eqref{eq::ellqnorm} holds.
The prediction error bound follows exactly the same line of arguments
as in~\cite{BRT14} which we omit here. 
See proof of Theorem~\ref{thm::DSoracle} in Section~\ref{sec::DSoracleproof} for details.
\end{proofof}

\section{Proof of Theorem~\ref{thm::lassora}}
\label{sec::classoproof}
\begin{proofof2}
The proof is identical to the proof of Theorem~\ref{thm::lasso} up till
~\eqref{eq::localtau}, except that we replace the condition on $d$ 
as in the theorem statement by~\eqref{eq::doracle}: that is,
\bens
&& d:= \abs{\supp(\beta^*)} \le C_A \frac{f}{\log m} \left\{c' C_\phi
  \wedge 2 \right\} \;\;\text{ where }  C_A := \inv{128 M_{A}^2}, \\
\nonumber
&& C_{\phi} := \frac{\twonorm{B} + a_{\max}}{D^2}
\left(\frac{K^2M^2_{\e}}{b_0^2}   +  \tau_B^+ K^4 \phi \right) \ge
\frac{\twonorm{B} + a_{\max}}{D^2} \tau_B^+ 
\eens
where $c', \phi, b_0, M_{\e}$ and $K$ are as defined in
Theorem~\ref{thm::lasso}, where we assume that
$ b_0^2 \ge \twonorm{\beta^*}^2 \ge \phi b_0^2  \; \text{ for some } \; 0 < \phi
\le 1$.
Theorem~\ref{thm::lassora} 
follows from Theorem~\ref{thm::main}, so long
as we can show that condition~\eqref{eq::taumain} holds for $\lambda \ge 2 \psi
\sqrt{\frac{\log m}{f}}$ where the parameter $\psi$ is as
defined~\eqref{eq::psioracle}, and $\alpha$ and $\tau = \frac{\alpha}{s_0}$ are as defined 
in~\eqref{eq::localtau}.
Combining~\eqref{eq::localtau} and~\eqref{eq::taumain},
we need to show~\eqref{eq::dcond} holds.
This is precisely the content of Lemma~\ref{lemma::dmainoracle}.
This is the end of the proof for Theorem~\ref{thm::lassora}.
\end{proofof2}

\section{Proof of Theorem~\ref{thm::DSoracle}}
\begin{proofof2}
\label{sec::DSoraproof}
Throughout this proof, we assume that $\B_0 \cap \B_{10}$ holds. 
The rest of the proof follows that of Theorem~\ref{thm::DS}, except for the last
part.
Let $\mu_1, \mu_2, \tau$ be as defined in
Lemma~\ref{lemma::grammatrix}. 
We have for $\mu_2  :=  2 \mu (1+\inv{2\lambda}) $ where  $\mu = D_0'
K r_{m,f}  \tilde\tau_B^{1/2}$,
and $d_0 =  c_0  \tau_B^-\sqrt{f/\log m}$, 
\ben
\label{eq::half}
\mu_2(2+\lambda) d_0 & = &
 2 C_0 D_0' K^2  \tilde\tau_B^{1/2} 
(\inv{2\lambda} + 1)
(2+\lambda)  c_0  \tau_B^- \\
\nonumber
&  \le &  2 c_0 C_0 D_0' K^2 (2+\lambda) (\inv{2\lambda} + 1)  
\le \half c(q) \kappa_{\RE}(2d_0, k_0) 
\een
which holds when $c_0$ is sufficiently small,  where by
\eqref{eq::tildetauBbound} $\tau_B^{-}   \tilde\tau_B^{1/2} \le 1$. Hence
\bens
\mu_2  d_0 \le \frac{c(q) \kappa_{\RE}(2d_0, k_0)}{2(2+\lambda)}
\eens
Thus for $c_0$ sufficiently small, $\mu_1 = 2 \mu$,
by \eqref{eq::kqbound}, \eqref{eq::half}, \eqref{eq::prelude} and~\eqref{eq::tildetauB},
\ben
\nonumber
\lefteqn{ d_0^{-1/q} \half (c(q) \kappa_{\RE}(2d_0, k_0)) \norm{v}_q }\\
& =  & 
\nonumber
d_0^{-1/q} (c(q) \kappa_{\RE}(2d_0, k_0)-  \mu_2 (2+\lambda)  d_0) \norm{v}_q \\
\nonumber
&\le & (\kappa_{q}(d_0, k_0) - \mu_2 (2+\lambda) d_0^{1-1/q})
\norm{v}_q \le \mu_1 \twonorm{\beta^*} + \tau  \\
\label{eq::middle}
& \le & 
 2 D_0' r_{m,f} K^2 ((\tau_B^{1/2} + (3/2)C_{6} K r_{m,m}^{1/2})\twonorm{\beta^*} + M_{\e}/K)
\een
and thus \eqref{eq::ellqnormimp} holds, following the proof in Theorem~\ref{thm::DS}.
The prediction error bound follows exactly the same line of arguments
as in~\cite{BRT14}, which we now include for the sake completeness.
Following~\eqref{eq::ellqnormimp}, we have by \eqref{eq::middle}, 
\bens
\onenorm{v} & \le &  
C_{11} d_0 (\mu_1 \twonorm{\beta^*} + \tau ) \; \text{ where } \;
C_{11} = 2/\left(c(q) \kappa_{\RE}(2d_0, k_0)\right) \\
\text{ and hence } \; 
\mu_2 \onenorm{v} & \le & C_{11} \mu_2 d_0 (\mu_1 \twonorm{\beta^*} + \tau) \\
& \le & 
C_{11} \inv{2(2+\lambda)}  \left(c(q) \kappa_{\RE}(2d_0, k_0)\right) (\mu_1\twonorm{\beta^*} + \tau) \\
& = & 
\inv{2  + \lambda} (\mu_1\twonorm{\beta^*} + \tau) 
\eens
Thus we have by \eqref{eq::middle}, the bounds immediately above, and~\eqref{eq::tildetauBbound}
\bens
\onef \twonorm{X (\hat{\beta} -\beta^*)}^2 & \le & 
\onenorm{v} \norm{\onef X_0^T X_0 v}_{\infty} \\
& \le &  
C_{11}  d_0(\mu_1 \twonorm{\beta^*}+ \tau)  \left( \mu_1 \twonorm{\beta^*}+
  \mu_2 \onenorm{v} + 2 \tau \right) \\ 
& \le &  
C_{11} d_0(\mu_1 \twonorm{\beta^*}+ \tau)(1+\inv{2+\lambda})  \left( \mu_1 \twonorm{\beta^*}+ 2 \tau \right) \\ 
& = &  
C' (D_0')^2  K^4 d_0 \frac{\log m}{f} \left( \tilde\tau_B^{1/2} \twonorm{\beta^*} +
  \frac{M_{\e}}{K}\right)^2\\
& \le &  
C'' (\twonorm{B} + a_{\max}) 
K^2 d_0 \frac{\log m}{f} \left( (2\tau_B+ 3 C_6^2 K^2 r_{m,m})
  K^2 \twonorm{\beta^*}^2 + M_{\e}^2\right)
\eens
 where $(D_0')^2 \le 2\twonorm{B} + 2a_{\max}$.
The theorem is thus proved.
\end{proofof2}

\section{Conclusion}
\label{sec::conclude}
In view of the main Theorems~\ref{thm::lasso} and~\ref{thm::DS}, at
this point, we do not really think one estimator is preferable to the other.
While the rates we obtain for both estimators are at the same order
 for $q=1, 2$, the conditions under which these  rates are obtained
 are somewhat different.  Lasso estimator allows large values of
 sparsity,  while Conic-programming estimator conceptually is more
 adaptive by not fixing an upper bound on $\twonorm{\beta^*}$ a priori,
 the cost of which seems to be a more stringent requirement on the
 sparsity level. The lasso-type procedure can recover a sparse model using 
$O(\log m)$ number of measurements per nonzero component 
despite the measurement error in $X$ and the stochastic noise $\e$
while the Dantzig selector-type allows only $d \asymp \sqrt{f/\log
  m}$ to achieve the error rate at the same order as the Lasso-type estimator.

However, we show in Theorem~\ref{thm::DSoracle} in Section
\ref{sec::DSoracle} that this restriction on the sparsity can be
relaxed for the Conic programming estimator~\eqref{eq::Conic}, 
when we make a different choice for the parameter $\mu$ based on a
more refined analysis.  Eventually, as $\tau_B \to 0$, this relaxation
on $d$ as in \eqref{eq::ora-sparsity-rem} enables the Conic
Programming estimator to achieve bounds which are essentially
identical to  the Dantzig Selector when the design matrix $X_0$ is a
subgaussian random matrix satisfying the Restricted Eigenvalue conditions;
See for example~\cite{CT07,BRT09,RZ13}.
For the Lasso estimator, when we require that  the stochastic error $\e$ in the response variable $y$ as in~\eqref{eq::oby}
does not converge to $0$ as quickly as the measurement error $W$ in
\eqref{eq::obX} does, then the sparsity constraint becomes essentially
unchanged as $\tau_B \to 0$.
These tradeoffs are somehow different from the behavior of the Conic
programming estimator versus the Lasso estimator; however, we believe the
differences are minor.

We now state a slightly sharper bound than those in Lemma~\ref{lemma::D2improv} which
provides a significant improvement on the error bounds in case $\tau_B =o(1)$ while
$\twonorm{A} \ge 1$ for the Lasso-type estimator in \eqref{eq::origin}
as well as the  Conic programming estimator~\eqref{eq::Conic}.
Recall $D'_0 :=\sqrt{\twonorm{B}} + a_{\max}^{1/2}$.
By~\eqref{eq::oracle},
\bens
\label{eq::oracleII}
\norm{\hat\gamma - \hat\Gamma \beta^*}_{\infty}
& \le &
D_0' K  \tau_B^{1/2} \twonorm{\beta^*} r_{m,f}
+ \frac{2D_1 K}{\sqrt{m}} \norm{\beta^*}_{\infty} r_{m,f} + D_0 M_{\e} r_{m,f}
\eens
When $\tau_B \to 0$, we have for $D_0 =\sqrt{\tau_B} + a_{\max}^{1/2} \to 
a_{\max}^{1/2}$
\bens
\norm{\hat\gamma - \hat\Gamma \beta^*}_{\infty}
& = &  
O\left(D_1 K\inv{\sqrt{m}}\norm{\beta^*}_{\infty} +  D_0 K M_{\e}\right)
K \sqrt{\frac{\log m}{f}}
\eens
where $D_1 = \frac{\fnorm{A}}{\sqrt{m}} +  \frac{\fnorm{B}}{\sqrt{f}} \to \twonorm{A}^{1/2}$  under (A1), given that $\fnorm{B}/
\sqrt{f} \le \tau_B^{1/2} \twonorm{B}^{1/2} \to 0$,  and the
  first term inside the bracket comes from the estimation error in
  $\hat\tr(B)/f$, which can be made go away if we were to assume  that
  $\tr(B)$ is also known.
In this case,  the error term involving $\twonorm{\beta^*}$ in~\eqref{eq::psijune} vanishes, and we only need
to set 
\ben
\label{eq::psijune15lasso}
&& \lambda \ge 2 \psi \sqrt{\frac{\log m}{f}} \; \; \text{ where } \;\;
\psi  \asymp D_0 K M_{\e} + \twonorm{A}^{1/2} K^2 m^{-1/2} \norm{\beta^*}_{\infty}.
\een
Moreover, suppose that $\tr(B)$ is given, then one can drop the second
term  in $\psi$  as in \eqref{eq::psijune15lasso} and hence recover
the lasso bound when the design matrix $X$  is assumed to be free of measurement errors.

Finally, we note that the bounds corresponding to the Upper $\RE$ 
condition as stated in Corollary~\ref{coro::BC}, Theorem~\ref{thm::AD}
and Lemma~\ref{lemma::lowerREI} are  not needed for
Theorem~\ref{thm::lasso}. 
They are useful to ensure algorithmic convergence and to bound the optimization error for the gradient
descent-type of algorithms as considered in~\cite{LW12}, when one is
interested in approximately solving the non-convex optimization
function~\eqref{eq::origin}.  Our numerical results validate such
algorithmic and statistical convergence properties. 

\section*{Acknowledgements}
The authors are grateful for the helpful discussions with Prof. Rob
Kass.

\appendix

\section{Outline}
In Sections~\ref{sec::aux} and~\ref{sec::stoc}, we 
present variations of  the Hanson-Wright inequality
as recently derived in~\cite{RV13} (cf. Lemma~\ref{lemma::oneeventA}),
concentration of measure bounds and  stochastic error bounds in
Lemma~\ref{lemma::Tclaim1}.

In Sections~\ref{sec::lassoall} and~\ref{sec::proofofDS}, we prove the
technical lemmas for Theorems~\ref{thm::lasso} and~\ref{thm::DS} respectively. 
In Section~\ref{sec::DSoracleproof}, we prove the Lemmas needed for Proof of Theorem~\ref{thm::DSoracle}.
In order to prove Corollary~\ref{coro::BC},
we need to first state some geometric analysis results Section~\ref{sec::geometry}.
We prove Corollary~\ref{coro::BC} in Section~\ref{sec::appendLURE} and 
Theorem~\ref{thm::AD} in Section~\ref{sec::proofofthmAD}.
Results presented in Section~\ref{sec::Best} are proved in
Section~\ref{sec::proofofgramB}.
In particular, we prove Theorem~\ref{thm::kronopB} in Section~\ref{sec::proofofgramB}.
We also prove the concentration of measure bounds on error-corrected
gram matrices in Corollaries~\ref{coro::kronopB}
and~\ref{coro::sparseK} in Sections~\ref{sec::kronopB}
and~\ref{sec::sparseK} respectively.
The results appearing in Section~\ref{sec::proofofgramB} are proved in Section~\ref{sec::proofofnormA}.

\section{Some auxiliary results}
\label{sec::aux}
We first need to state the following form of the Hanson-Wright
inequality as recently derived in Rudelson and Vershynin~\cite{RV13},
and an auxiliary result in Lemma~\ref{lemma::oneeventA} which may be
of independent interests.
\begin{theorem}
\label{thm::HW}
Let $X = (X_1, \ldots, X_m) \in \R^m$ be a random vector with independent components $X_i$ which satisfy
$\expct{X_i} = 0$ and $\norm{X_i}_{\psi_2} \leq K$. Let $A$ be an $m \times m$ matrix. Then, for every $t > 0$,
\bens
\prob{\abs{X^T A X - \expct{X^T A X} } > t} 
\leq 
2 \exp \left[- c\min\left(\frac{t^2}{K^4 \fnorm{A}^2}, \frac{t}{K^2 \twonorm{A}} \right)\right].
\eens
\end{theorem}
We note that following the proof of Theorem~\ref{thm::HW}, it
is clear that the following holds:
Let $X = (X_1, \ldots, X_m) \in \R^m$ be a random vector 
as defined in Theorem~\ref{thm::HW}.
Let $Y, Y'$ be  independent copies of $X$. Let $A$ be an $m \times m$ matrix. Then, for every $t > 0$,
\ben
\label{eq::HWdecoupled} 
\prob{\abs{Y^T A Y'} > t} 
\leq 
2 \exp \left[- c\min\left(\frac{t^2}{K^4 \fnorm{A}^2}, \frac{t}{K^2 \twonorm{A}} \right)\right].
\een
We next need to state Lemma~\ref{lemma::oneeventA}, which we prove in Section~\ref{sec::proofoftensorA}.
\begin{lemma}
\label{lemma::oneeventA}
Let $u, w \in S^{f-1}$. 
Let $A \succ 0$ be a $m \times m$ symmetric positive definite matrix.
Let $Z$ be an $f \times m$
random matrix with independent entries $Z_{ij}$ satisfying
$\E Z_{ij} = 0$ and  $\norm{Z_{ij}}_{\psi_2} \leq K$.
Let $Z_1, Z_2$ be independent copies of $Z$.
Then for every $t > 0$,
\bens
\prob{\abs{u^T Z_1 A^{1/2} Z_2^T w} >  t}
& \le & 
2 \exp\left(-c\min\left(\frac{t^2}{K^4\tr(A)}, \frac{t}{K^2 \twonorm{A}^{1/2}}\right)\right), \\
\prob{\abs{u^T Z A Z^T w - \E u^T Z A Z^T w } >  t}
& \le & 
2 \exp\left(-c\min\left(\frac{t^2}{K^4\fnorm{A}^2}, \frac{t}{K^2 \twonorm{A}}\right)\right)
\eens
where $c$ is the same constant as defined in Theorem~\ref{thm::HW}.
\end{lemma}

\subsection{Proof of Lemma~\ref{lemma::oneeventA}} 
\label{sec::proofoftensorA}
Lemma~\ref{lemma::Auv} is a well-known fact.
\begin{lemma}
\label{lemma::Auv}
Let $A_{u w} := (u \otimes w) \otimes A \; \text{where } \; u, w \in
\Sp^{p-1}$ where $p \ge 2$. Then $\twonorm{A_{uw}}  \le \twonorm{A} \; \text{ and } \;
\fnorm{A_{uw}}  \le \fnorm{A}.$
\end{lemma}

\begin{proofof}{Lemma~\ref{lemma::oneeventA}}
\silent{
Let $x_1, \ldots, x_m, x'_1, \ldots, x'_m \in \R^f$ be the column vectors
$Z_1, Z_2$ respectively.
We can re-write the quadratic forms as follows:
\bens
u^T Z_1^T B^{1/2} Z_2 w
& = & \sum_{i,j=1, m} u_i w_j x_i B^{1/2} x'_j  \\
& = & \mvec{Z_1}^T \big((u \otimes w) \otimes B^{1/2} \big)\mvec{Z_2} \\
& =: & \mvec{Z_1}^T B_{uw}^{1/2} \mvec{Z_2} \\
u^T Z^T B Z w
& = & \mvec{Z}^T \big((u \otimes w) \otimes B \big)\mvec{Z}
=:\mvec{Z}^T B_{uw} \mvec{Z}
\eens
where clearly by independence of $Z_1, Z_2$,
\bens
\E \mvec{Z_1}^T \big((u \otimes w) \otimes B^{1/2}\big) \mvec{Z_2} & = & 0 \\
\E \mvec{Z}^T \big((u \otimes u) \otimes B\big) \mvec{Z} & = & 
\tr\big((u \otimes u) \otimes B\big) = \tr(B).
\eens
Thus we invoke Theorem~\ref{thm::HW},~\eqref{eq::HWdecoupled} and Lemma~\ref{lemma::Auv} to  conclude
\bens
\prob{\abs{u^T Z_1^T B^{1/2} Z_2 w} >  t}
& \le & 
2 \exp\left(-c\min\left(\frac{t^2}{K^4\fnorm{B^{1/2}_{uw}}^2},
 \frac{t}{K^2 \twonorm{B_{uw}^{1/2}}}\right)\right)  \\
& \le & 
2 \exp\left(-c\min\left(\frac{t^2}{K^4\tr(B)},
 \frac{t}{K^2 \twonorm{B^{1/2}}}\right)\right)   \\
\prob{\abs{u^T Z^T B Z w - \E u^T Z^T B Z w } >  t}
& \le & 
2 \exp\left(-c\min\left(\frac{t^2}{K^4\fnorm{B_{uw}}^2},
 \frac{t}{K^2 \twonorm{B_{uw}}}\right)\right)  \\
& \le & 
2 \exp\left(-c\min\left(\frac{t^2}{K^4\fnorm{B}^2},
 \frac{t}{K^2 \twonorm{B}}\right)\right).
\eens}
Let $z_1, \ldots, z_f, z'_1, \ldots, z'_f \in \R^m$ be the row vectors
$Z_1, Z_2$ respectively.
Notice that we can write the quadratic form as follows:
\bens
u^T Z_1 A^{1/2} Z_2^T w
& = &
\sum_{i,j=1, m} u_i w_j z_i A^{1/2} z'_j  \\
& = & \mvec{Z_1^T}^T \big((u \otimes w) \otimes A^{1/2}
\big)\mvec{Z_2^T} \\
& =: & \mvec{Z_1^T}^T A_{uw}^{1/2} \mvec{Z_2^T}, \\
u^T Z A Z^T w
& = & \mvec{Z^T}^T \big((u \otimes w) \otimes A \big)\mvec{Z^T} \\
& =: & \mvec{Z^T}^T A_{uw} \mvec{Z^T}
\eens
where clearly by independence of $Z_1, Z_2$,
\bens
\E \mvec{Z_1^T}^T \big((u \otimes w) \otimes A^{1/2}\big) \mvec{Z_2^T}
& = & 0, \; \; \text{ and }  \\
\E \mvec{Z}^T \big((u \otimes u) \otimes A\big) \mvec{Z} & = & 
\tr\big((u \otimes u) \otimes A\big) = \tr(A).
\eens
Thus we invoke~\eqref{eq::HWdecoupled} and
Lemma~\ref{lemma::Auv} to show the concentration bounds on event
$\{\abs{u^T Z_1 A^{1/2} Z_2^T w} >  t\}$:
\bens
\prob{
 \abs{u^T Z_1 A^{1/2} Z_2^T w} >  t}
& \le & 
2 \exp\left(-\min\left(\frac{t^2}{K^4\fnorm{A^{1/2}_{uw}}^2},
 \frac{t}{K^2 \twonorm{A_{uw}^{1/2}}}\right)\right)  \\
& \le &
2 \exp\left(-\min\left(\frac{t^2}{K^4\tr(A)},
 \frac{t}{K^2 \twonorm{A^{1/2}}}\right)\right).
\eens
Similarly, we have  by Theorem~\ref{thm::HW} and Lemma~\ref{lemma::Auv},
\bens
\lefteqn{
\prob{\abs{u^T Z A Z^T w - \E u^T Z A Z^T w } >  t}}\\
& \le & 
2 \exp\left(-c\min\left(\frac{t^2}{K^4\fnorm{A_{uw}}^2},
 \frac{t}{K^2 \twonorm{A_{uw}}}\right)\right)  \\
& \le & 
2 \exp\left(-c\min\left(\frac{t^2}{K^4\fnorm{A}^2},
 \frac{t}{K^2 \twonorm{A}}\right)\right).
\eens
The Lemma thus holds.
\end{proofof}

\subsection{Stochastic error terms}
\label{sec::stoc}
The following large deviation bounds in
Lemmas~\ref{lemma::Tclaim1} and \ref{lemma::trBest}
are the key results in proving Lemmas~\ref{lemma::low-noise}
and~\ref{lemma::grammatrix}.
Let $C_0$ satisfy \eqref{eq::defineC0} for $c$ 
as defined in Theorem~\ref{thm::HW}.
Throughout this section, we denote by:
\bens
r_{m,f} =  C_0 K \sqrt{\frac{\log m}{f}} \; \; \text{ and } \; \
r_{m,m} =  2 C_0 \sqrt{\frac{\log m}{mf}}.
\eens
We also define some events $\B_4, \B_5, \B_6, \B_{10}$;
Denote by $\B_0 := B_4 \cap \B_5 \cap \B_6$, which we use throughout
this paper. 
\begin{lemma}
\label{lemma::Tclaim1}
Assume that the stable rank of $B$, $\fnorm{B}^2/\twonorm{B}^2 \ge \log m$.
Let $Z, X_0$ and $W$ as defined in Theorem~\ref{thm::lasso}.
Let $Z_0, Z_1$ and  $Z_2$ be independent copies of $Z$.
Let $\e^T \sim Y  M_{\e}/K$ where $Y := e_1^T Z_0^T$.
Let  $\tau_B = \frac{\tr(B)}{f}$. 
Denote by $\B_4$ the event such that 
\bens
\onef \norm{A^{\half} Z_1^T \e}_{\infty}
& \le & r_{m,f} M_{\e} a_{\max}^{1/2} \\
 \; \text{ and } \; 
\onef \norm{ Z_2^T B^{\half} \e}_{\infty}
& \le &   r_{m,f} M_{\e}\sqrt{\tau_B}.
 \eens
Then $\prob{\B_4} \ge  1 - 4/m^3$.
Moreover, denote by $\B_5$ the  event such that 
\bens
\onef \norm{(Z^T B Z- \tr(B) I_{m}) \beta^*}_{\infty} & \le &  
r_{m,f} K  \twonorm{\beta^*}  \frac{\fnorm{B}}{\sqrt{f}} \\
\text{and} \;\;\onef \norm{X_0^T W \beta^*}_{\infty}
& \le &   r_{m,f} K  \twonorm{\beta^*}  \sqrt{\tau_B} a^{1/2}_{\max}.
\eens
Then $\prob{\B_5} \ge 1 -  4 /m^3$.

Finally, denote by $\B_{10}$ the event such that
\bens
\onef \norm{(Z^T B Z- \tr(B) I_{m})}_{\max} & \le &  r_{m,f} K \frac{\fnorm{B}}{\sqrt{f}} \\
\text{and} \;\; \onef \norm{X_0^T W}_{\max}
& \le &  r_{m,f} K  \twonorm{\beta^*} \sqrt{\tau_B} a^{1/2}_{\max}.
\eens
Then $\prob{\B_{10}} \ge 1 -  4/m^2$.
\end{lemma}

We prove Lemmas~\ref{lemma::Tclaim1} in Section~\ref{sec::proofoferrors}.

\subsection{Stochastic error bounds}s
\label{sec::proofoferrors}
Following Lemma~\ref{lemma::oneeventA}, we have for all $t >0$, $B
\succ 0$ being an $f \times f$ symmetric positive definite matrix, and 
$v, w \in \R^m$
\ben
\label{eq::crossB}
\; \; \; 
\prob{\abs{v^T Z_1^T B^{1/2} Z_2 w} >  t}
& \le & 
2 \exp\left[-c\min\left(\frac{t^2}{K^4\tr(B)},
 \frac{t}{K^2 \twonorm{B}^{1/2}}\right) \right] \\
\nonumber
\prob{\abs{v^T Z^T B Z w - \E v^T Z^T B Z w } >  t}
&  \le & 
2 \exp\left(-c\min\left(\frac{t^2}{K^4\fnorm{B}^2},
 \frac{t}{K^2 \twonorm{B}}\right)\right).
\een
\subsection{Proof for Lemma~\ref{lemma::Tclaim1}}
\begin{proofof2}
Let $e_1, \ldots, e_m \in \R^m$ be the canonical basis spanning
$\R^m$.
Let $x_1, \ldots, x_m, x'_1, \ldots, x'_m \in \R^f$ be the column vectors
$Z_1, Z_2$ respectively. Let $Y \sim e_1^T Z_0^T$. Let $w_i =
\frac{A^{1/2}  e_i}{\twonorm{A^{1/2} e_i}}$ for all $i$.
Clearly the condition on the stable rank of $B$ guarantees that 
$$f \ge r(B) = \frac{\tr(B)}{\twonorm{B}} =
\frac{\tr(B)\twonorm{B}}{\twonorm{B}^2}\ge \fnorm{B}^2 /\twonorm{B}^2 
 \ge \log m.$$
By~\eqref{eq::HWdecoupled}, we obtain for $t' = C_0 M_{\e} K
\sqrt{\tr(B)\log m}$ and $t = C_0 K^2 \sqrt{\log m} \tr(B)^{1/2}$:
\bens
\lefteqn{
\prob{\exists j, \abs{\e^T B^{1/2} Z_2 e_j} >  t'}= } \\
&& 
\prob{\exists j, \frac{M_{\e}}{K}\abs{e_1^T Z_0^T B^{1/2} Z_2 e_j} > 
C_0 M_{\e} K \sqrt{\log m}\tr(B)^{\half}} \\
& \le &
\exp(\log m)\prob{\abs{Y^T B^{1/2} x'_j} > C_0 K^2
  \sqrt{\log m}\tr(B)^{\half} } 
\le  2/m^3
\eens
where the last inequality holds by the union bound, 
given that $\frac{\tr(B)}{\twonorm{B}} \ge \log m$, and for all $j$
\bens 
\prob{\abs{Y^T B^{1/2} x'_j} >  t} & \le &
2 \exp\left(-c\min\left(\frac{t^2}{K^4 \tr(B)},
 \frac{t}{K^2 \twonorm{B}^{1/2}}\right)\right), \\
& \le & 2 \exp\left(-c\min\left(C_0^2 \log m,
\frac{C_0 \log^{1/2} m \sqrt{\tr(B)}}{\twonorm{B}^{1/2}} \right)\right) \\
& \le &  2 \exp\left(-c \min(C_0^2, C_0) \log m\right)  \le 
2 \exp\left(-4 \log m\right).
\eens
Let $v, w \in S^{m-1}$. 
Thus we have by Lemma~\ref{lemma::oneeventA}, 
for $t_0 = C_0 M_{\e} K \sqrt{f \log m}$ and $\tau= C_0 K^2 \sqrt{f \log m}$, 
$w_j= \frac{A^{1/2}  e_j}{\twonorm{A^{1/2} e_j}}$ and $f \ge \log m$,
\bens
\lefteqn{
\prob{\exists j, \abs{\e^T Z_1 w_j} >  t_0}
\le  \prob{\exists j, \frac{M_{\e}}{K}\abs{Y^T Z_1 w_j} > C_0 M_{\e} K \sqrt{f \log m} } }\\
& \le &
m \prob{\abs{Y^T  Z_1 w_j} > C_0 K^2 \sqrt{f \log m} } \\ 
& = & \exp(\log m)\prob{\abs{e_1^T Z_0^T Z_1 w_j} > \tau} 
 \le  2 \exp\left(-c\min\left(\frac{\tau^2}{K^4 f},\frac{\tau}{K^2}\right)\right), \\
& \le & 2 \exp\left(-c\min\left(\frac{(C_0 K^2 \sqrt{f \log m})^2}{K^4 f},
 \frac{  C_0  K^2 \sqrt{f \log m}}{K^2}\right)+ \log m\right) \\
 & \le &
2 m \exp\left(-c\min\left(C_0^2 \log m, C_0 \log^{1/2} m \sqrt{f}
\right)\right) \\
& \le & 
2 m \exp\left(-c \min(C_0^2, C_0) \log m\right)  \le  2 \exp\left(-3\log m\right).
\eens
Therefore we have  with probability at least $1 - 4/m^3$,
\bens
\label{eq::eventB4a}
\norm{ Z_2^T B^{\half} \e}_{\infty}
& := & \max_{j=1, \ldots, m} \ip{\e^T B^{1/2} Z_2, e_j}  \le t' = C_0
M_{\e} K \sqrt{\tr(B)\log m}  \\
\norm{A^{\half} Z_1^T \e}_{\infty}
& := & \max_{j=1, \ldots, m} \ip{A^{1/2} e_j, Z_1^T  \e} \le 
\max_{j=1, \ldots, m} \twonorm{A^{1/2} e_j}
\max_{j=1, \ldots, m}  \ip{w_j, Z_1^T \e} \\
& \le & 
a_{\max}^{1/2} t_0 = a_{\max}^{1/2}  C_0 M_{\e} K \sqrt{f \log m}.
\eens
The ``moreover'' part follows exactly the same arguments as above.
Denote by $\bar\beta^* := \beta^*/\twonorm{\beta^*} \in E \cap
S^{m-1}$ and $w_i := A^{1/2} e_i/\twonorm{ A^{1/2} e_i}$.
By~\eqref{eq::crossB}
\bens
\lefteqn{\prob{\exists i,  \ip{w_i, Z_1^T B^{1/2} Z_2 \bar \beta^*}
    \ge  
C_0 K^2 \sqrt{\log m } {\tr(B)^{1/2}} }} \\
 & \le &
\sum_{i=1}^m \prob{\ip{w_i, Z_1^T B^{1/2} Z_2 \bar\beta^*} \ge  C_0 K^2
\sqrt{\log m \tr(B)}} \\
& \le & 
2 \exp\left(-c\min\left(C_0^2\log m,C_0 \log m\right)+ \log m\right) \le  2/ m^3.
\eens
Now for $t = C_0 K^2 \sqrt{\log m} \fnorm{B}$, and  $\fnorm{B}/\twonorm{B} \ge \sqrt{\log m}$,
\bens
\lefteqn{\prob{\exists e_i:  \ip{e_i, (Z^T B Z- \tr(B) I_{m}) \bar \beta^*} 
 \ge  C_0 K^2 \sqrt{\log m} \fnorm{B}}} \\
& \le & 
2m\exp \left[- c\min\left(\frac{t^2}{K^4  \fnorm{B}^2}, 
\frac{t}{K^2 \twonorm{B}} \right)\right] 
\le  2/ m^3.
\eens
By the two inequalities immediately above, we have with probability at
least $1 -  4 /m^3$, 
\bens
\lefteqn{\norm{X_0^T W \beta^*}_{\infty} 
= \norm{A^{1/2} Z_1^T B^{1/2} Z_2 \beta^*}_{\infty}  } \\
& \le & \twonorm{\beta^*} \max_{e_i}\twonorm{A^{1/2} e_i}
 \left(\sup_{w_i} \ip{w_i, Z_1^T B^{1/2} Z_2 \bar\beta^*}\right)  \\
& \le &  C_0 K^2 \twonorm{\beta^*}\sqrt{\log m} a_{\max}^{1/2}  \sqrt{\tr(B)}
\eens
and 
\bens
\lefteqn{
\norm{(Z^T B Z- \tr(B) I_{m}) \beta^*}_{\infty}
 =  \norm{(Z^T B Z- \tr(B) I_{m}) \bar \beta^*}_{\infty} \twonorm{\beta^*} }\\
&=  & \twonorm{\beta^*} \left(\sup_{e_i} \ip{e_i, (Z^T B Z- \tr(B)
    I_{m}) \bar \beta^*}\right)  \\
& \le &  C_0 K^2 \twonorm{\beta^*}\sqrt{\log m} \fnorm{B}.
\eens
The last two bounds follow exactly the same arguments as above, except
that we replace $\beta^*$ with $e_j, j=1, \ldots, m$ and apply the
union bounds to $m^2$ events instead of $m$, and thus 
$\prob{\B_{10}} \ge 1 -  4 /m^2$, 
\end{proofof2}

\section{Proofs for the Lasso-type estimator}
\label{sec::lassoall}

\subsection{Proof of Lemma~\ref{lemma::low-noise}}
\label{sec::lownoiseproof}
\begin{proofof2}
Clearly the condition on the stable rank of $B$ guarantees that 
$$f \ge r(B) = \frac{\tr(B)}{\twonorm{B}} =
\frac{\tr(B)\twonorm{B}}{\twonorm{B}^2}\ge \fnorm{B}^2 /\twonorm{B}^2 
 \ge \log m.$$
Thus the conditions in
Lemmas~\ref{lemma::Tclaim1} and \ref{lemma::trBest} hold.
First notice that 
\bens
\hat\gamma & = & 
\inv{f} \left(X_0^T X_0 \beta^*
 +  W^T X_0 \beta^* + X_0^T \e +  W^T \e\right) \\
(\inv{f} X^T X - \frac{\hat\tr(B)}{f} I_{m})\beta^* 
& = & 
\inv{f} (X_0^T X_0 + W^T X_0 + X_0^T W 
+ W^T W - \frac{\hat\tr(B)}{f} I_{m})\beta^*
\eens
Thus 
\bens
\norm{\hat\gamma - \hat\Gamma \beta^*}_{\infty} 
& \le &
\norm{\hat\gamma - \onef \left(X^T X - \hat\tr(B) I_{m}\right)\beta^*}_{\infty} \\
 & = & 
\inv{f}\norm{ X_0^T \e +  W^T \e- \left(W^T W + X_0^T W - \hat\tr(B) I_m\right)\beta^* }_{\infty} \\
 & \le & 
\inv{f}\norm{ X_0^T \e +  W^T \e}_{\infty}  +\inv{f}\norm{(W^T W-
  \hat\tr(B) I_{m}) \beta^*}_{\infty} 
+\norm{\inv{f} X_0^T W \beta^*}_{\infty}  \\
 & \le & 
 \onef \norm{ X_0^T \e +  W^T \e}_{\infty} + \onef(\norm{(Z^T B
     Z- \tr(B) I_{m}) \beta^*}_{\infty}) + \onef \norm{X_0^T W \beta^*}_{\infty} \\
&& + \onef \abs{\hat\tr(B) - \tr(B)} \norm{\beta^*}_{\infty} =:  U_1+  U_2 + U_3 + U_4 
\eens
By Lemma~\ref{lemma::Tclaim1} we have  on $\B_4$ for $D_0 := \sqrt{\tau_B} + a_{\max}^{1/2}$,
\bens
U_1 = \onef \norm{ X_0^T \e +  W^T \e}_{\infty}   = 
\onef\norm{A^{\half} Z_1^T \e +  Z_2^T B^{\half} \e}_{\infty}  \le
r_{m,f}M_{\e} D_0
\eens
and on event $\B_5$ for  $D'_0 :=\sqrt{\twonorm{B}} + a_{\max}^{1/2}$,
\bens
\lefteqn{U_2 +  U_3 = 
\onef\norm{(Z^T B Z- \tr(B) I_{m}) \beta^*}_{\infty} +
\onef \norm{X_0^T W \beta^*}_{\infty}} \\
& \le & r_{m,f} K \twonorm{\beta^*}
\left(\frac{\fnorm{B}}{\sqrt{f}} + \sqrt{\tau_B}
  a^{1/2}_{\max}\right) \le  K r_{m,f} \twonorm{\beta^*}  \tau_B^{1/2} D_0'
\eens
where recall $\fnorm{B} \le \sqrt{\tr(B)}\twonorm{B}^{1/2}$.
Denote by $\B_0 := \B_4 \cap \B_5 \cap \B_6$.
We have on $\B_0$ and under (A1),  by Lemmas~\ref{lemma::Tclaim1}
and~\ref{lemma::trBest} and $D_1$ defined therein,
\ben
\nonumber
\norm{\hat\gamma - \hat\Gamma \beta^*}_{\infty} 
& \le &
U_1+  U_2 + U_3 + U_4 \\
& \le & 
\nonumber
 r_{m,f} M_{\e} D_0 + D_0' \tau_B^{1/2} K r_{m,f} \twonorm{\beta^*} 
+ \onef \abs{\hat\tr(B) - \tr(B)} \norm{\beta^*}_{\infty} \\
& \le & 
\label{eq::oracleone}
D_0 M_{\e} r_{m,f} + D_0' K  \tau_B^{1/2} \twonorm{\beta^*} r_{m,f}  +
D_1 K^2 \norm{\beta^*}_{\infty} r_{m,m} \\
& \le & 
\label{eq::oracle}
D_0 M_{\e} r_{m,f} + D_0' K  \tau_B^{1/2} \twonorm{\beta^*} r_{m,f}
+ 2 D_1 K \inv{\sqrt{m}} \norm{\beta^*}_{\infty} r_{m,f} \\
& \le &
\nonumber 
r_{m,f} 
\left(\left(\frac{3}{4}D_2 + D_2 \inv{\sqrt{m}}\right)
 K \twonorm{\beta^*} + D_0  M_{\e} \right)
\een
where $2D_1 \le 2 \twonorm{A} +  2 \twonorm{B} = D_2$, for 
$(D'_0)^2 \le 2 \twonorm{B} + 2 a_{\max}$
\bens
D_0 \le D'_0  & \le &  \sqrt{2(\twonorm{B} + a_{\max})}
 \le 2(a_{\max} + \twonorm{B}) = D_2, \\
\text{ and } \; \; 
D'_0 \tau_{B}^{1/2} & \le &  (\twonorm{B}^{1/2} + a_{\max}^{1/2})\tau_{B}^{1/2} 
 \le \tau_B + \half(\twonorm{B} + a_{\max})  \le \frac{3}{4} D_2
\eens
given that under (A1) : $\tau_A = 1$, $\twonorm{A} \ge a_{\max} \ge a_{\max}^{1/2} \ge 1$.
Hence the lemma holds for $m \ge 16$ and $\psi = C_0 D_2 K \left(K \twonorm{\beta^*} + M_{\e} \right)$.
Finally, we have by the union bound, $\prob{\B_0} \ge 1 -  16/m^3$.
\end{proofof2}

\subsection{Proof of Lemma~\ref{lemma::trBest}}
\begin{proofof2}
First we write
\bens
 X X^T - \tr(A) I_{f}
& = & 
\big( Z_1 A^{1/2} + B^{1/2} Z_2) \big(Z_1 A^{1/2} +B^{1/2} Z_2\big)^T - \tr(A) I_{f} \\ 
& = & 
\big( Z_1 A^{1/2} + B^{1/2} Z_2) \big(Z_2^T B^{1/2} +A^{1/2} Z_1^T\big) 
- \tr(A) I_{f} \\
& = & 
 Z_1 A^{1/2} Z_2^T B^{1/2} + B^{1/2} Z_2  Z_2^T B^{1/2} \\
&& + B^{1/2} Z_2  A^{1/2} Z_1^T+
 Z_1 A Z_1^T- \tr(A) I_{f}.
\eens
Thus we have for $\check\tr(B)  := \onem \big(\fnorm{X}^2 -f \tr(A)\big)$
\bens
\lefteqn{\onef(\check\tr(B) - \tr(B)) := 
\inv{mf}\big(\fnorm{X}^2 -f \tr(A) -m \tr(B) \big) } \\
& = & \inv{mf} (\tr(XX^T) -   f\tr(A) - m\tr(B))\\
& = &  
\frac{2}{mf} 
\tr( Z_1 A^{1/2} Z_2^T B^{1/2} ) + 
\left(\frac{\tr(B^{1/2} Z_2  Z_2^T B^{1/2})}{mf} -\frac{
    \tr(B)}{f}\right) \\
&&
+ \frac{\tr( Z_1 A Z_1^T)}{mf}- \frac{\tr(A)}{m}
\eens
By constructing a new matrix $A_f = I_f \otimes A$ which is block diagonal with $f$
identical submatrices $A$ along its diagonal, we prove the following
large deviation bound: for $t_1 = C_0 K^2 \fnorm{A} \sqrt{f \log m }$
and $f > \log m$,
\bens
\lefteqn{\prob{\abs{\tr(Z_1 A Z_1^T)- f \tr(A)} \ge t_1} = 
\prob{\abs{\mvec{Z_1}^T (I \otimes A) \mvec{Z_1}- f \tr(A)} \ge t_1}
}\\
& \le &  \exp\left(-c\min\left(\frac{t_1^2}{K^4 \fnorm{A_f}^2},
    \frac{t_1}{K^2 \twonorm{A_f}}\right)\right) \\
& \le & 2\exp\left(-c\min\left(\frac{ (C_0 K^2 \sqrt{f \log m}
      \fnorm{A})^2}{K^4 f \fnorm{A}^2},
\frac{C_0 K^2 \sqrt{f \log m} \fnorm{A}}{K^2 \twonorm{A}}\right)\right) \\
& \le &  2 \exp\left(-4 \log m\right)
\eens
where the first inequality holds by Theorem~\ref{thm::HW} and the
second inequality holds given that $\fnorm{A_f}^2 = f \fnorm{A}$ and $\twonorm{A_f}^2 = \twonorm{A}$.
Similarly, by constructing a new matrix $B_m = I_m \otimes B$ which is block diagonal with $m$
identical submatrices $B$ along its diagonal, we prove the following
large deviation bound: for $t_2 =C_0  K^2\fnorm{B} \sqrt{m \log m}$
and $m \ge 2$,
\bens
\lefteqn{\prob{\abs{\tr(Z_2^T B Z_2)- m \tr(B)} \ge t_2} = 
\prob{\abs{\mvec{Z_2}^T (I_m \otimes B) \mvec{Z_2}- m \tr(B)} \ge t_2}}\\
& \le &  \exp\left(-c\min\left(\frac{t_2^2}{K^4 m \fnorm{B}^2}, \frac{t_2}{K^2 \twonorm{B}}\right)\right) \\
& \le & 2\exp\left(-c\min\left(\frac{ (C_0 K^2 \sqrt{m \log m}  \fnorm{B})^2}{K^4 m\fnorm{B}^2},
\frac{C_0 K^2 \sqrt{m \log m} \fnorm{B}}{K^2 \twonorm{B}}\right)\right) \\
& \le &  2 \exp\left(-4 \log m\right).
\eens
Finally, we have by~\eqref{eq::HWdecoupled} for $t_0 = C_0 K^2 \sqrt{\tr(A) \tr(B) \log m}$, 
\bens
\lefteqn{\prob{ \abs{\mvec{Z_1}^T B^{1/2} \otimes
    A^{1/2} \mvec{Z_2}} >t_0}} \\
& \le & 2\exp\left(-c\min\left(\frac{t_0^2}{K^4 
\fnorm{B^{1/2} \otimes A^{1/2}}^2},\frac{t_0}{K^2 \twonorm{B^{1/2}\otimes A^{1/2}}}\right)\right)\\
& = & 
2\exp\left(-c\min\left(\frac{(C_0 \sqrt{\tr(A) \tr(B) \log
        m})^2}{\tr(A) \tr(B)}, 
\frac{C_0\sqrt{\tr(A) \tr(B) \log m} }{\twonorm{B}^{1/2} \twonorm{A}^{1/2}}\right)\right)\\
& \le & 2 \exp(-4 \log m)
\eens
where we used and the fact that $r(A) r(B) \ge \log m$,
$\twonorm{B^{1/2} \otimes A^{1/2}} = \twonorm{B}^{1/2} \twonorm{A}^{1/2}$ and
\bens
\fnorm{B^{1/2} \otimes A^{1/2}}^2 & = & 
\tr((B^{1/2} \otimes A^{1/2})(B^{1/2} \otimes A^{1/2}))
=\tr(B \otimes A) = \tr(A)\tr(B).
\eens
Thus we have with probability $1- 6/m^4$,
\bens
\lefteqn{\onef\abs{\check\tr(B) - \tr(B)}
= \inv{mf} \abs{\tr(XX^T) -   f\tr(A) - m\tr(B)}}\\
& \le &
\frac{2}{mf} \abs{\mvec{Z_1}^T (B^{1/2} \otimes A^{1/2})\mvec{Z_2}} \\
&& + \abs{\frac{\tr(Z_2^T B Z_2)}{mf} - \frac{\tr(B)}{f}}
+ \abs{\frac{\tr(Z_1 A Z_1^T)}{mf}- \frac{\tr(A)}{m }}\\
& \le &    \inv{mf} (2t_0 + t_1 + t_2 )  =\frac{\sqrt{\log m}}{\sqrt{mf}}  C_0 K^2 \left(\frac{\fnorm{A}  }{\sqrt{m} }+ 2\sqrt{\tau_A \tau_B} +
  \frac{\fnorm{B}}{\sqrt{f}}\right) \\
& \le & 2 C_0 \frac{\sqrt{\log m}}{\sqrt{mf}}  K^2 D_1 =: D_1 K^2 r_{m,m} 
\eens
where recall  $ r_{m,m} = 2 C_0 \frac{\sqrt{\log m}}{\sqrt{mf}}$,
$D_1=\frac{\fnorm{A}  }{\sqrt{m} }+ \frac{\fnorm{B}}{\sqrt{f}}$, and
\bens
 2\sqrt{\tau_A \tau_B} \le  \tau_A + \tau_B \le \frac{\fnorm{A}
 }{\sqrt{m} }+  \frac{\fnorm{B}}{\sqrt{f}}.
\eens
To see this, recall
\ben
\label{eq::tracefnorm}
m \tau_A & = & \sum_{i=1}^m\lambda_{i}(A) \le \sqrt{m} (\sum_{i=1}^m
\lambda^2_{i}(A))^{1/2} = \sqrt{m} \fnorm{A} \\
\nonumber
f \tau_B & = & \sum_{i=1}^f \lambda_{i}(B) \le \sqrt{f} (\sum_{i=1}^f
\lambda^2_{i}(B))^{1/2} = \sqrt{f} \fnorm{B} 
\een
where $\lambda_{i}(A), i=1, \ldots, m$ and $\lambda_{i}(B), i=1,
\ldots, f$  denote the eigenvalues of positive semidefinite covariance
matrices $A$ and $B$ respectively.

Denote  by $\B_6$ the following event  
$$\left\{\onef\abs{\check\tr(B) - \tr(B)}  \le D_1
  K^2 r_{m,m}\right\}$$
Clearly $\hat\tr(B) := (\check\tr(B))_+$ by definition~\eqref{eq::trBest}.
As a consequence, on $\B_6$, 
 $\hat\tr(B) = \check\tr(B) > 0$ when   $\tau_B > D_1  K^2 r_{m,m}$;  hence
\bens
\onef\abs{\hat\tr(B) - \tr(B)} = \onef\abs{\check\tr(B) - \tr(B)} \le
D_1  K^2 r_{m,m}.
\eens
Otherwise,  it is possible that $\check\tr(B) < 0$.
However, suppose we set  
$$\hat \tau_B : = 
\onef \hat\tr(B) :=\onef (\check\tr(B) \vee 0),$$ 
then we can also guarantee that 
\bens
\abs{\hat\tau_B - \tau_B} = \abs{\tau_B} \le D_1  K^2 r_{m,m} \; \; 
\text{ in case } \; \; \tau_B \le D_1  K^2 r_{m,m}.
\eens
The lemma is thus proved.
\end{proofof2}

\section{Proof of Theorem~\ref{thm::main}}
\label{sec::proofofmain}
Denote by  $\beta = \beta^*$. Let $S := \supp{\beta}$, $d = \size{S}$ and 
$$\upsilon = \hat{\beta} - \beta.$$
where $\hat\beta$ is as defined in \eqref{eq::origin}.
We first show Lemma~\ref{lemma:magic-number}, followed by the 
proof of Theorem~\ref{thm::main}.

\begin{lemma}{\textnormal~\cite{BRT09,LW12}}
\label{lemma:magic-number}
Suppose that~\eqref{eq::psimain} holds.
Suppose that there exists a parameter $\psi$ such that 
\bens
\sqrt{d} \tau \le \frac{\psi}{b_0} \sqrt{\frac{\log m}{f}}, \quad \text{ and } \quad
\lambda \geq 4 \psi \sqrt{\frac{\log m}{f}}
\eens
where $b_0, \lambda$ are as defined in \eqref{eq::origin}. Then 
$\norm{\upsilon_{S^c}}_1 \leq 3 \norm{\upsilon_{S}}_1.$
\end{lemma}
\begin{proof}
By the optimality of $\hat{\beta}$, we have
\begin{eqnarray*}
\lambda_{n} \norm{\beta}_1 - 
\lambda_{n} \norm{\hat\beta}_1 
& \geq & 
\inv{2} \hat\beta \hat\Gamma  \hat\beta - \inv{2} \beta \hat\Gamma \beta -
\ip{\hat\gamma, v} \\
& = & 
\inv{2} \up \hat\Gamma \up +\ip{\up, \hat\Gamma \beta} -
\ip{\up, \hat\gamma} \\
& = & 
\inv{2} \up \hat\Gamma \up -\ip{\up, \hat\gamma - \hat\Gamma \beta} 
\end{eqnarray*}
Hence,
we have for $\lambda \geq 4 \psi \sqrt{\frac{\log m}{f}}$,
\begin{eqnarray}
\label{eq::precondition}
\half \up \hat\Gamma \up 
& \leq &
\ip{\up, \hat\gamma - \hat\Gamma \beta} +
 \lambda_{n} \left(\norm{\beta}_1-  \norm{\hat\beta}_1 \right)\\
\nonumber
& \leq & \lambda_{n}
\left(\norm{\beta}_1-  \norm{\hat\beta}_1\right) +
\norm{\hat\gamma - \hat\Gamma \beta}_{\infty} \norm{\upsilon}_1 
\end{eqnarray}
Hence
\begin{eqnarray}
\label{eq::upperbound}
\up \hat\Gamma \up 
& \leq & 
\lambda_{n} \left(2 \norm{\beta}_1- 2 \norm{\hat\beta}_1\right) +
2 \psi \sqrt{\frac{\log m}{f}} \norm{\upsilon}_1 \\
& \leq & 
\nonumber
\lambda_{n} \left(2\norm{\beta}_1 - 2\norm{\hat\beta}_1 + \half
  \norm{\upsilon}_1\right) \\
\label{eq::finalupperbound}
& \leq & \lambda_{n} \half \left(5 \onenorm{\upsilon_{S}} - 3\onenorm{\upsilon_{S^c}}\right).
\end{eqnarray}
where by the triangle inequality, and $\beta_{\Sc} = 0$, we have
\begin{eqnarray} 
\nonumber
2 \onenorm{\beta} -  2 \onenorm{\hat\beta} + \half \onenorm{\upsilon}
& = &
2 \onenorm{\beta_S} - 2 \onenorm{\hat\beta_{S}} -
2 \onenorm{\upsilon_{\Sc}} + \half\onenorm{\upsilon_S} +\half \onenorm{\upsilon_{S^c}} \\
& \leq &
\nonumber
2 \norm{\up_{S}}_1 -2 \norm{\up_{\Sc}}_1 + \half \norm{\up_S}_1 
+ \half \norm{\up_{S^c}}_1 \\ 
& \leq & 
\label{eq::magic-number-2}
 \half \left(5 \onenorm{\upsilon_{S}} - 3\onenorm{\upsilon_{S^c}}\right).
\end{eqnarray}
We now give a lower bound on the LHS of~\eqref{eq::precondition},
applying the lower-$\RE$ condition as in Definition~\ref{def::lowRE},
\ben
\nonumber
\up^T \hat\Gamma \up
& \ge &
\alpha \twonorm{\up}^2 - \tau \onenorm{\up}^2
\ge  - \tau \onenorm{\up}^2\\
\text{ thus } \;
- \up^T \hat\Gamma \up
& \le & 
\nonumber
\onenorm{\up}^2 \tau \le \onenorm{\up} 2 b_0 \sqrt{d} \tau\\
& \le &
\nonumber
 \onenorm{\up} 2 b_0 
\frac{\psi}{b_0} \sqrt{\frac{\log m}{f}}
= \onenorm{\up} 2 \psi \sqrt{\frac{\log m}{f}}\\
\label{eq::magic-number-neg}
& \le & \half \lambda (\onenorm{\up_S} + \onenorm{\up_{\Sc}})
\een
where we use the assumption that 
\bens
\sqrt{d} \tau \le \frac{\psi}{b_0} \sqrt{\frac{\log m}{f}}, 
  \quad \text{ and } \;
\onenorm{\up} \le \onenorm{\hat{\beta}} + \onenorm{\beta}
\le 2 b_0 \sqrt{d}
\eens
which holds by the triangle inequality and the fact that both 
$\hat{\beta}$ and $\beta$ have $\ell_1$ norm being bounded by $b_0 \sqrt{d}$.
Hence by~\eqref{eq::finalupperbound}
and~\eqref{eq::magic-number-neg}
\ben
\label{eq::magic-number}
0 
& \le &
 - \up \hat\Gamma \up + \frac{5}{2} \lambda \onenorm{\upsilon_{S}} -
\frac{3}{2} 
\lambda \onenorm{\upsilon_{S^c}} \\
\nonumber
& \le &
\half \lambda \onenorm{\upsilon_{S}} + \half \lambda \onenorm{\upsilon_{S^c}}
+ \frac{5}{2}\lambda \onenorm{\upsilon_{S}} - \frac{3}{2} \lambda \onenorm{\upsilon_{S^c}} \\
& \le &
3 \lambda \onenorm{\upsilon_{S}} - \lambda \onenorm{\upsilon_{S^c}}
\een
Thus we have 
\bens
 \onenorm{\upsilon_{S^c}} \le 3 \onenorm{\upsilon_{S}}
\eens
Thus Lemma~\ref{lemma:magic-number} holds.
\end{proof}

\begin{proofof}{Theorem~\ref{thm::main}}
Following the conclusion of Lemma~\ref{lemma:magic-number}, we have
\ben
\label{eq::onenorm}
\onenorm{\upsilon} \le 4 \onenorm{\upsilon_{S}} \le 4 \sqrt{d} \twonorm{\upsilon}.
\een
Moreover, we have 
by the lower-$\RE$ condition as in
Definition~\ref{def::lowRE}
\ben
\label{eq::prelow}
\up^T \hat\Gamma \up
& \ge & 
\alpha \twonorm{\up}^2 - \tau \onenorm{\up}^2 \ge 
(\alpha  - 16 d \tau) \twonorm{\up}^2 \ge \half \alpha \twonorm{\up}^2
\een
where the  last inequality follows from the assumption that 
$16 d \tau \le \alpha/2$.

Combining the bounds in   \eqref{eq::prelow}, \eqref{eq::onenorm} and \eqref{eq::upperbound},
 we have
\bens
\half \alpha  \twonorm{\upsilon}^2 
&\le &
\up^T \hat\Gamma \up \le 
\lambda_{n} \left(2 \norm{\beta}_1- 2 \norm{\hat\beta}_1\right) +
2 \psi \sqrt{\frac{\log m}{f}} \norm{\upsilon}_1 \\
&\le &
\frac{5}{2} \lambda  \norm{\upsilon_S}_1 
\le 10 \lambda \sqrt{d} \twonorm{\up}
\eens
And thus we have $\twonorm{\up} \le 20 \lambda \sqrt{d}$.
The theorem is thus proved. 
\end{proofof}

\subsection{Proof of Lemma~\ref{lemma::lowerREI}}
\label{sec::records}
\begin{proofof2}
In view of Remark~\ref{rem::error-bound}, Condition \eqref{eq::trBlem}
implies that  \eqref{eq::trB} in Theorem~\ref{thm::AD} holds for 
$\zeta = s_0$ and $\ve =\inv{2 M_A}$.
Now, by Theorem~\ref{thm::AD}, we have $\forall u, v \in E \cap
S^{m-1}$, under (A1) and (A3), condition \eqref{eq::Deltacond} holds
under event $\A_0$, and so long as  
$m f \ge 1024 C_0^2 D_2^2 K^4 \log m/\lambda_{\min}(A)^2$, 
\bens
\abs{u^T \Delta v} & \le&
  8C \vp(s_0) \ve + 2 C_0 D_2 K^2\sqrt{\frac{\log m}{mf}} =: \delta  \text{ with } \; 
\delta \le \inv{8}   \lambda_{\min}(A)  \le \inv{8} \; \; \\
&& \text{ which holds for all } \; \;
\ve  \le \half \frac{\lambda_{\min}(A) }{64 C \vp(s_0)} 
:= \inv{2M_A}  \le \inv{128 C}
 \eens
with $\prob{\A_0} \ge 1- 4 \exp\left(-c_2\ve^2 \frac{\tr(B)}{K^4\twonorm{B}}\right)-2
  \exp\left(-c_2\ve^2 \frac{f}{K^4}\right) - 6 /m^3$. 
Hence, by Corollary~\ref{coro::BC}, 
$\forall  \theta \in \R^m$,
\bens
\theta^T \hat\Gamma_A \theta \ge \alpha \twonorm{\theta}^2 - \tau
\onenorm{\theta}^2\; \; \text{ and } \; \; 
\theta^T \hat\Gamma_A \theta \le \bar\alpha \twonorm{\theta}^2 + \tau
\onenorm{\theta}^2
\eens
where $\alpha = \half \lambda_{\min}(A)$ and $\bar\alpha = \frac{3}{2}
\lambda_{\max}(A)$ and 
\bens
\lefteqn{
\frac{512 C^2 \vp(s_0)^2}{\lambda_{\min}(A)}\frac{\log m}{f}
 \le  \tau =  \frac{\alpha}{s_0} \le \frac{2\alpha}{s_0+1}} \\
& \le &\frac{1024 C^2 \vp^2(s_0+1)}{\lambda_{\min}(A)}\frac{\log
  m}{f}.
\eens 
where we plugged in $s_0$ as defined in~\eqref{eq::s0cond}.
The lemma is thus proved in view of Remark~\ref{rem::error-bound}.
\end{proofof2}

\begin{remark}
\label{rem::error-bound}
Clearly the condition on $\tr(B)/\twonorm{B}$ as stated in
Lemma~\ref{lemma::lowerREI} ensures that we have
for $\ve = \inv{2M_A}$ and $s_0 \asymp \frac{4 f}{M_A^2 \log m}$
\bens
\ve^2 \frac{\tr(B)}{K^4\twonorm{B}} 
& \ge & \frac{\ve^2}{K^4} c' K^4 \frac{s_0}{\ve^2} \log\left(\frac{3e m}{s_0 \ve}\right) \\
& \ge & \inv{4 M_A^2 K^4} 4c' K^4 M_A^2 s_0 \log\left(\frac{6e mM_A}{s_0}\right) \\
& \ge & c' s_0 \log\left(\frac{6e m M_A}{s_0}\right) 
\eens
and hence
\bens
\exp\left(-c_2\ve^2 \frac{\tr(B)}{K^4\twonorm{B}}\right)
& \le &
\exp\left(-c' c_2 s_0 \log\left(\frac{6e m M_A}{s_0}\right)\right) \\
& \asymp &
\exp\left(-c_3\frac{4 f}{M_A^2 \log m} \log\left(\frac{3e M_A^3 m \log m}{2f}\right)\right) 
\eens
\end{remark}

\subsection{Proof of Lemma~\ref{lemma::dmain}}
\label{sec::dmain}
\begin{proofof2}
Let
\bens 
M_{+} & = & \frac{64 C \vp(s_0+1)}{\lambda_{\min}(A)}  \text{ where }
\; \; \vp(s_0+1) = \rho_{\max}(s_0+1,A) + \tau_B =: D
\eens
By definition of $s_0$, we have 
\bens
\sqrt{s_0+1} \vp(s_0+1) & \ge & \frac{\lambda _{\min}(A)}{32
  C}\sqrt{\frac{f }{\log m}}  \; \; \text{ and hence}\\
s_0+ 1& \ge & \frac {\lambda^2 _{\min}(A)}{1024 C^2\vp^2(s_0+1)}
\frac{f}{\log m}=
\left(\frac{\alpha}{16 C D} \right)^2 \frac{f}{\log  m} 
\ge\inv{M_A^2}\frac{f}{\log m} 
\eens
The first inequality in~\eqref{eq::taumain} holds given that $M_{+} \le 2 M_{A}$ and hence
\bens
d \le
\inv{64 M_{A}^2}\frac{f}{\log m} &\le &  
\inv{16 M_{+}^2}\frac{f}{\log m}  \le \frac{s_0 + 1}{64} \le \frac{s_0}{32}
\eens

Moreover, for $D= \rho_{\max}(s_0+1,A) +\tau_B \le D_2$ and  $C = C_0
/\sqrt{c'}$, we have
\bens 
d  & \le  &  C_A  c' D_{\phi}\frac{f}{\log m}  \le
\inv{128M_A^2}
\left(\frac{C_0 D_2}{C D}\right)^2 D_{\phi}  \frac{f}{\log m} \\
& \le & 
\half \left(\inv{16 C D}\right)^2  
4 C_0^2 D_2^2 D_{\phi}  \frac{f}{M_A^2 \log m} \\
& \le  & 
\half \frac{(s_0 +1)^2}{\alpha^2} \frac{\log m}{f}
\left(\frac{\psi}{b_0}\right)^2 \le \frac{(s_0)^2}{\alpha^2} \frac{\log m}{f}
\left(\frac{\psi}{b_0}\right)^2
\eens 
where assuming that $s_0 \ge 3$, we have
\ben
\nonumber
\frac{2 s_0^2}{\alpha^2} & \ge & \left(\frac{s_0+1}{\alpha}\right)^2 
  \ge \frac{\alpha^2}{(16 C D )^4} \left(\frac{f}{\log
      m}\right)^2  \\
\label{eq::dphicondition}
\left(\frac{\psi}{b_0}\right)^2 
& = &
 4 C_0^2 D_2^2  \frac{K^2}{b_0^2} \left(M_{\e}+
  K\twonorm{\beta^*} \right)^2  \\
\nonumber
& \ge & 
 4  C_0^2 D_2^2 D_{\phi} = 
4 C_0^2 D_2^2\left(\frac{K^2M^2_{\e}}{b_0^2}+  K^4 \phi \right).
\een 
We have shown that \eqref{eq::dcond} indeed holds, and the lemma is
thus proved.
\end{proofof2}

\begin{remark}
Throughout this paper, we assume that $C_0$ is a large enough constant
such that  for $c$ as defined in Theorem~\ref{thm::HW},
\ben\label{eq::defineC0}
c \min\{C_0^2, C_0\} \ge 4.
\een 
By definition of $s_0$, we have for $\vp^2(s_0) \ge 1$,
\bens
s_0  \vp^2(s_0) & \le &  \frac{c'\lambda^2_{\min}(A)}{1024
  C_0^2}\frac{f }{\log m} \; \; \text{ and hence}\\
s_0 & \le & 
\frac{c'\lambda^2_{\min}(A)}{1024
  C_0^2}\frac{f }{\log m} \le \frac{\lambda^2_{\min}(A)}{1024
  C_0^2}\frac{f }{\log m} =: \check{s}_0.
\eens
\end{remark}

\begin{remark}
The proof shows that one can take $C = C_0/\sqrt{c'}$, and take
\bens
\V = 3 e M_A^3/2 =
  \frac{3 e 64^3 C^3 \vp^3(s_0)}{2\lambda^3_{\min}(A)} \le 
  \frac{3 e 64^3 C_0^3 \vp^3(\check{s}_0)}
{2 (c')^{3/2}\lambda^3_{\min}(A)}.
\eens
Hence a sufficient condition on $r(B)$ is:
\ben
\label{eq::trBLassorem}
r(B) \ge 16c' K^4 \frac{f}{\log m}
\left(3\log\frac{ 64 C_0 \vp(\check{s}_0)}{\sqrt{c'}
\lambda_{\min}(A)}  + \log \frac{3 e m \log m }{2f} \right).
\een
\end{remark}

It remains to prove  Lemmas~\ref{lemma::D2improv} and~\ref{lemma::dmainoracle}.

\begin{proofof}{Lemma~\ref{lemma::D2improv}}
Suppose that event $\B_0$ holds.
By~\eqref{eq::oracle} and that fact that 
$2 D_1 := 2(\frac{\fnorm{A}  }{\sqrt{m}} +\frac{\fnorm{B}  }{\sqrt{f}}
)\le  2(\twonorm{A}^{1/2} + \twonorm{B}^{1/2}) (\sqrt{\tau_A} +
\sqrt{\tau_B}) \le 
D_{\ora} D_0'$, where recall $D_0'=\twonorm{B}^{1/2} + a_{\max}^{1/2}$,
\bens
\label{eq::oracleII}
\norm{\hat\gamma - \hat\Gamma \beta^*}_{\infty}
& \le &
D_0' K  \tau_B^{1/2} \twonorm{\beta^*} r_{m,f}
+ 2D_1 K \inv{\sqrt{m}} \norm{\beta^*}_{\infty} r_{m,f} + D_0 M_{\e} r_{m,f} \\
& \le &
D_0' K \twonorm{\beta^*} r_{m,f}
\left( \tau_B^{1/2}  +  \frac{D_{\ora}}{\sqrt{m}} \right) + D_0  M_{\e} r_{m,f}  \\
& \le &
D_0' \left( \tau_B^{1/2}  +  \frac{D_{\ora}}{\sqrt{m}} \right) K \twonorm{\beta^*} r_{m,f}
+ D_0  M_{\e} r_{m,f} 
\eens
The lemma is thus proved.
\end{proofof}

\begin{proofof}{Lemma~\ref{lemma::dmainoracle}}
Recall that we require 
\bens
d & \le &  
 C_A \left\{c' C_{\phi} \wedge 2 \right\}
\frac{f}{\log m} \; \text{ where }\; \; C_{\phi} =
\frac{\twonorm{B} + a_{\max} }{D^2} D_{\phi} \\
&& \; \text{ where }  \; C_A = \inv{128 M_A^2}\; \; \; \text{ and } \;
\; b_0^2 \ge \twonorm{\beta^*}^2 \ge \phi b_0^2.
\eens
The proof for $d \le s_0/32$ follows exactly that of
Lemma~\ref{lemma::dmain}.
In order to show the second inequality, we follow the same line of
arguments where we need to replace one inequality. By definition of $D_0'$, we have
$ \twonorm{B} + a_{\max} \le (D_0')^2 \le 2(\twonorm{B} + a_{\max} )$.
Now suppose that for $C_{\phi} = 
\frac{\twonorm{B} + a_{\max} }{D^2} D_{\phi}$
\bens
d & :=  &  C_A c'  C_{\phi}  \frac{f}{\log m} \le 
 C_A  \frac{f}{\log m}\left(\frac{C_0 D_0'}{C D}\right)^2 D_{\phi}
\eens
where $1 \le D= \rho_{\max}(s_0+1,A) +\tau_B \le D_2$ and  $C = C_0
/\sqrt{c'}$. 
\bens
d  & \le  &  C_A  c' C_{\phi}\frac{f}{\log m}  \le
\inv{128M_A^2}
\left(\frac{C_0 D_0'}{C D}\right)^2 D_{\phi}  \frac{f}{\log m} \\
& \le & 
\half \left(\inv{16 C D}\right)^2 
4 C_0^2 (D_0')^2 D_{\phi}  \frac{f}{M_A^2 \log m} \\
& \le  & 
\half \frac{(s_0 +1)^2}{\alpha^2} \frac{\log m}{f}
\left(\frac{\psi}{b_0}\right)^2 \le \frac{(s_0)^2}{\alpha^2} \frac{\log m}{f}
\left(\frac{\psi}{b_0}\right)^2
\eens
where assuming that $s_0 \ge 3$, we have the following
 inequality by
definition of $s_0$ and $\alpha = \lambda_{\min}(A)/2$ 
\bens
\nonumber
\frac{2 s_0^2}{\alpha^2} & \ge & \left(\frac{s_0+1}{\alpha}\right)^2 
  \ge \frac{\alpha^2}{(16 C D )^4} \left(\frac{f}{\log
      m}\right)^2  
\eens
which is identical in the proof of Lemma~\ref{lemma::dmain}, while  
we replace~\eqref{eq::dphicondition} with 
\bens
4  C_0^2 (D_0')^2 D_{\phi}
& = & 
4  C_0^2 (D_0')^2 (\frac{K^2M^2_{\e}}{b_0^2}+  \tau_B^+ K^4 \phi)\\
& \le  & 
4 C_0^2 (D_0')^2  \frac{K^2}{b_0^2} \left(M_{\e}+ \tau_B^{+/2}  K\twonorm{\beta^*} \right)^2  \le 
\left(\frac{\psi}{b_0}\right)^2
\eens
where $D_{\phi} :=\frac{K^2M^2_{\e}}{b_0^2}+  \tau_B^+ K^4 \phi$ and 
$\psi =2  C_0  \left(D_0' K^2 (\tau_B^{1/2} + \frac{D_{\ora}}{\sqrt{m}} )\twonorm{\beta^*} + D_0 M_{\e} K\right)$ as
in~\eqref{eq::psioracle}.
\end{proofof}

\section{Proofs for the Conic Programming estimator}
\label{sec::proofofDS}
\subsection{Proof of Lemmas~\ref{lemma::DS} and~\ref{lemma::grammatrix}}
\label{sec::proofofDSlemma}
We next provide proofs for Lemmas~\ref{lemma::DS} and~\ref{lemma::grammatrix} in this section.

\begin{proofof}{Lemma~\ref{lemma::DS}}
Suppose event $\B_0$ holds.
Then by the proof of Lemma~\ref{lemma::low-noise},
\bens
\norm{\onef X^T(y - X \beta^*) + \onef \hat\tr(B) \beta^*}_{\infty} 
& = & \norm{\hat\gamma - \hat\Gamma  \beta^*}_{\infty}  \\
& \le & 
2 C_0 D_2 K^2\twonorm{\beta^*}\sqrt{\frac{\log m}{f}}
+ C_0 D_0 K M_{\e} \sqrt{\frac{\log m}{f}} \\
& =:& 
\mu \twonorm{\beta^*} + \tau
\eens
The lemma follows immediately for the chosen $\mu, \tau$ as
in~\eqref{eq::paraDS} given that $(\beta^*, \twonorm{\beta^*}) \in \U$.
\end{proofof}

\begin{proofof}{Lemma~\ref{lemma::DS-cone}}
By optimality of $(\hat\beta, \hat{t})$, we have
\bens
\onenorm{\hat\beta} 
+ \lambda \twonorm{\hat\beta} \le \norm{\hat\beta}_1 +
\lambda \hat{t} \le \onenorm{\beta^*} + \lambda \twonorm{\beta^*}
\eens
Thus we have for $S := \supp(\beta^*)$,
\bens
\onenorm{\hat\beta} =
\onenorm{\hat\beta_{\Sc} }+ \onenorm{\hat\beta_{S}}
& \le &  \onenorm{\beta^*}   +   \lambda (\twonorm{\beta^*} -  
 \twonorm{\hat\beta}  )
\eens
Now by the triangle inequality, 
\bens
\onenorm{\hat\beta_{\Sc} } =  
\onenorm{v_{\Sc} } 
& \le & 
 \onenorm{\beta^*_{S}} -   \onenorm{\hat\beta_{S}} + 
 \lambda (\twonorm{\beta^*} -  \twonorm{\hat\beta} ) \\
& \le & 
\onenorm{v_{S}} +  \lambda ( \twonorm{\beta^*} - 
 \twonorm{\hat\beta} ) \\
& \le & 
\onenorm{v_{S}} +  \lambda ( \twonorm{\beta^*} -  \twonorm{\hat\beta_S} ) \\
& = & 
\onenorm{v_{S}} 
+  \lambda  \twonorm{v_S}  \le (1+\lambda) \onenorm{v_{S}}.
 \eens
The lemma thus holds given 
\bens
\hat{t} & \le & 
\inv{\lambda }( \onenorm{\beta^*} -  \norm{\hat\beta}_1)+
\twonorm{\beta^*} 
\le \inv{\lambda }\onenorm{v} + \twonorm{\beta^*} 
\eens
\end{proofof}

\begin{proofof}{Lemma~\ref{lemma::grammatrix}}
Recall the following shorthand notation:
\bens 
D_0 & = &  (\sqrt{\tau_B} + \sqrt{a_{\max}}) \; 
\;\text{ and }  D_2\; = \; 2 (\twonorm{A} + \twonorm{B}) 
\eens 
First we rewrite an upper bound for $v = \hat\beta -\beta^*$, 
$D = \tr(B)$ and $\hat{D} = \hat\tr(B)$
\bens
\norm{X_0^T X_0 v}_{\infty} & =
& \norm{(X-W)^T X_0 (\hat\beta- \beta^*)}_{\infty} 
\le 
\norm{X^T X_0 (\hat\beta - \beta^*)}_{\infty} +
\norm{ W^T X_0 v}_{\infty}\\
& \le & 
\norm{X^T(X \hat\beta - y)- \hat{D} \hat\beta}_{\infty} +  \norm{X^T \e}_{\infty}
+ \norm{(X^T W -D) \hat\beta}_{\infty} \\
&+& \norm{(\hat{D}-D) \hat\beta}_{\infty} +\norm{W^T X_0 v}_{\infty}
\eens
where 
\bens
\norm{X^T X_0 (\hat\beta - \beta^*)}_{\infty}
& \le & 
\norm{X^T (X_0 \hat\beta -  y + \e )}_{\infty}\\
& = & 
\norm{X^T ((X -W)\hat\beta -  y)}_{\infty} + \norm{X^T  \e }_{\infty} \\
& \le &
\norm{X^T (X \hat\beta -  y) - \hat{D} \hat\beta}_{\infty} + \norm{X^T
  \e }_{\infty} \\
&& +  
\norm{(X^T W - D) \hat \beta}_{\infty} + \norm{(\hat{D}- D) \hat \beta}_{\infty}.
\eens
On event $\B_0$, we have by Lemma~\ref{lemma::DS-cone} and the fact that $\hat\beta \in \U$
\bens
I :=  \norm{\hat\gamma - \hat\Gamma  \hat\beta}_{\infty} 
& = & 
\norm{\onef X^T(y - X \hat\beta) + \onef \hat{D}
  \hat\beta}_{\infty}\le 
\mu \hat{t}  + \tau \\
& \le & 
 \mu (\inv{\lambda}\onenorm{v} +
\twonorm{\beta^*} )+ \tau \\
& = &
2 D_2 K r_{m,f} (\inv{\lambda}\onenorm{v} +
\twonorm{\beta^*} )+ D_0 r_{m,f} M_{\e}  
\eens
and on event $\B_4$,
\bens
II & := & \onef \norm{X^T \e}_{\infty}
\le  \onef (\norm{X_0^T \e}_{\infty} + \norm{W^T \e}_{\infty}) \\
& \le &  r_{m,f} M_{\e}(a_{\max}^{1/2} + \sqrt{\tau_B})  = D_0 r_{m,f} M_{\e}
\eens
Thus on event $\B_0$, we have 
\bens
I + II \le 2 D_2 K r_{m,f} (\inv{\lambda}\onenorm{v} +\twonorm{\beta^*} )+ 2 D_0 r_{m,f} M_{\e} =
\mu((\inv{\lambda}\onenorm{v} +\twonorm{\beta^*} )+ 2 \tau.
\eens
Now on event $\B_6$, we have for $2 D_1 \le  D_2$
\bens
IV := \norm{(\hat{D}- D) \hat \beta}_{\infty} 
& \le & 
\abs{\hat{D}- D} \norm{\hat\beta}_{\infty} 
\le  2 D_1  K  \inv{\sqrt{m}}  r_{m,f}  (\norm{\beta^*}_{\infty} + \norm{v}_{\infty} ) \\
& \le &  D_2 K \inv{\sqrt{m}} r_{m,f} (\twonorm{\beta^*} + \norm{v}_1)
\eens
On event $\B_5 \cap \B_{10}$, we have
\bens
III := \onef \norm{(X^T W -D) \hat \beta}_{\infty} 
& \le &
\onef \norm{(X^T W -D) \beta^*}_{\infty} + \onef \norm{(X^T W -D) v}_{\infty} \\
& \le &
 \onef\norm{X_0^T W\beta^*}_{\infty} + \norm{(W^T W -D)
   \beta^*}_{\infty} \\
& + & 
\onef \left( \norm{(Z^T B Z- \tr(B) I_{m})}_{\max} + \onef \norm{X_0^T
    W}_{\max}\right)\onenorm{v} \\
& \le &
r_{m,f} K \left( \frac{\fnorm{B}}{\sqrt{f}} +\sqrt{\tau_B}  a^{1/2}_{\max} \right)(\onenorm{v} + \twonorm{\beta^*}) \\
\text{ and } \; 
V = \onef  \norm{W^T X_0 v}_{\infty} & \le & 
 \onef \norm{W^T X_0 }_{\max} \onenorm{v} \le 
r_{m,f} K \sqrt{\tau_B}  a^{1/2}_{\max} \onenorm{v}.
\eens
Thus we have on $\B_0 \cap \B_{10}$, for $D_0 \le D_2$ and $\tau_A = 1$
\bens 
III + IV + V
& \le &  
r_{m,f} K \left(\twonorm{B} +\tau_B+  a_{\max} +  \frac{2}{\sqrt{m}} 
 (\twonorm{A} +  \twonorm{B} ) \right)(\onenorm{v} + \twonorm{\beta^*}) \\
& \le &  
r_{m,f} K \left(4 \twonorm{B} + 3 \twonorm{A}\right)(\onenorm{v} +
\twonorm{\beta^*}) \\
& \le & 2 D_2 K  r_{m,f}  (\onenorm{v} +\twonorm{\beta^*}) \\
& \le &  \mu   (\onenorm{v}
+\twonorm{\beta^*}) 
\eens
Thus we have
\bens
\norm{\onef X_0^T X_0 v}_{\infty}
& \le &I + II + III + IV + V \\
& \le &  
\mu (\inv{\lambda}\onenorm{v} 
+\twonorm{\beta^*}) + 2 D_0 M_{\e} r_{m,f}
+  \mu  (\onenorm{v} +\twonorm{\beta^*})\\
& \le &  
2 \mu \twonorm{\beta^*} + \mu (\inv{\lambda} + 1)\onenorm{v} + 2 \tau.
\eens 
The lemma thus holds.
\end{proofof}

\section{Proof for Theorem~\ref{thm::DSoracle}}
\label{sec::DSoracleproof}

We prove Lemmas~\ref{lemma::DSimprov} to~\ref{lemma::grammatrixopt} in
this section.

\begin{proofof}{Lemma~\ref{lemma::DSimprov}}
Suppose event $\B_0$ holds. 
Then by the proof of Lemma~\ref{lemma::D2improv},
we have for 
$D_0' = \twonorm{B}^{1/2} + a_{\max}^{1/2}$ and $\tau_B^{+/2} =
\sqrt{\tau_B} + \frac{D_{\ora}}{\sqrt{m}}$, where $D_{\ora} = 2 ( \twonorm{B}^{1/2} + \twonorm{A}^{1/2})$,
\bens
\norm{\hat\gamma - \hat\Gamma \beta^*}_{\infty}
& \le & 
D_0' \tau_B^{+/2}  K r_{m,f} \twonorm{\beta^*} + D_0 M_{\e} r_{m,f}.
\eens
The lemma follows immediately for $\mu, \tau$ as chosen in \eqref{eq::paraDSimprov}.
\end{proofof}

\begin{proofof}{Lemma~\ref{lemma::tauB}}
We first show \eqref{eq::tildetauB} and~\eqref{eq::tildetauBbound}.
Recall $r_{m,m}:=2 C_0 \sqrt{\frac{\log m}{m f}}\ge  2C_0 \frac{\log^{1/2}
 m}{m}$.
By Lemma~\ref{lemma::trBest}, we have on event $\B_6$,  
\bens
\label{eq::tauBoracle}
\abs{\hat\tau_{B} - \tau_B}
&  \le  & D_1 K^2 r_{m,m}.
\eens
Moreover, we have under (A1)  $1 = \tau_A  \le  D_1 :=
\frac{\fnorm{A}}{m^{1/2}} + \frac{\fnorm{B}}{f^{1/2}}$ in view of
\eqref{eq::tracefnorm}.
And
\bens
D_1 \le \twonorm{A} + \twonorm{B} \le (\frac{D_{\ora}}{2})^2
\eens 
and hence
$$\sqrt{D_1} \le \frac{D_{\ora}}{2} = \twonorm{B}^{1/2} + \twonorm{A}^{1/2}.$$
By definition and construction, we have $\tau_B, \hat\tau_B \ge 0$,
\bens
\abs{\hat\tau_B^{1/2} - \tau^{1/2}_B} & \le &  \hat\tau_B^{1/2} +
  \tau^{1/2}_B; \\
\text{ and hence } \; \;  \abs{\hat\tau_B^{1/2} - \tau^{1/2}_B}^2 & \le &  
\abs{(\hat\tau_B^{1/2} + \tau_B^{1/2})(\hat\tau_B^{1/2} - \tau_B^{1/2})} 
= \abs{\hat\tau_{B} - \tau_B} 
\eens
Thus, on event $\B_6$, we have
\bens
\abs{\hat\tau_B^{1/2} - \tau^{1/2}_B} & \le &  
 \sqrt{\abs{\hat\tau_{B} - \tau_B}} \le
\sqrt{D_1} K r^{1/2}_{m,m} \le \frac{D_{\ora}}{2} K r^{1/2}_{m,m}
\eens
Thus we have for $C_{6} \ge D_{\ora} \ge 2\sqrt{D_1}$ and $D_{\ora}
=2 (\twonorm{A}^{1/2} +\twonorm{B}^{1/2})$,
\ben
\label{eq::right}
\hat\tau_B ^{1/2}  - \frac{D_{\ora}}{2} K r_{mm}^{1/2} \le
\tau_B ^{1/2}  \le  
\hat\tau_B ^{1/2}  + \frac{ D_{\ora}}{2} K r_{mm}^{1/2} 
\een
Thus we have for $\tau_B^{+/2}$ as defined
in~\eqref{eq::defineDtau}, \eqref{eq::right} and the fact that
\bens
r^{1/2}_{m,m}:= \sqrt{2 C_0 } \frac{(\log m)^{1/4}}{\sqrt{m}} \ge
2/\sqrt{m} \; \text{  for $m \ge 16$ and $C_0 \ge 1$},
\eens
the following inequalities hold: for $K \ge 1$,
\ben 
\label{eq::tildeBbounds}
\tau_B^{+/2} &  := &   \tau_B^{1/2} + D_{\ora} m^{-1/2} \\
\nonumber
& \le &  
\hat\tau_B^{1/2} + \frac{D_{\ora}}{2} K r_{mm}^{1/2} +
\frac{D_{\ora}}{2} r^{1/2}_{m,m} \\
\nonumber
& \le & 
\hat\tau_B^{1/2} + D_{\ora} K r_{mm}^{1/2} \le \tilde\tau_B^{1/2}  
\een
where the last inequality holds by the choice of 
$\tilde\tau_B^{1/2} \ge \hat\tau_B^{1/2} + D_{\ora} K r_{mm}^{1/2}$ as in \eqref{eq::muchoice}.
Moreover, we have on event $\B_6$, by \eqref{eq::right}
\bens
\tilde\tau_B^{1/2}  & := &  \hat\tau_B^{1/2} +   C_{6} K r_{mm}^{1/2}
\le  \tau_B^{1/2} +   \frac{D_{\ora}}{2} K r_{mm}^{1/2} +   C_{6} K r_{mm}^{1/2} \\
& \le &  \tau_B^{1/2} + \frac{3}{2} C_{6} K r_{mm}^{1/2}\\
\tilde\tau_B  
& := &  
(\PaulBhalf+ C_{6} K r_{mm}^{1/2})^2 \le 
2 \hat \tau_B +2 C_{6}^2 K^2
 r_{mm} \\
& \le &  2 \tau_B + 2 D_1 K^2 r_{m,m} + 2 C_{6}^2 K^2 r_{mm} \\
& \le &  2 \tau_B + \frac{D_{\ora}^2}{2} K^2 r_{m,m} + 2 C_{6}^2 K^2
r_{mm}
\le  2 \tau_B + 3 C_{6}^2 K^2 r_{mm}
\eens
and thus \eqref{eq::tildetauB} and~\eqref{eq::tildetauBbound} hold
given that $2 D_1 \le D_{\ora}^2/2 \le C_6^2/2$.
Finally, we have
\bens
\tilde\tau_B^{1/2}  \tau_B^- \le 
(\tau_B^{1/2} + \frac{3}{2} C_{6} K r_{mm}^{1/2} )
\tau_B^- \le \frac{\tau_B^{1/2} + \frac{3}{2} C_{6} K r_{mm}^{1/2}}
{\tau_B^{1/2} + 2C_{6} K r_{m,m}^{1/2}} \le 1
\eens
for $\tau_B^-$ as defined in~\eqref{eq::ora-sparsity}.
\end{proofof}

\begin{remark}
The set $\U$ in our setting is equivalent to the following:
for $\mu, \tau$ as defined in \eqref{eq::muchoice} and $\beta \in \R^m$,
\ben
\label{eq::defineGamma}
\; \; \; \; 
\U=  \left\{(\beta, t) \; : \; \norm{\onef X^T(y - X \beta) + \onef \hat\tr(B) \beta}_{\infty} \le \mu t + \tau, \twonorm{\beta} \le t\right\}.
\een
\end{remark}

\begin{proofof}{Lemma~\ref{lemma::grammatrixopt}}
For the rest of the proof, we will follow the notation in the proof for
Lemma~\ref{lemma::grammatrix}.
Notice that the bounds as stated in Lemma~\ref{lemma::DS-cone} remain
true with $\tau, \mu$ chosen as in~\eqref{eq::paraDSimprov}, so long
as  $(\beta^*, \twonorm{\beta^*}) \in \U$. This indeed holds by Lemma~\ref{lemma::DSimprov}:
for $\tau$~\eqref{eq::tauchoice} and $\mu$  \eqref{eq::muchoice} as chosen  in Theorem~\ref{thm::DSoracle}, we have by
\eqref{eq::tildeBbounds}, 
\bens
 \mu \asymp D_0' \tilde\tau_B^{1/2}  K r_{m,f}   \ge D_0' K r_{m,f} \tau_B^{+/2}
\eens
where $\tau_B^{+/2} =  ( \sqrt{\tau_B} + \frac{D_{\ora}}{\sqrt{m}})$, 
which ensures that  $(\beta^*, \twonorm{\beta^*}) \in \U$ by
Lemma~\ref{lemma::DSimprov}.

On event $\B_0$, we have by Lemma~\ref{lemma::DS-cone} and 
the fact that $\hat\beta \in \U$ as in \eqref{eq::defineGamma}
\bens
I + II & := & 
\norm{\hat\gamma - \hat\Gamma  \hat\beta}_{\infty} + \onef
\norm{X^T \e}_{\infty} \\
& \le & 
\norm{\onef X^T(y - X \hat\beta) + \onef \hat{D}  \hat\beta}_{\infty} + \tau \le \mu \hat{t}  + 2 \tau \\
& \le & 
\mu (\inv{\lambda}\onenorm{v} + \twonorm{\beta^*} )+2 \tau 
\eens
for  $\mu, \tau$ as chosen in \eqref{eq::muchoice}  and
\eqref{eq::tauchoice} respectively.
Now on event $\B_6$, we have
\bens
IV := \norm{(\hat{D}- D) \hat \beta}_{\infty} 
& \le & 
\abs{\hat{D}- D} \norm{\hat\beta}_{\infty} 
\le  2 D_1  K  \inv{\sqrt{m}}  r_{m,f}  (\norm{\beta^*}_{\infty} + \norm{v}_{\infty} ) \\
& \le &  D_0' \frac{D_{\ora}}{\sqrt{m}} K  r_{m,f} (\twonorm{\beta^*} + \norm{v}_1)
\eens
where $2D_1 \le D_{\ora} D_0'$ for $1 \le D_0' :=  \twonorm{B}^{1/2} +
a_{\max}^{1/2}$ and $D_{\ora} =   2 \left(\twonorm{B}^{1/2} + \twonorm{A}^{1/2} \right)$, where  $a_{\max} \ge \tau_A = 1$ under (A1).
Hence
\bens
\lefteqn{III + IV + V 
\le r_{m,f} K \sqrt{\tau_B} \left(\twonorm{B}^{1/2} + a^{1/2}_{\max} \right)(\onenorm{v} + \twonorm{\beta^*})} \\
&&
+ 2 D_1 K \inv{\sqrt{m}} r_{m,f} (\twonorm{\beta^*} + \norm{v}_1)
+ r_{m,f} K \sqrt{\tau_B}  a^{1/2}_{\max} \onenorm{v} \\
& \le &  
D_0' K r_{m,f}  (\onenorm{v} + \twonorm{\beta^*})  
( \sqrt{\tau_B} +  \frac{D_{\ora}}{\sqrt{m}}) + r_{m,f} K \sqrt{\tau_B}  a^{1/2}_{\max} \onenorm{v} \\ 
&\le &
D_0' K r_{m,f} \tau_B^{+/2} (\onenorm{v} + \twonorm{\beta^*} ) + D_0'  K r_{m,f} \sqrt{\tau_B} 
\onenorm{v} \\ 
&\le &
C_0 D_0' K^2 \sqrt{\frac{\log m}{f}}
(\tau_B^{1/2} + \frac{D_{\ora}}{\sqrt{m}}) (2\onenorm{v} + \twonorm{\beta^*} ) \\
& \le & \mu (2\onenorm{v} + \twonorm{\beta^*}) 
\eens
for $\mu$ as defined in   \eqref{eq::muchoice} in view of \eqref{eq::tildeBbounds}.
Thus we have 
\bens
I + II + III + IV + V 
& \le &  \mu  (\inv{\lambda}\onenorm{v} + \twonorm{\beta^*} )+  2 \tau
+  \mu (2\onenorm{v} + \twonorm{\beta^*})  \\
& = & 
2\mu ((1 + \inv{2\lambda})\onenorm{v} + \twonorm{\beta^*}) + 2 \tau
\eens
and the improved bounds as stated in the Lemma thus holds.
\end{proofof}

\section{Some geometric analysis results}
\label{sec::geometry}
Let us define the following set of vectors in $\R^m$:
\bens
\Cone(s_0) := \{\up: \onenorm{\up} \le \sqrt{s_0} \twonorm{\up}\}
\eens
For each vector $x \in \R^m$, let ${T_0}$ denote the locations of the $s_0$
largest coefficients of $x$ in absolute values.
Any vector $x \in S^{m-1}$ satisfies:
\ben
\label{eq::init-norm-inf}
\norm{x_{T_0^c}}_{\infty}  \leq \norm{x_{T_0}}_{1}/s_0
& \leq & \frac{ \twonorm{x_{T_0}}}{\sqrt{s_0}}  
\een
We need to state the following result from~\cite{MPT08}.
Let $S^{m-1}$ be the unit sphere in $\R^m$, for $1 \leq s \leq m$, 
\beq
U_s \; := \; \{ x \in \R^{m}: |\supp(x)| \leq s \}
\eeq
The sets $U_s$ is an union of the $s$-sparse vectors.
The following three lemmas are well-known and mostly standard; 
See~\cite{MPT08} and~\cite{LW12}.
\begin{lemma}
\label{eq::embedding}
For every $1\le s_0 \le m$ and every $I \subset \{1, \ldots, m\}$ with
$\abs{I} \le s_0$,
\bens
\sqrt{\abs{I}} B_1^m \cap S^{m-1} \subset 2 \conv( U_{s_0} \cap S^{m-1}) 
=: 2 
\conv\left(\bigcup_{\size{J} \leq s_0} E_J \cap S^{m-1}\right) 
\eens
and moreover, for $\rho \in (0, 1]$.
\bens
\sqrt{\abs{I}} B_1^m \cap \rho B_2^{m} \subset (1+\rho) \conv( U_{s_0} \cap B_2^{m}) 
=: (1+\rho) \conv\left(\bigcup_{\size{J} \leq {s_0}} E_J \cap S^{m-1}\right) 
\eens
\end{lemma}

\begin{proof}
Fix  $x \in \R^m$.
Let $x_{T_0}$ denote the subvector of $x$
confined to the locations of its $s_0$ largest coefficients in absolute values;
moreover, we use it to represent its $0$-extended
version $x' \in \R^p$ such that $x'_{T^c} =0$ and
 $x'_{T_0} =x_{T_0}$.
Throughout this proof, $T_0$ is understood to be the locations of the $s_0$ largest 
coefficients in absolute values in $x$.

Moreover, let $(x_i^*)_{i=1}^m$ be non-increasing rearrangement of 
$(\abs{x_i})_{i=1}^m$.
Denote by 
\bens
L & = & \sqrt{s_0} B_1^m \cap \rho B_2^m \\ 
R & = & 2 \conv\left(\bigcup_{\size{J} \leq s} E_J \cap B_2^{m}\right) =
 2 \conv\big( E \cap B_2^{m}\big)
\eens
Any vector $x \in \R^{m}$ satisfies:
\ben
\label{eq::init-norm-inf}
\norm{x_{T_0^c}}_{\infty}  
\leq \norm{x_{T_0}}_{1}/s_0
& \leq & \frac{ \twonorm{x_{T_0}}}{\sqrt{s_0}}  
\een
It follows that for any $\rho >0$, $s_0 \ge 1$ and 
for all $z \in L$, we have the $i^{th}$ largest coordinate in absolute value in $z$ 
is at most $\sqrt{s_0}/i$,
\bens
\sup_{z \in L} \ip{x, z} & \le &
\max_{\twonorm{z} \le \rho} \ip{x_{T_0}, z}
+ \max_{\onenorm{z} \le \sqrt{s_0}} \ip{x_{T_0^c}, z}   \\
 & \le &  \rho \twonorm{x_{T_0}} + 
\norm{x_{T_0^c}}_{\infty} \sqrt{s_0} \\
 & \le & 
\twonorm{x_{T_0}} \left( \rho  + 1\right)
\eens
where clearly
$\max_{\twonorm{z} \le \rho} \ip{x_{T_0}, z} = \rho  \sum_{i=1}^{s_0} (x_i^{*2})^{1/2}$.
And denote by $S^J := S^{m-1} \cap E_J$,
\bens
\sup_{z \in R} \ip{x, z} & = & (1+\rho)
\max_{J: \size{J} \le s_0} \max_{z \in S^J} \ip{x, z} \\
 & = & (1+ \rho) \twonorm{x_{T_0}}
\eens
given that for a convex function $\ip{x, z}$, the maximum happens at an extreme point, and in this case, it happens for $z$ such that $z$ is supported on $T_0$, such that $z_{T_0} = \frac{x_{T_0}}{\twonorm{x_{T_0}}}$, and $z_{T_0^c} =0$.
\end{proof}

\begin{lemma}
\label{lemma::bigcone}
Let $1/5 > \delta > 0$.
Let $E=\cup_{|J| \leq s_0} E_J$ for $0 < s_0 < m/2$ and $k_0>0$.
Let $\Delta$ be a $m \times m$ matrix such that
\ben
\label{eq::conecond}
\abs{u^T \Delta v} \le \delta  \;\; \; \forall u, v \in E \cap S^{m-1}
\een
Then for all $v \in \big(\sqrt{s_0} B_1^m \cap B_2^m\big)$, we have
\ben
\label{eq::origset}
\abs{\up^T \Delta \up} & \le & 4\delta.
\een
\end{lemma}

\begin{proof}
First notice that
\ben
\label{eq::decoupledbig}
\max_{\up \in \big(\sqrt{s_0} B_1^m \cap B_2^m\big)} \abs{\up^T \Delta \up}
& \le & 
\max_{w, u \in  \big(\sqrt{s_0} B_1^m \cap B_2^m\big)}  \abs{w^T \Delta u} 
\een
Now that we have decoupled $u$ and $w$ on the RHS of
\eqref{eq::decoupledbig},
 we first fix $u$.
Then for any fixed $u \in S^{m-1}$ and matrix $\Delta \in \R^{m \times m}$,
$f(w) = \abs{w^T \Delta u}$ is a convex function of $w$, and hence 
for $w \in \big(\sqrt{s_0} B_1^m \cap  B_2^{m}\big) \subset 
2 \conv\left(\bigcup_{\size{J} \leq s_0} E_J \cap S^{m-1}\right)$,
\bens
\max_{w \in \big(\sqrt{s_0} B_1^m \cap B_2^m\big)}  \abs{w^T \Delta u} 
& \le & 
2 \max_{w \in \conv( E \cap  S^{m-1})} \abs{w^T \Delta  u} \\
& =& 
2 \max_{w \in E \cap  S^{m-1}} \abs{w^T \Delta u}
\eens
where the maximum occurs at an extreme point of the set 
$ \conv( E \cap  S^{m-1})$, because of the convexity of the function
$f(w)$,

Clearly the RHS of \eqref{eq::decoupledbig} is bounded by
\bens
\max_{u, w \in \big(\sqrt{s_0} B_1^m \cap B_2^m\big)}  \abs{w^T
  \Delta u}
& = &
\max_{u \in \big(\sqrt{s_0} B_1^m \cap B_2^m\big)} 
\max_{w \in \big(\sqrt{s_0} B_1^m \cap B_2^m\big)}  \abs{w^T \Delta u}
\\
& \le &
2 \max_{u \in \big(\sqrt{s_0} B_1^m \cap B_2^m\big)} 
\max_{w \in \big(E\cap S^{m-1}\big)}  \abs{w^T \Delta u}\\
& = &
2 \max_{u \in \big(\sqrt{s_0} B_1^m \cap B_2^m\big)} g(u) 
\eens
where the function $g$ of $u \in \big(\sqrt{s_0} B_1^m \cap
B_2^m\big)$ is defined as
\bens
g(u) =\max_{w \in \big(E \cap S^{m-1}\big)}  \abs{w^T \Delta u}
\eens
which is convex since it is the maximum of a function 
$f_w(u) := \abs{w^T \Delta u}$ 
which is convex in $u$ for each $w \in (E \cap S^{m-1})$.
Thus we have 
for $u \in (\sqrt{s_0} B_1^m \cap  B_2^{m}) \subset 2 \conv\left(\bigcup_{\size{J} \leq s_0} E_J \cap  S^{m-1}\right) =: 2\conv\left(E \cap S^{m-1}\right)$
\ben
\nonumber
\max_{u \in  \big(\sqrt{s_0} B_1^m \cap B_2^m\big)} g(u)
& \le & 
2 \max_{u\in \conv(E \cap  S^{m-1})} g(u) \\
& = & 
\label{eq::extremeg}
2 \max_{u\in E \cap  S^{m-1}} g(u) \\
& = & 
\label{eq::last}
2 \max_{u \in  E \cap  S^{m-1}} \max_{w \in  E \cap  S^{m-1}} \abs{w^T  \Delta u}
\le 4 \delta
\een
where~\eqref{eq::extremeg} holds given that the maximum occurs at an extreme point of the set 
$ \conv( E \cap  B_2^{m})$, because of the convexity of the function
$g(u)$.
\end{proof}

\begin{corollary}
\label{coro::bigcone} 
Suppose all conditions in Lemma~\ref{lemma::bigcone} hold.
Then $\forall \up \in \Cone(s_0)$,
\ben
\label{eq::cone}
\abs{\up^T \Delta \up} & \le & 4\delta  \twonorm{\up}^2.
\een
\end{corollary}

\begin{proof}
It is sufficient to show that $ \forall \up \in \Cone(s_0) \cap S^{m-1}$,
\bens
\abs{\up^T \Delta \up} & \le & 4\delta.
\eens
Denote by $\Cone := \cone(s_0)$.
Clearly this set of vectors satisfy:
\bens
\cone \cap S^{m-1}\subset \big(\sqrt{s_0} B_1^m \cap B_2^m\big)
\eens
Thus~\eqref{eq::cone} follows from~\eqref{eq::origset}.
\end{proof}

\begin{remark}
Suppose we relax the definition of $\Cone(s_0)$ to be:
\bens
\Cone(s_0) := \{\up: \onenorm{\up} \le 2 \sqrt{s_0} \twonorm{\up}\}
\eens
Clearly, $\Cone(s_0, 1) \subset \Cone(s_0)$.
given that $\forall u \in  \Cone(s_0, 1)$, we have 
\bens
\onenorm{u} \le 2 \onenorm{u_{T_0}} \le 2 \sqrt{s_0} \twonorm{u_{T_0}} \le 
2 \sqrt{s_0} \twonorm{u}
\eens
\end{remark}

\begin{lemma}
\label{lemma::bigconeII}
Suppose all conditions in Lemma~\ref{lemma::bigcone} hold.
Then for all $\up \in \R^m$,
\ben
\label{eq::bigcone}
\abs{\up^T \Delta \up} \le 4 \delta (\twonorm{\up}^2 + \inv{s_0} \onenorm{\up}^2) 
\een
\end{lemma}

\begin{proof}
The lemma follows given that $\forall \up \in \R^m$, one of the following must hold:
\ben
\label{eq::cone-II}
\text{ if }  \up \in \Cone(s_0) \;\; \;
\abs{\up^T \Delta \up} & \le & 4 \delta \twonorm{\up}^2 \\
\label{eq::anticone}
\text{ otherwise }  \;\;\;
\abs{\up^T \Delta \up} & \le &  \frac{4\delta} 
{s_0}\onenorm{\up}^2,
\een
leading to the same conclusion in \eqref{eq::bigcone}.
We have shown~\eqref{eq::cone-II} in Lemma~\ref{lemma::bigcone}.
Let $\Cone(s_0)^c$ be the  complement set of $\cone(s_0)^c$ in $\R^{m}$. That is, we focus now on the set of vectors such that 
\bens
\Cone(s_0)^c := \{\up: \onenorm{\up} \ge \sqrt{s_0} \twonorm{\up}\}
\eens
and show that for $u = \sqrt{s_0} \frac{v}{\onenorm{v}}$,
\bens
\frac{\abs{v^T \Delta v} }{\onenorm{v}^2} 
& := & \inv{s_0} \abs{u^T \Delta u}  \le    \inv{s_0} \delta
\eens
where the last inequality holds by Lemma~\ref{lemma::bigcone} given that 
\bens
u \in (\sqrt{s_0} B_1^m \cap B_2^m) \subset 
2 \conv\left(\bigcup_{\size{J} \leq s_0} E_J \cap B_2^{m}\right) 
\eens 
and thus 
\bens
\frac{\abs{v^T \Delta v} }{\onenorm{v}^2} 
& \le & \inv{s_0} \sup_{u \in \sqrt{s_0} B_1^m \cap B_2^m}
 \abs{u^T \Delta u}  \le  \inv{s_0}4 \delta
\eens
\end{proof}

\section{Proof of Corollary~\ref{coro::BC}}
\label{sec::appendLURE}
\begin{proofof2}
First we show that for all $\up \in \R^m$, \eqref{eq::conebound} holds.
It is sufficient to check that the condition~\eqref{eq::conecond} in
Lemma~\ref{lemma::bigcone} holds. Then, \eqref{eq::conebound} follows
from Lemma~\ref{lemma::bigconeII}: for $\up \in \R^m$, 
\ben
\label{eq::conebound}
\abs{\up^T \Delta \up} \le 4 \delta (\twonorm{\up}^2 + \inv{\zeta} \onenorm{\up}^2) 
\le \half \lambda_{\min}(A)  (\twonorm{\up}^2 + \inv{\zeta} \onenorm{\up}^2).
\een
The Lower and Upper $\RE$ conditions thus immediately follow.
The Corollary is thus proved.
\end{proofof2}
\section{Proof of Theorem~\ref{thm::AD}}
\label{sec::proofofthmAD}
We first state the following preliminary results in 
Lemmas~\ref{lemma::normA}  and~\ref{lemma::orthogSp}; their proofs 
appear in Section~\ref{sec::proofofnormA}.
Throughout this section, the choice of $C = C_0/\sqrt{c'}$ satisfies the conditions on $C$ in Lemmas~\ref{lemma::normA}
and~\ref{lemma::orthogSp}, where recall $\min\{C_0, C_0^2\} \ge 4/c$ for $c$ as
defined in Theorem~\ref{thm::HW}.
For a set $J \subset \{1, \ldots, m\}$, denote $F_J=A^{1/2} E_J$ where recall
$E_J = \spin\{e_j: j \in J\}$.
\begin{lemma}
\label{lemma::normA}
Suppose all conditions in Theorem~\ref{thm::AD} hold.
Let $$E  = \bigcup_{\abs{J}=k} E_J \cap S^{m-1}.$$
Suppose that for some $c' > 0$ and $\ve \le \inv{C}$, where $C = C_0/\sqrt{c'}$,
\ben
 \label{eq::ALocalkronsum}
r(B) := \frac{\tr(B)}{\twonorm{B}} & \ge & c' k K^4 \frac{\log(3e m/k \ve)}{\ve^2}.
\een
Then for all vectors $u, v \in E \cap S^{m-1}$, on event $\B_1$, 
where $\prob{\B_1} \ge 1- 2 \exp\left(-c_2\ve^2 \frac{\tr(B)}{K^4\twonorm{B}}\right)$ for $c_2 \ge 2$,
\bens
\abs{u^T Z^T B Z v - \E u^T Z^T B Z v } 
& \le & 4 C \ve \tr(B).
\eens
\end{lemma}

\begin{lemma}
\label{lemma::orthogSp}
Suppose that $\ve \le 1/C$, where $C$ is as defined in Lemma~\ref{lemma::normA}.
Suppose that \eqref{eq::ALocalkronsum} holds.
Let
\ben
\label{eq::spsetEF}
E = \bigcup_{\abs{J} =k} E_J \; \; \text{ and } \; \; F =
\bigcup_{\abs{J} =k} F_J.
\een
Then on event $\B_2$, 
where $\prob{\B_2} \ge 1- 2 \exp\left(-c_2\ve^2
  \frac{\tr(B)}{K^4\twonorm{B}}\right)$ for $c_2 \ge 2$,
we have for all vectors $u \in E \cap  S^{m-1}$ and $w \in F \cap  S^{m-1}$,
\bens
\abs{w^T Z_1^T B^{1/2} Z_2 u} 
& \le & 
\frac{C \ve \tr(B)}{(1-\ve)^2\twonorm{B}^{1/2}} \le
{4C \ve \tr(B)}/{\twonorm{B}^{1/2}}
\eens
where $Z_1, Z_2$ are independent copies of $Z$, as defined in
Theorem~\ref{thm::AD}.
\end{lemma}
In fact, the same conclusion holds for all $y, w \in F\cap S^{m-1}$; 
and in particular, for $B = I$, we have the following.
\begin{corollary}
\label{coro::tartan}
Suppose all conditions in 
Lemma~\ref{lemma::normA} hold.
Suppose that $F = A^{1/2} E$ for $E$ as defined in Lemma~\ref{lemma::normA}.
Let
\ben
\label{eq::BI} 
f & \ge &  c'k K^4 \frac{\log(3e m/k \ve)}{\ve^2}. 
\een
Then on event $\B_3$, 
where $\prob{\B_3} \ge 1- 2 \exp\left(-c_2\ve^2 f  \inv{K^4}\right)$,
we have for all vectors $w,y \in F \cap S^{m-1}$ and  $\ve \le 1/C$ for $C$ is as defined in Lemma~\ref{lemma::normA},
\ben
\label{eq::wyFnorm}
\abs{y^T (\onef Z^T Z - I) w } & \le & 4 C\ve.
\een
\end{corollary}
We prove Lemmas~\ref{lemma::normA} and~\ref{lemma::orthogSp}
  and Corollary~\ref{coro::tartan} in Section~\ref{sec::proofofnormA}.
We are now ready to prove Theorem~\ref{thm::AD}.

\begin{proofof}{Theorem~\ref{thm::AD}}
Recall the following for $X_0 = Z_1 A^{1/2}$,
\bens
\lefteqn{
\Delta := \hat\Gamma_{A} -A := \onef X^TX - \onef \hat\tr(B) I_{m} -A} \\
& = & (\onef X_0^T X_0 -A)+  
 \onef \big(W^T X_0 + X_0^T W\big) + \onef \big(W^T W  - \hat\tr(B) I_{m}\big).
\eens
Notice that 
\bens
\lefteqn{
\abs{u^T(\hat\Gamma_A -A) \up} =  
\abs{u^T(X^T X - \hat\tr(B) I_{m} -A) \up} }\\
& \le &  
\abs{u^T  (\onef X_0^T X_0 - A) \up} + 
\abs{u^T \onef(W^T X_0 + X_0^T W) \up} + 
\abs{u^T (\onef W^T W - \frac{\hat\tr(B)}{f} I_{m})\up}\\
& \le &  
\abs{u^T A^{1/2}\onef Z_1^T Z_1 A^{1/2} \up - u^T A \up} +
\abs{u^T \onef(W^T X_0 + X_0^T W) \up} \\
&& +
\abs{u^T (\onef Z_2^T B Z_2 - \tau_B I_{m})\up} + 
\onef \abs{\hat\tr(B) - \tr(B)} \abs{u^T \up} =: I + II + III + IV.
\eens
For $u  \in E \cap S^{m-1}$, define $h(u) := \frac{A^{1/2}
  u}{\twonorm{A^{1/2} u}}$.
The conditions in~\eqref{eq::ALocalkronsum} and~\eqref{eq::BI} hold
for $k$. We first bound the middle term as follows.
Fix $u, \up \in E \cap S^{m-1}$
Then on event $\B_2$, for $\Upsilon = Z_1^T B^{1/2} Z_2$,
\bens
\abs{u^T (W^T X_0 + X_0^T W) \up} & = & 
 \abs{u^T Z_2^T B^{1/2} Z_1 A^{1/2} \up + u^T A^{1/2} Z_1^T B^{1/2} Z_2 \up} \\ 
& \le & 
\abs{u^T \Upsilon^T h(v)}\twonorm{A^{1/2} v}
 + \abs{h(u)^T \Upsilon \up} \twonorm{A^{1/2} u}\\
& \le &  2\max_{w \in F \cap S^{m-1}, \up \in E \cap S^{m-1}}
\abs{w^T \Upsilon \up} \rho_{\max}^{1/2}(k, A) \\
& \le & 8 C \ve \tr(B)\left(\frac{\rho_{\max}(k, A)}{\twonorm{B}}\right)^{1/2}.
 \eens
We now use Lemma~\ref{lemma::normA} to bound both $I$ and $III$.
We have for $C$ as defined in Lemma~\ref{lemma::normA}, on event $\B_1 \cap \B_3$,
\bens
 \abs{u^T (Z_2^T B Z_2 - \tr(B)I_{m})\up} 
& \le  4 C \ve \tr(B).
\eens 
Moreover, by Corollary~\ref{coro::tartan}, we have on event $\B_3$, for all 
 $u, v \in E \cap S^{m-1}$,
\bens
\abs{u^T  (\onef X_0^T X_0 - A) \up} 
& = & 
\abs{u^T A^{1/2} Z^T Z A^{1/2} \up - u^T A \up} \\
& = & 
\abs{h(u)^T (\onef Z^T Z -I) h(\up)} \twonorm{A^{1/2} u} \twonorm{A^{1/2} \up} \\
& \le &
\onef \max_{w, y \in F \cap S^{m-1}} \abs{w^T (Z^T Z -I) y} \rho_{\max}(k, A) \\
& \le & 4 C \ve  \rho_{\max}(k, A).
\eens
Thus we have on event $\B_1 \cap \B_2 \cap \B_3$ and for $\tau_B := \tr(B)/f$
\bens
I + II + III 
& \le & 
 4 C \ve \left(\rho_{\max}(k, A)  +  
2 \tau_B \left(\frac{\rho_{\max}(k, A)}{\twonorm{B}}\right)^{1/2} + \tau_B \right) \\
& \le &  8 C \ve\left(\tau_B +  \rho_{\max}(k, A) \right).
\eens
On event $\B_6$, we have for $D_1$ as defined in Lemma~\ref{lemma::trBest},
\bens
IV \le \abs{\hat\tau_B - \tau_B} \le  2 C_0 D_1 K^2 \sqrt{\frac{\log m}{fm}}.
\eens
The theorem thus holds by the union bound.
\end{proofof}

\section{Proof for Theorem~\ref{thm::kronopB}}
\label{sec::proofofgramB}
We first state the following bounds in~\eqref{eq::middleHWA} 
before we prove Theorem~\ref{thm::kronopB}.
On event $\A_2$, where $\prob{\A_2} \ge 
1- 2 \exp\left(-c_3 \ve^2 \frac{\tr(A)}{K^4\twonorm{A}}\right)$
\ben
\label{eq::middleHWA} 
\forall u, w \in S^{f-1} \; \; \; 
\abs{u^T Z_1 A^{1/2} Z_2^T w} 
\le 
\frac{4C \ve \tr(A)}{\twonorm{A}^{1/2}}.
\een
To see this, first note that by Lemma~\ref{lemma::oneeventA}, we have 
for $t = C \ve \tr(A)/\twonorm{A}^{1/2}$ and $\ve \le 1/2$,
\bens
\prob{\abs{u^T Z_1 A^{1/2} Z_2^T w} >  t}& \le & 
2 \exp\left(-c\min\left(\frac{C^2\ve^2 \tr(A)}{K^4\twonorm{A}},
\frac{C\ve \tr(A)}{K^2\twonorm{A}}\right)\right) \\
& \le & 
2 \exp\left(-c\min\left(C^2, 2C\right) \frac{\ve^2 \tr(A)}{K^4\twonorm{A}}\right) 
\eens
where recall  
$$C' = c c' \min\left(2C, C^2\right) > 4.$$

Before we proceed, we state the following well-known result on {\em volumetric estimate}; see e.g.~\cite{MS86}.
\begin{lemma}
\label{eq::Pi-cover-numbers}
Given $m \geq 1$ and $\ve >0$. 
There exists an $\ve$-net
$\Pi \subset B_2^m$ of $B_2^m$ with respect to the Euclidean metric such
that $B_2^m \subset (1- \ve)^{-1} \conv \Pi$ and  $|\Pi| \leq (1+2/\ve)^m$.
Similarly, there exists an $\ve$-net of the sphere
$S^{m-1}$, $\Pi' \subset S^{m-1}$ such that $|\Pi'| \leq (1+2/\ve)^m$.
\end{lemma}

Choose an $\ve$-net $\Pi \subset S^{f-1}$ such that 
$\abs{\Pi} \le (1+2/\ve)^{f} = \exp(f\log(3/\ve)).$
The existence of such $\Pi$ is guaranteed by Lemma~\ref{eq::Pi-cover-numbers}.
By the union bound and Lemma~\ref{lemma::oneeventA}, we have for
some  $C \ge 2$ and $c' \ge 1$ large enough such that
\bens
\prob{\exists u, w \in \Pi s.t. 
\abs{u^T Z_1 A^{1/2} Z_2^T w} 
\ge C \ve \frac{\tr(A)}{\twonorm{A}^{1/2}}}
& \le &  2 \exp\left(-c_3 \frac{\ve^2 \tr(A)}{K^4\twonorm{A}}\right).
\eens
Hence,~\eqref{eq::middleHWA} follows from
a standard approximation argument.
\begin{lemma}  
\label{lemma::tracebound}
 Let $\ve>0$. Let $Z$ as defined in Definition~\ref{def::subgdata}.
 Assume that 
  \[
   \frac{\tr(A)}{\norm{A}} \ge c' f \frac{\log(3/\ve)}{\ve^2}.
 \]
 Then 
 \[
   \prob{\exists x \in S^{f-1} \ \ \left| \norm{A^{1/2}Z^T x}_2 - (\tr(A))^{1/2} \right|
         > \ve (\tr(A))^{1/2}}
   \le \exp \left( -c \ve^2 \frac{\tr(A)}{K^4 \norm{A}} \right).
 \]
\end{lemma}

\begin{proof}
 Let $x \in S^{f-1}$. Then $Y=Z^T x \in \R^m$ is a random vector with independent coordinates satisfying $\E Y_j=0$ and $\norm{Y_j}_{\psi_2} \le CK$ for all $j \in 1 \ldots m$.
 The last estimate follows from Hoeffding inequality. By Theorem 2.1 \cite{RV13},
 \[
  \prob{\left| \norm{A^{1/2}Y}_2 - (\tr(A))^{1/2} \right| >\ve (\tr(A))^{1/2}}
  \le \exp \left( -c \ve^2 \frac{\tr(A)}{K^4 \norm{A}} \right). 
 \]
 Choose an $\ve$-net $\Pi \subset S^{f-1}$ such that $|\Pi| \le (3/\ve)^f$. 
 By the union bound and the assumption of the Lemma,
 \bens
  \prob{\exists x \in \Pi \ \ \left| \norm{A^{1/2}Z^T x}_2 - (\tr(A))^{1/2} \right| >\ve (\tr(A))^{1/2}}
  &\le & |\Pi| \cdot \exp \left( -c \ve^2 \frac{\tr(A)}{K^4 \norm{A}} \right)  \\
  &\le & \exp \left( -c' \ve^2 \frac{\tr(A)}{K^4 \norm{A}} \right).   
 \eens
 A standard approximation argument shows that if $\left| \norm{A^{1/2}Z^T x}_2 - (\tr(A))^{1/2} \right|  \le \ve (\tr(A))^{1/2}$ for all $x \in \Pi$, then $\left| \norm{A^{1/2}Z^T x}_2 - (\tr(A))^{1/2} \right| \le 3 \ve (\tr(A))^{1/2}$
 for all $x \in S^{f-1}$. This finishes the proof of the Lemma.
\end{proof}

\begin{proofof}{Theorem~\ref{thm::kronopB}}
First we write
\bens
\lefteqn{
 X X^T - \tr(A) I_{f}
 = \big( Z_1 A^{1/2} + B^{1/2} Z_2) \big(Z_1 A^{1/2} +B^{1/2}
 Z_2\big)^T - \tr(A) I_{f} } \\
& = & 
\big( Z_1 A^{1/2} + B^{1/2} Z_2) \big(Z_2^T B^{1/2} +A^{1/2} Z_1^T\big) 
- \tr(A) I_{f} \\
& = & 
 Z_1 A^{1/2} Z_2^T B^{1/2} + B^{1/2} Z_2  Z_2^T B^{1/2}
+ B^{1/2} Z_2  A^{1/2} Z_1^T+ Z_1 A Z_1^T- \tr(A) I_{f}. 
\eens
Hence,
\bens
\nonumber
\lefteqn{
\abs{\frac{u^T (X X^T)u}{m}-\frac{u^T\tr(A)I u }{m} - u^T B u }
\le 
\abs{\inv{m}u^T  Z_1 A  Z_1^T u- \frac{\tr(A)}{m} u^T u}}\\
&& 
+ \abs{\inv{m} u^T B^{1/2} Z_2  Z_2^T B^{1/2} u - u^T B u}
+\frac{2}{m}\abs{ u^T Z_1 A^{1/2} Z_2^T B^{1/2} u}.
\eens
where by~\eqref{eq::middleHWA},
we have on event $\A_2$, for $\tau_A :=\frac{\tr(A)}{m}$
and $w := \frac{B^{1/2} u}{\twonorm{B^{1/2} u}}$,
\bens
\lefteqn{
\frac{2}{m}
\abs{u^T Z_1 A^{1/2} Z_2^T B^{1/2} u}  = 
\frac{2}{m} \abs{u^T Z_1 A^{1/2} Z_2^T w}
 \twonorm{B^{1/2} u} }\\
 &\le &\frac{8 C \ve \tr(A)  \twonorm{B^{1/2} u} }{\twonorm{A}^{1/2}
   m} =: 8 C \ve \tau_A \twonorm{B^{1/2} u} /{\twonorm{A}^{1/2}}.
\eens
Moreover, by the union bound and Lemma~\ref{lemma::tracebound}, we have 
on event $\A_1$, where $\prob{\A_1} \ge  1 -\exp(c \ve^2 \frac{m}{K^4})-\exp(c \ve^2
\frac{\tr(A)}{K^4\twonorm{A}})$,
\bens
(1-\ve) \twonorm{B^{1/2} u}& \le & 
\frac{1}{\sqrt{m}}\twonorm{ Z_2 B^{1/2} u} \le (1+\ve)
\twonorm{B^{1/2} u}\\
(1-\ve) \frac{\tr(A)^{1/2}}{\sqrt{m}} 
& \le &
\inv{\sqrt{m}}\twonorm{A^{1/2} Z_1^T u}
\le (1+\ve) \frac{\tr(A)^{1/2}}{\sqrt{m}}.
\eens
Hence on event $\A_1$, we have 
\bens
\inv{m}\abs{\twonorm{A^{1/2}  Z_1^T u }^2
- \tr(A)}& \le &
\max((1+\ve)^2-1, 1-(1-\ve)^2) \frac{\tr(A)}{m}, \\
\abs{\inv{m} \twonorm{  Z_2^T B^{1/2} u}^2- u^T B u}
& \le &
\max((1+\ve)^2-1, 1-(1-\ve)^2)\twonorm{B^{1/2} u}^2.
\eens
Thus we have for all $u \in S^{f-1}$, on event $\A_1 \cap \A_2$, for
$C_2 := 4C + 3$
\bens
\lefteqn{
\abs{\inv{m}u^T (X X^T)u -u^T\frac{\tr(A)I_f}{m} u -u^T B u } \le} \\  
& \le &
\abs{\twonorm{Z_2^T B^{1/2} u}^2/m - u^T B u} +
\inv{m}\abs{\twonorm{A^{1/2}  Z_1^T u }^2 - \tr(A)} + 8C \ve \tau_A \twonorm{B^{1/2} u} /{\twonorm{A}^{1/2}} \\
& \le &
 3\ve \twonorm{B^{1/2} u}^2 + 3 \ve \tau_A + 8 C \ve \tau_A \twonorm{B^{1/2}
   u} /{\twonorm{A}^{1/2}} \le
C_2 \ve \twonorm{B^{1/2} u}^2 + C_2 \ve \tau_A
\eens
where $2\tau^{1/2}_A \twonorm{B^{1/2} u} \le  \tau_A +\twonorm{B^{1/2} u}^2$.
The theorem thus holds.
\end{proofof}

\subsection{Proof of Corollary~\ref{coro::kronopB}}
\label{sec::kronopB}
\begin{proofof2}
\noindent {\bf Lower bound:}
For all $u \in S^{f-1}$ and 
\bens
\lefteqn{
\inv{m}u^T (X X^T)u -u^T\frac{\tr(A)I_f}{m} u}\\
& \ge &  u^T B u (1-3\ve) -3\ve \tau_A -
8 C\twonorm{B^{1/2}u} \ve \tau_A /\twonorm{A}^{1/2} \\
& \ge &  u^T B u (1-3\ve - 4C \ve) -3\ve \tau_A  - 4 C \ve \tau_A\\
& \ge &  u^T B u (1-C_2\ve) - C_2\ve \tau_A \ge u^T B u (1-2\delta)
\eens
where we bound the term  using the fact that 
$1 \le \tau_A \le \lambda_{\max}(B)$ and
\bens
C_2\tau_A \ve & \le & \delta \lambda_{\min}(B) 
\; \text{ and } \;
\ve  \le \delta \lambda_{\min}(B) /(C_2\tau_A )  \\
C_2 \ve & \le & \delta \; \text{ and } \;
C_ 3 \ve  \le \delta \min \left( \frac{\lambda_{\min}(B) }{\tau_A}, 1\right).
\eens
By a similar argument, we can prove the upper bound on the isometry property 
as stated in the corollary.
\end{proofof2}

\subsection{Proof of Corollary~\ref{coro::sparseK}}
\label{sec::sparseK}
\begin{proofof2}
Recall the following
\bens
\lefteqn{
\tilde{A} := X^TX - \tr(B) I_{m}
 = \big(Z_1 A^{1/2} + B^{1/2} Z_2)^T \big(Z_1 A^{1/2} +B^{1/2}
Z_2\big) - \tr(B) I_{m}}\\ 
& = & 
\big(Z_2^T B^{1/2} +A^{1/2} Z_1^T\big)\big(Z_1 A^{1/2} + B^{1/2} Z_2) - \tr(B) I_{m} \\ 
& = &  \big(Z_2^T B^{1/2} Z_1 A^{1/2} + A^{1/2} Z_1^T B^{1/2} Z_2\big)
 +A^{1/2} Z_1^T Z_1 A^{1/2} +  \big(Z_2^T B Z_2  - \tr(B) I_{m}\big).
\eens
Hence, for all vectors $u \in \Sp^{m-1} \cap E$
\bens
\nonumber
\lefteqn{
\frac{u^T (X^T X)u}{f}-\frac{u^T\tr(B)I u }{f} - u^T A u 
\le 
\inv{f}\abs{u^T  Z_2 B  Z_2^T u- \tr(B) u^T u}}\\
&& 
+ \abs{\inv{f} u^T A^{1/2} Z_1^T  Z_1 A^{1/2} u - u^T A u}
+\frac{2}{f}\abs{u^T A^{1/2} Z_1^T B^{1/2} Z_2 u}.
\eens
By Lemma~\ref{lemma::normA}, we have   on event $\B_1$,  
\bens 
\forall u \in E \cap S^{m-1}\; \; \abs{u^T Z^T B Z u - \tr( B)} 
& \le & 4 C \ve \tr(B);
\eens
By Lemma~\ref{lemma::orthogSp}, we have on event $\B_2$,
\bens
\forall u \in E \cap  S^{m-1} \; \; \abs{u^T A^{1/2} Z_1^T B^{1/2} Z_2 u} 
& \le & 4C \ve \tr(B) \twonorm{A^{1/2} u}/\twonorm{B}^{1/2}.
\eens
For all $u \in S^{m-1} \cap E$,
\bens
\lefteqn{ 8 C\ve \tau_B\twonorm{A^{1/2}u} 
/\twonorm{B}^{1/2} \le
 2(  2 C \ve^{1/2}  \frac{\tau_B}{\twonorm{B}^{1/2} })
(2\ve^{1/2} \twonorm{A^{1/2}u} ) }\\
& \le &
   4 C^2 \ve \frac{\tau_B^2}{\twonorm{B}} + 
4\ve \twonorm{A^{1/2}u}^2
\le  4 C^2 \ve \tau_B + 4\ve \twonorm{A^{1/2}u}^2.
\eens 
And finally, we have also shown that for all $u \in E$ on event $\B_9$,
\bens
&& (1-\ve) \twonorm{A^{1/2} u}
 \le \frac{1}{\sqrt{f}}\twonorm{Z_1 A^{1/2} u} \le (1+\ve) \twonorm{A^{1/2} u}.
\eens
Thus we have for all $u \in S^{m-1} \cap E$, on event $\B_1 \cap \B_2
\cap \B_9$,
\ben
\nonumber
\lefteqn{
\abs{\frac{u^T (X^T X)u}{f}-\frac{u^T\tr(B)I u }{f} - u^T A u }
\le \inv{f}\abs{u^T  Z_2^T B  Z_2 u- \tr(B) u^T u}}\\
&& 
\nonumber
+ \abs{\inv{f}\twonorm{Z_1 A^{1/2} u}^2 - \twonorm{A^{1/2} u}^2}
+\frac{2}{f}\abs{u^T A^{1/2} Z_1^T B^{1/2} Z_2 u} \\
& \le &  
\nonumber
4 C \ve \tau_B +6\ve \twonorm{A^{1/2} u}^2 +  8C \ve \tau_{B}   \twonorm{A^{1/2}u}/\twonorm{B}^{1/2} \\
& \le &
\nonumber
  4 C \ve \tau_B +6\ve \twonorm{A^{1/2} u}^2 + 4 C^2 \ve \tau_B +
  4\ve \twonorm{A^{1/2}u}^2 \\
& \le &
\label{eq::upper}
10\ve \twonorm{A^{1/2} u}^2 + 4 (C^2 + C)\ve \tau_B.
\een
\noindent {\bf Upper bound:}
Thus we have by~\eqref{eq::upper} for the maximum sparse eigenvalue of $\tilde{A}$ at order $k$:
\bens
\rho_{\max}(k, \tilde{A} )
& := & \max_{u \in E \cap S^{m-1}} \abs{u^T \tilde{A}u} 
 \le \max_{u \in E \cap S^{m-1}}
\abs{u^T \tilde{A}u  -u^T A u } + \rho_{\max}(k, A) \\
& \le & 
\rho_{\max}(k, A)(1 + 10\ve) + C_4  \ve \tau_B  
\eens
where $C_4 = 4(C+C^2)$.
The upper bound on $\rho_{\max}(k, \tilde{A} - A)$ in the theorem
statement thus holds. \\

\noindent {\bf Lower bound:}
Suppose $C_4 = 4(C+C^2) \vee 10$
\bens
\ve \le \frac{\delta}{C_4} \min\left( \frac{\rho_{\min}(k,
    A)}{\tau_B},1\right)
= \frac{\delta}{C_5} \; \text{ and } \;
C_ 4 \ve  \le \delta \min\left( \frac{\rho_{\min}(k, A)}{\tau_B},1\right).
\eens
We have by~\eqref{eq::upper} for all $u \in S^{m-1} \cap E$, on event $\B_1 \cap \B_2 \cap \B_9$,
\bens
\lefteqn{
\inv{f}u^T (X^T X)u -u^T\frac{\tr(B)I_m}{f} u}\\
& \ge &  u^T A u - \left(6 \ve u^T A u +4 C\ve \tau_B + 8 C \ve \tau_B\twonorm{A^{1/2}u} 
/\twonorm{B}^{1/2} \right)\\
& \ge &  u^T A u -6\ve u^T A u -4 C\ve \tau_B - 8 C \ve \tau_B^{1/2} \twonorm{A^{1/2}u} \\ 
& \ge &  u^T A u -10\ve u^T A u -4(C + C^2) \ve \tau_B \ge   u^T A u (1-10 \ve -\delta) \\
& \ge &  u^T A u (1 -2\delta)
\eens
where  $4(C +C^2) \ve \tau_B \le \delta \rho_{\min}(k, A)$
and $10 \ve \le \delta$.
\end{proofof2}

\section{Proofs of  Lemmas~\ref{lemma::normA} and~\ref{lemma::orthogSp}
  and Corollary~\ref{coro::tartan}}
\label{sec::proofofnormA}
Throughout the following proofs, we denote by $r(B) =
\frac{\tr(B)}{\twonorm{B}}$. Let $\ve \le \inv{C}$ where $C$ is large enough so that 
$ c c' C^2 \ge 4$, and hence the choice of $C = C_0/\sqrt{c'}$
satisfies our need.

\begin{proofof}{Lemma~\ref{lemma::normA}}
First we prove concentration bounds for all pairs of $u, v \in \Pi'$, where 
$\Pi' \subset \Sp^{m-1}$ is an $\ve$-net of $E$. 
Let $t  = C K^2 \ve \tr(B)$.
We have by Lemma~\ref{lemma::oneeventA}, and the union bound,
\bens
\lefteqn{
\prob{\exists u, v \in \Pi', \; 
\abs{u^T Z^T B Z v - \E u^T Z^T B Z v} >  t}} \\
& \leq &
2 \abs{\Pi'}^2 \exp \left[- c\min\left(\frac{t^2}{K^4  \fnorm{B}^2}, 
\frac{t}{K^2 \twonorm{B}} \right)\right] \\
& \leq &
2 \abs{\Pi'}^2
\exp \left[- c\min\left(C^2, \frac{C K^2}{\ve} \right)\frac{\ve^2 r(B)}{K^4}\right]
\le  2 \exp\left(-c_2\ve^2 r(B) /K^4\right)
\eens
where we use the fact that 
$\fnorm{B}^2 \le  \twonorm{B}\tr(B)$, and 
\bens
\abs{\Pi'} \le {m \choose k}(3/\ve)^k \le \exp(k\log(3 e m/k\ve))
\eens
while 
\bens
c \min\left(C^2, \frac{C K^2}{\ve} \right) \ve^2
\frac{r(B)}{K^4}
= c C^2 \ve^2 \frac{\tr(B)}{\twonorm{B} K^4} \ge c C_0^2 k \log\big(\frac{3e m}{k \ve}\big)
\ge 4k\log\big(\frac{3e m}{k \ve}\big)
\eens
Denote by $\B_2$ the event such that for $\Lambda := \inv{\tr(B)} (Z^T B Z - I)$, 
\bens
\sup_{u, v \in \Pi'} \abs{v^T \Lambda u} 
& \le & C  \ve =:r'_{f,k}
\eens
holds. A standard approximation argument shows that under $\B_2$ and
for $\ve \le 1/2$, 
\ben
\label{eq::sphereL}
\sup_{x, y\in \Sp^{m-1} \cap E} \abs{y^T \Lambda  x} 
\le \frac{r'_{k,f}}{(1-\ve)^2} \le 4 C  \ve.
\een
The lemma is thus proved.
\end{proofof}

\begin{proofof}{Lemma~\ref{lemma::orthogSp}}
By Lemma~\ref{lemma::oneeventA}, we have 
for $t = C \ve \tr(B)/\twonorm{B}^{1/2}$ for $C = C_0/\sqrt{c'}$
\bens
\prob{\abs{w^T Z_1^T B^{1/2} Z_2 u} >  t}
& \le &
\exp\left(-c\min\left(\frac{C^2 \frac{\tr(B)^2}{\twonorm{B}} \ve^2}
{K^4\tr(B)}, \frac{C\ve \tr(B)}{K^2\twonorm{B}}
\right)\right) \\
& \le & 
2 \exp\left(-c\min\left(\frac{C^2\ve^2 r_B}{K^4}, \frac{C\ve r_B}{K^2}\right)\right) \\
& \le & 
2 \exp\left(-c \min\left(C^2, \frac{CK^2}{\ve} \right) \ve^2 r_B/K^4\right)
\eens
Choose an $\ve$-net $\Pi' \subset S^{m-1}$ such that
\ben
\label{eq::Enet}
\Pi' = \bigcup_{\abs{J} = k} \Pi'_{J} \; \; \text{ where } \;\; 
\Pi'_{J} \subset E_J \cap S^{m-1} 
\een
is an $\ve$-net for $E_J \cap S^{m-1}$ and
\bens
\abs{\Pi'} \le {m \choose k}(3/\ve)^k \le \exp(k\log(3 e m/k\ve)).
\eens
Similarly, choose $\ve$-net $\Pi$ of $F \cap S^{m-1}$ of size at most 
$\exp(k\log(3 e m/k\ve))$.
By the union bound and Lemma~\ref{lemma::oneeventA}, and for $K^2 \ge 1$,
\bens
\lefteqn{
\prob{\exists w \in \Pi, u \in \Pi'\; s.t. \; 
\abs{w^T Z_1^T B^{1/2} Z_2 u} 
\ge C \ve{\tr(B)}/{\twonorm{B}^{1/2}}}}\\
& \le &
\abs{\Pi'} \abs{\Pi}  2 \exp\left(-c \min\left(C K^2/\ve, C^2\right) \ve^2 r_B/K^4\right)\\
& \le & 
\exp\left(2 k \log(3em/k\ve)\right) 2 \exp\left(-cC^2 \ve^2 r_B/K^4 \right) \\
& \le & 
2 \exp\left(-c_2\ve^2 r_B/K^4\right) 
\eens
where $C$ is large enough such that 
$c c' C^2 := C' > 4$ and  for $\ve \le \inv{C}$,
$$c \min\left(C K^2/\ve, C^2\right) \ve^2\frac{\tr(B)}{\twonorm{B}K^4} 
\ge C' k \log(3 e  m/k\ve)
\ge 4 k \log(3 e  m/k\ve).$$ 
Denote by $\U := Z_1^T B^{1/2} Z_2$.
A standard approximation argument shows that if 
\bens
\sup_{w \in \Pi, u \in \Pi'} \; \; \abs{w^T \U u} 
\le C \ve \frac{\tr(B)}{\twonorm{B}^{1/2}} =: r_{k,f}
\eens
an event which we denote by $\B_2$, then for all $u\in E$ and $w \in F$,
\ben
\label{eq::complete}
\abs{w^T Z_1^T B^{1/2} Z_2 u} \le \frac{r_{k,f}}{(1-\ve)^2}.
\een
The lemma thus holds for $c_2 \ge C'/2 \ge 2$.
\end{proofof}

\begin{proofof}{Corollary~\ref{coro::tartan}}
Clearly~\eqref{eq::wyFnorm} implies that \eqref{eq::ALocalkronsum}
holds for $B=I$.
Clearly~\eqref{eq::BI}  holds following the analysis of 
Lemma~\ref{lemma::normA} by setting $B = I$, while replacing event $\B_1$ with
$\B_3$, which denotes an event such that 
\bens
\sup_{u, v \in \Pi} \onef \abs{v^T (Z^T Z - I) u} 
& \le &  C \ve
\eens
The rest of the proof follows by replacing $E$ with $F$ everywhere.
The corollary thus holds.
\end{proofof}

\bibliography{subgaussian}

\end{document}